\documentclass[11pt,leqno]{amsart}
\usepackage{a4wide}
\usepackage{amssymb}
\usepackage{amsmath}
\usepackage{amsthm}
\usepackage{verbatim}
\usepackage{txfonts}
\usepackage{fullpage}
\usepackage[all,cmtip]{xy} 
\usepackage{yfonts}

\usepackage[T1]{fontenc}
\usepackage[utf8]{inputenc}
\usepackage{mathtools}   
\usepackage{amsfonts}
\usepackage{mathrsfs, todonotes}
\usepackage[mathscr]{euscript}

\usepackage{hyperref}
\numberwithin{equation}{section}

\theoremstyle{plain}

\newcommand{\yh}[1]{{
#1}}
\newcommand{\yhn}[1]{}
\renewcommand{\todo}[1]{}
\newcommand{\old}[1]{}

\newcommand{\A}{\ensuremath{{\mathbb{A}}}}

\newcommand{\Z}{\ensuremath{{\mathbb{Z}}}}

\newcommand{\Q}{\ensuremath{{\mathbb{Q}}}}
\newcommand{\R}{\ensuremath{{\mathbb{R}}}}
\newcommand{\F}{\ensuremath{{\mathbb{F}}}}

\newcommand{\D}{\ensuremath{{M}}}

\newcommand{\KLv}[1]{\text{KL}_v \lb #1 \rb }
\newcommand{\KL}{\text{KL}}
\newcommand{\el}{a}

\renewcommand{\L}{\ensuremath{{\mathbb{L}}}}
\renewcommand{\mod}{\ensuremath{{\text{\ mod\ }}}}

\renewcommand{\O}{\ensuremath{O}}
  \newcommand{\fA}{{\mathfrak{A}}}
     \newcommand{\CO}{{ {O}}} \newcommand{\CB}{{\mathcal {B}}}

\newcommand{\Vol}{\text{Vol}}
\newcommand{\GL}{\ensuremath{{\text{GL}}}}
\newcommand{\PGL}{\ensuremath{{\text{PGL}}}}
\newcommand{\Char}{\ensuremath{{\text{char}}}}

\newtheorem{theo}{Theorem}[section]
\newtheorem{lem}[theo]{Lemma}
\newtheorem{prop}[theo]{Proposition}
\newtheorem{cor}[theo]{Corollary}

\theoremstyle{remark}
\newtheorem{rem}[theo]{Remark}
\newtheorem{example}[theo]{Example}
\theoremstyle{definition}
\newtheorem{defn}[theo]{Definition}

\newcommand{\zxz}[4]{\begin{pmatrix} #1 & #2 \\ #3 & #4 \end{pmatrix}}

\renewcommand{\Im}{\operatorname{Im}}
\newcommand{\Hom}{\operatorname{Hom}}
\newcommand{\Ind}{\operatorname{Ind}}

\newcommand{\Tr}{\text{Tr}}
\newcommand{\Nm}{\text{Nm}}
\newcommand{\Supp}{\text{Supp }}

\newcommand{\lb}{\left(}
\newcommand{\rb}{\right) }
\newcommand{\lB}{\left\{}
\newcommand{\rB}{\right\}}
\newcommand{\cc}{\mathfrak{c}}
\renewcommand{\aa}{d}
\begin{document}
\bibliographystyle{plain}
\title{The Petersson/Kuznetsov trace formula with prescribed local ramifications}
\author{Yueke Hu}
\address{Yau Mathematical Sciences Center\\ Tsinghua University\\
	Beijing 100084\\
	China}

\bibliographystyle{plain}

\begin{abstract}
	In this paper we derive  refined Petersson/Kuznetsov trace formulae with prescribed local ramifications.  The spectral side of these formulae 
\yh{picks out newforms whose associated local components come from specific sub-families of representations of given level, and are much shorter compared with the classical versions.} We use them to study the first moment and the subconvexity bound of certain Rankin-Selberg $L-$function in a hybrid setting, \yh{ obtaining Weyl bound in a wider range compared to previous works}. 
\end{abstract}

\maketitle
\tableofcontents

\section{Introduction}
The Petersson and the Kuznetsov trace formulae are very close in nature, and they can be both derived from a relative trace formula as in \cite{knightly_kuznetsovs_2013,KL06}, by integrating pretrace formula against characters over unipotent subgroups, with the  difference coming only from the Archimedean component. They
have been  important tools in analytic number theory to study various types of problems like the vertical Sato-Tate conjecture, the moments of $L-$functions and their subconvexity bounds. See for example \cite{Blo19} for a survey.

	In this paper we derive  refined Petersson/Kuznetsov trace formulae with prescribed local ramifications.  
More precisely,  the spectral side of these formulae consists of newforms which are  associated to automorphic representations whose local component at a given place $p$ belongs to a small family of supercuspidal  representations or principal series representations. \yh{These formulae have shorter spectral sums or integrals, at the cost of encountering generalized Kloosterman sum and longer geometric side. They can be useful to balance the contributions from the main term and error terms of the geometric side for applications.}

 As a first application, we shall use them to study the first moment of the Rankin-Selberg $L-$function \yh{ for different families of local representations}. In the special case where we know the positivity of the $L-$functions, we further obtain hybrid subconvexity bounds, which is as strong as the Weyl bound \yh{in a much wider  range compared with the previous work in the square-free case \cite{feigon_averages_2009} and the most recent work in the depth aspect \cite{HN18}.} We expect many other applications in the near future. 
\subsection{The classical  trace formulae}

 Consider for simplicity the classical Petersson trace formula, which relates the Fourier coefficients of holomorphic modular forms to the Kloosterman sums as follows:
\begin{align}\label{Eq1.1:Pettracefull}
\frac{\Gamma \lb \kappa-1 \rb }{ \lb 4\pi \rb ^{\kappa-1}}\sum\limits_{\varphi}\frac{\lambda_{m_1} \lb \varphi \rb \overline{\lambda_{m_2} \lb \varphi \rb }}{||\varphi||^2}= N\lb\delta_{m_1=m_2} +2\pi i^{-\kappa} \sum\limits_{c\equiv 0 \mod{N}, c>0} \frac{\KL \lb m_1, m_2, c \rb }{c} J_{\kappa-1} \lb \frac{4{\pi}\sqrt{m_1 m_2}}{c} \rb\rb.	
\end{align}
Here the sum of $\varphi$ is over an orthonormal basis (with respect to \eqref{Eq2.1:AdelicL2}) of holomorphic automorphic forms of weight $\kappa$, level $N$ and trivial nebentypus. 
$\lambda_m(\varphi)$ is the $m-$th normalized Fourier coefficient.
$\KL\lb m_1,m_2,c \rb $ is the classical Kloosterman sum with conductor $c$:
\begin{equation}
\KL\lb m_1,m_2,c \rb =\sum\limits_{x\in \lb\Z/c\Z\rb^\times}e\lb \frac{m_1x+m_2\overline{x}}{c}\rb,
\end{equation}
where $\overline{x}$ is the inverse of $x$ in $\lb\Z/c\Z\rb^\times$.
$J_{\kappa-1}$ is the J-Bessel function.
The Kloosterman sum can be written as a product of local Kloosterman sums.

 The formula \eqref{Eq1.1:Pettracefull} can be obtained from the relative trace formula where the test function $f_p$ at $p|N$ is chosen to be essentially the characteristic function of a congruence subgroup. The $\delta_{m_1=m_2}$ term comes from the first-cell terms in the Bruhat decomposition, and the Kloosterman sum parts come from second-cell terms. See Section \ref{Sec4} or \cite{knightly_kuznetsovs_2013,KL06} for general settings.

\begin{rem}\label{Rem1.1:Issues}
In applications to depth aspect problems, there are however two issues  with  \eqref{Eq1.1:Pettracefull}:
\begin{enumerate}
\item \eqref{Eq1.1:Pettracefull} picks out newforms as well as old forms on the spectral side. So it is not convenient to use when aiming only for newforms. Contributions from old forms have to be subtracted, which usually make computations more complicated, and  also trickier for depth-aspect problems. For this reason, many results using the classical formula deal with square-free or even prime levels only.
\yh{In some recent works \cite{BBDDM, Petrow}, this issue has been addressed at the cost of additional complications for the sum of Kloosterman sums on the geometric side. See for example \cite[Theorem 3.1]{Petrow}.}

	\item There is an asymmetry between the Archimedean aspect and the level aspect. More precisely, in the Archimedean aspect, the analytic conductor of $\varphi$ is roughly $k^2$, whereas the length of the sum in $\varphi$ is roughly $k$. On the other hand in the level aspect, the finite conductor of $\varphi$ is $N$, whereas the length of the spectral sum is also roughly $N$. Thus the spectral sum is much longer in the level aspect in terms of the relation with the conductor.
	
	
\end{enumerate}
\end{rem}

\subsection{Main results}
For simplicity, we shall be interested in automorphic representation $\pi$ over $\Q$ with trivial central character and level $N=C(\pi)=p^{\cc}$ for some integer $\cc\rightarrow \infty$ and $p\neq 2$. All the results can be directly extended to composite levels \yh{and general number fields due to the nature of the method}. 

The local component of $\pi$ at $p$ will then be either a supercuspidal representation or a principal series representation, associated to a character $\theta$ over some \'{e}tale quadratic algebra $\L/\Q_p$ by compact induction  or parabolic induction as in Section \ref{Sec3}. \yh{(Being an \'{e}tale quadratic algebra means that $\L$ is either a quadratic field extension over $\Q_p$, or $\L=\Q_p\times \Q_p$.)}
\subsubsection{The refined Petersson trace formula}
For fixed even weight $\kappa\geq 4$, let $\mathcal{F}_\theta[l]$ be the subset of holomorphic newforms of weight $\kappa$, level $N=p^\cc$ with $\cc\geq 3$, and trivial nebentypus, whose associated local representation $\pi_p$ belongs to a `neighboring' family  $\pi_\theta[l]$. 
\yh{This means that $\pi_p$ is associated to some character from $ \theta[l]=\left\{\theta', \theta'|_{\Q_p^\times}=\theta|_{\Q_p^\times}\text{\ and }\cc \lb \theta^{-1}\theta' \rb \leq e_\L l\ \right\}$,  where $e_\L$ is the ramification index of $\L$.}  See
Definition \ref{defn:FamilythetaN}, \ref{Def3.1:pithetan}. \yh{The relative size of $\theta[l]$ is discussed in Lemma \ref{Lem:Indexoffamily} and the following remark. These families arise naturally when we pick the test function to be truncated matrix coefficients in the relative trace formula.}

\yh{We first take $$
l=l_0=\begin{cases}
1, &\text{\ if $\L/\F$ is an inert quadratic field extension, \old{and $i_0$ is odd;}}\\
0, &\text{\ otherwise}.
\end{cases}
$$
as in \eqref{Eq4.5:l0}. 
Denote
\begin{equation}\label{Eq:introi0}
i_0=\cc \lb \theta \rb /e_\L
\end{equation}
as in Definition \ref{defn:FamilythetaN}, and
$$c_0=\begin{cases}
p^{i_0+1}, &\text{if $\pi_\theta$ is supercuspidal,}\\
p^{i_0}, &\text{otherwise}
\end{cases}  $$
as in Definition \ref{Def4.5:c0}, which is roughly $p^{\cc/2}=\sqrt{N}$.
}

Using the relative trace formula approach as in Section \ref{Sec4.1:testfun}, we show the following main result:

\begin{theo}[Theorem \ref{Theo4.5:smallfamilyPTF}]\label{Theo1.2:PTF}
\yh{For fixed even weight $\kappa \geq 4$ and the family of newforms $\mathcal{F}_\theta[l_0]$ as above,}   we have
\begin{align}
\sum\limits_{\varphi\in \mathcal{F}_\theta[l_0]}\frac{1}{||\varphi||^2}\lambda_{m_1} \lb \varphi \rb \overline{\lambda}_{m_2} \lb \varphi \rb 
= 		C_\mathcal{F}[l_0]\frac{ \lb 4\pi \rb ^{\kappa-1}}{ \lb \kappa-2 \rb !}		\lb\delta_{m_1=m_2}+2\pi i^\kappa\sum\limits_{\yh{c\equiv 0 \mod{c_0}, c>0}}\frac{G \lb m_1,m_2,\theta,c^{-2} \rb }{c}J_{\kappa-1} \lb \frac{4\pi\sqrt{m_1m_2}}{c} \rb \rb.	\notag
\end{align}
\end{theo}

Here $C_\mathcal{F}[l_0]\asymp (1+p^{-1})c_0$ is given before \eqref{Eq4.5:CF}, \yh{and it replaces the role of $N$ in the classical formula \eqref{Eq1.1:Pettracefull}.}

 $G\lb m_1,m_2,\theta,c^{-2}\rb$ is the generalized Kloosterman sum which is a product of local factors as in Definition \ref{Def4.5:GlobalGKloosterman}, where the local factors at $v\neq p$ are the same as the standard Kloosterman, while the local factor $G_p\lb m_1,m_2,\theta,c^{-2}\rb$ given in  Definition \ref{Def4.4:KloostermanSC}/Definition \ref{Def4.4:GpPS} involves the character $\theta$ and an integration inside $\L^\times$.
\begin{rem}\label{Rem1.2:Squarecancel}
The square-root-cancellation-type upper bounds for $G(m_1,m_2,\theta,c^{-2})$ are proven in Lemma \ref{Lem4.4:GeneralKLWholerange}, \ref{Lem4.4:StationaryPS}. The implied constants can depend on some fixed powers of $p$. But it should be possible to remove this dependence by a more careful study of character sums over residue fields.

We will also explain in Remark \ref{Rem:Gpbecomeskloosterman}, \ref{Rem4.4.2:isstandard} that $G_p\lb m_1,m_2,\theta,c^{-2}\rb$ recovers the standard Kloosterman sum when $v_p(c)\geq \cc$.
\end{rem}

\yh{
\begin{rem}
We do not take $l_0=0$ in the case where $\L$ is inert, for the following technical reasons:
\begin{enumerate}
\item[(1)]The computations for the geometric side, while possible, would be slightly more complicated, and the resulting formula for the generalized Kloosterman sum is not as uniform;
\item[(2)]For applications, considering the family $\mathcal{F}_\theta[1]$ instead of $\mathcal{F}_\theta[0]$ would affect asymptotic bounds by a fixed power of $p$, which is negligible for depth aspect problem considered in this paper.
\end{enumerate}
Of course pushing for the case $l=0$ is necessary for horizontal aspect problems. This will be addressed in future works.

We also avoid the case $\kappa=2$ for our main results, due to the reason that the conjugate of the test function at infinity $\overline{f_\infty}$ is chosen to be  the matrix coefficient for the lowest weight element, which is not $L^1$ when $\kappa=2$. For applications it is possible to circumvent this issue by taking $\overline{f_\infty}$ to be a smooth cut-off for the matrix coefficient, which would somewhat complicate both the spectral side and the geometric side of the formula. As we are mainly interested in the depth aspect in this paper, we just assume $\kappa\geq 4$ for simplicity.
\end{rem}
\begin{rem}
There is an independent work \cite{PWZ} which derives a similar formula  essentially in the special case $N=p^3$ by using the known formula for the matrix coefficient of newforms in this case. The  matrix coefficients for newforms in general are however more complicated to utilize directly, and one technical novelty of this paper is to understand the matrix coefficients by exploiting the relations between newforms and minimal vectors/microlocal lifts in Section \ref{Sec3}. 

We also note that Definition \ref{Def4.4:KloostermanSC} only involves the value of $\theta$ on $ZU_\L(1)$. In the special case where $\cc(\pi_p)=3$ and $\cc(\theta)=2$ (see \eqref{Eq:cpictheta}), the behavior of $\theta$ on $ZU_\L(1)$ can be reduced to an additive character in the sense of  \eqref{Eq:multitoadditive}, which may explain why
the generalized Kloosterman sum does not seem to appear in their work.
\end{rem}
}

\begin{rem}
The main advantage of Theorem \ref{Theo1.2:PTF} is that it addresses both issues mentioned in Remark \ref{Rem1.1:Issues}: it picks out only newforms; the length of the spectral side and the first-cell term have size $C_\mathcal{F}[l_0]\asymp N^{1/2}$ compared to $N$ in \eqref{Eq1.1:Pettracefull}. There are also two trade-offs:
\begin{enumerate}
\item[(1)]The generalized Kloosterman sums are more complicated than the standard Kloosterman sum to analyze;
\item[(2)]The length of the sum of  Kloosterman sums  is longer, in the sense that in Theorem \ref{Theo1.2:PTF} $v_p(c)\geq v_p(c_0)$  which is roughly $\frac{\cc}{2}$, while in \eqref{Eq1.1:Pettracefull} $v_p(c)\geq \cc$.
\end{enumerate}
 We shall develop tools and tricks to mitigate these disadvantages. For example, we already discussed the square-root cancellation for the generalized Kloosterman sum in Remark \ref{Rem1.2:Squarecancel}; In Theorem \ref{Theo1.2:PTFaverage} we shall develop a formula picking out a larger family with shorter  sum of Kloosterman sums, helping us to reach a balance between the first-cell term and the second-cell terms; In Section \ref{Sec1.3.2} we shall discuss alternative perspective for the generalized Kloosterman sum, and how to deal with the character sum after applying the Voronoi summation formula, which is commonly used after the Petersson/Kuznetsov trace formula in dealing with many analytic number theory problems.
\end{rem}
\subsubsection{Trace formulae for larger families}

Let $l$ be an integer such that $l_0\leq l<i_0$, where $i_0$ is as in \eqref{Eq:introi0} and is roughly $\frac{\cc}{2}$.
Let $c_l=c_0p^{l-l_0}$, and $\mathcal{F}_\theta[l]$ be as above. $\mathcal{F}_\theta[l]$ is a larger family of newforms compared to $\mathcal{F}_\theta[l_0]$. Then we have a relative trace formula for $\mathcal{F}_{\theta}[l]$ as below.
\begin{theo}[Theorem \ref{Theo4:spectralaverg}]\label{Theo1.2:PTFaverage}
\yh{For fixed even weight $\kappa \geq 4$ and the family of newforms $\mathcal{F}_\theta[l]$, we have}
\begin{align*}
\sum\limits_{\varphi\in \mathcal{F}_{\theta}[l]}\frac{1}{||\varphi||^2}\lambda_{m_1} \lb \varphi \rb \overline{\lambda}_{m_2} \lb \varphi \rb 
= 		C_\mathcal{F}[l]\frac{ \lb 4\pi \rb ^{\kappa-1}}{ \lb \kappa-2 \rb !}	\lb\delta_{m_1=m_2}+2\pi i^\kappa\sum\limits_{c\equiv 0 \mod{c_l}, c>0}\frac{G \lb m_1,m_2,\theta,c^{-2} \rb }{c}J_{\kappa-1} \lb \frac{4\pi\sqrt{m_1m_2}}{c} \rb \rb	
\end{align*}
Here $C_\mathcal{F}[l]\asymp  C_\mathcal{F}[l_0] p^{l-l_0}$ is as in Definition \ref{Eq4.6:CFl}.
\end{theo}
\begin{rem}
One can obtain Theorem \ref{Theo1.2:PTFaverage} from Theorem \ref{Theo1.2:PTF} by taking a sum. The nontrivial part is however to show that the length of the sum of  Kloosterman sums  becomes shorter, which comes from a local cancellation. Theorem \ref{Theo1.2:PTFaverage} displays a nice transition from Theorem \ref{Theo1.2:PTF} to the classical formula \eqref{Eq1.1:Pettracefull}.
\end{rem}
\subsubsection{Refined Kuznetsov trace formula}
One can similarly consider the case of Maass forms.
In this case, let $\mathcal{F}^0_\theta[l]$  be  the set of cuspidal Maass newforms of level $N=p^\cc$, trivial nebentypus, whose associated local component $\pi_p$ is related to characters from $\theta[l]$. 

The residue spectrum will not be picked out by our choice of  test function.
\yh{The Eisenstein series are constructed from parabolic inductions, }thus the
contribution from the continuous spectrum will be nontrivial only when the targeted family $\pi_\theta[l]$ consists of principal series representations. For this reason,
define \begin{equation*}
\epsilon_\L=\begin{cases}
1, &\text{\ if }\L\simeq \Q_p\times\Q_p;\\
0, &\text{\ otherwise.}
\end{cases}
\end{equation*}
When $\epsilon_\L=1$,   for each finite order Hecke character $\chi$ and  $\theta'=(\chi,\chi^{-1})$, such that $\theta'_v$ is unramified when $v\neq p$, and $\theta'_p\in \theta[l]$, define as usual  $\varphi_s\in \pi(\chi|\cdot|^s,\chi^{-1}|\cdot|^{-s})$ to be a flat section associated to an $L^2-$normalized newform, and define $$E_{\theta',s}(g)=\sum\limits_{\gamma\in B(\Q)\backslash \GL_2(\Q)}\varphi_s(\gamma g).$$

Then using the general setup from \cite{knightly_kuznetsovs_2013}, together with the test function at $p$ and the relevant computations for the refined Petersson trace formula above, one can get the following:

\begin{theo}[Kuznetsov for prescribed local component]\label{Theo:Kuz1}
\yh{Let $h$ be an even function such that $h(t)$ is holomorphic and controlled by $(1+|t|)^{-B}$ in the region $|\Im (t)|<A$ for some  sufficiently large constants $A$ and $B$ as in \cite[(8.1)]{knightly_kuznetsovs_2013}. }
For $l_0\leq l< i_0$, 
 we have
	\begin{align*}
&\sum\limits_{\varphi\in \mathcal{F}^0_{\theta}[l]}\frac{\lambda_{m_1}(\varphi )\overline{\lambda_{m_2}(\varphi )}}{||\varphi ||^2}\frac{h(t_{\varphi })}{\cosh({\pi} t_{\varphi })}+
\frac{\epsilon_\L}{\pi}\sum\limits_{\theta', \theta'_p\in\theta[l]}\, \int\limits_{-\infty}^\infty 
\lambda_{m_1}\left( E_{\theta',1/2+it}\right) \overline{\lambda_{m_2}\left( E_{\theta',1/2+it}\right)}h(t)dt
\\
=&C_\mathcal{F}[l]\left[\frac{\delta(m_1=m_2)}{\pi^2}\int\limits_{-\infty}^{\infty}h(t)\tanh({\pi} t) tdt+\frac{2{i}}{\pi } \sum\limits_{c\equiv 0 \mod{c_l}, c>0} \frac{G\lb m_1, m_2, \theta,c^{-2}\rb}{c}		\int\limits_{-\infty}^{\infty} J_{2it}\lb\frac{4{\pi}\sqrt{m_1 m_2}}{c}\rb\frac{h(t)t}{\cosh({\pi}t)}dt	\right].	\notag 
\end{align*}
Here $t_\varphi$ is the spectral parameter of $\varphi$ such that $\Delta\varphi=(1/4+t_\varphi^2)\varphi$ for the Laplace operator $\Delta$.
\end{theo}
\begin{rem}
Note that it is possible to compute $\lambda_{m}\lb E_{\theta',s}\rb$ more explicitly in terms of the twisted divisor functions and the $L-$functions for Hecke characters. We skip the details here.
\end{rem}

\subsubsection{Application to the first moment and the hybrid subconvexity bound for the Rankin--Selberg $L-$function}
We expect  several  possible applications for the above theorems.
One of them is to the vertical Sato-Tate law. Using Theorem \ref{Theo1.2:PTF}, \ref{Theo1.2:PTFaverage} or \ref{Theo:Kuz1}, the bound for the generalized Kloosterman sum discussed in Remark \ref{Rem1.2:Squarecancel}, and the recipe in \cite{BBR-14}, one should be able to get some variants of the vertical Sato-Tate law for  small families of newforms in the depth aspect.
We also expect applications in future works to the large sieve inequality, cubic moment of $L-$functions, etc,. 

In this paper we focus on the first moment of the Rankin--Selberg $L-$functions.
This problem is relatively easier to study, and our tools can already give interesting new result.
 For simplicity, we restrict ourselves to the holomorphic newforms whose levels are prime powers.
\begin{theo}\label{Theo:subconv}
Let  $\mathcal{F}_\theta[l]$ be the set of holomorphic newforms of weight $\kappa\geq 4$, level $N=p^\cc$ and trivial nebentypus, whose associated local component $\pi_p$ belongs to a small family $\pi_\theta[l]$ as above.
 Let $g$ be a  holomorphic cuspidal newform with square-free level $M$ which is coprime to $N$,  fixed weight $\kappa_g\geq 4$, and trivial nebentypus.  Then we have
\begin{equation*}
	\sum\limits_{f\in \mathcal{F}_\theta[l]}\frac{L(f\times g, 1/2)}{||f||^2}\ll_{p,\epsilon} (MN)^\epsilon\lb N^{1/2}p^l+N^{1/4}M^{1/2}p^{-l/2}\rb.
\end{equation*}

Furthermore suppose that $L(f\times g,1/2)\geq 0$ for all $f\in \mathcal{F}_\theta[l]$. 
Suppose that $N=M^\delta$ for $0<\delta<\infty$, so that the finite conductor $C \lb f\times g \rb =M^{2+2\delta}$ for any $f\in \mathcal{F}_\theta[l]$.
By picking $l$ to be the closest integer to $\log_p\lb M^{1/3}N^{-1/6} \rb$ while $1\leq l<i_0$, we get that
	\begin{equation*}
L \lb f\times g,1/2 \rb \ll_{\epsilon, p}  M^{\max\{\frac{1}{2}, \frac{  1+\delta  }{3}, \frac{\delta}{2}\}+\epsilon}.
	\end{equation*}
	In particular we obtain a hybrid subconvexity bound for $\delta$ in any compact subset of $(0,\infty)$, which is furthermore a Weyl bound in the range $1/2\leq \delta\leq 2$.
\end{theo}
\begin{rem}
The condition that $L(f\times g,1/2)\geq 0$ for all $f$ can be guaranteed when, for example, $g$ is dihedral. See the discussion in \cite[Section 1.1]{HT}.
\end{rem}
\begin{rem}
Note that from the proof in Section \ref{Sec6}, the $N^{1/2}p^l$ part comes from the accurate computations for the first-cell terms and $N^{1/4}M^{1/2}p^{-l/2}$ part comes from bounds for the second-cell terms.
Thus it is actually possible to obtain an asymptotic formula when \yh{$\epsilon(f\times g,1/2)=1$ and}  $N^{1/2}p^{3l}$ is sufficiently larger compared to $M$.
\end{rem}
\begin{rem}
Compared to \cite{feigon_averages_2009, HT}, this result has three differences/improvements. 
\begin{enumerate}
\item They assume $N$ to be square-free.
\item  They obtain a
subconvexity bound for any compact subset of $\delta\in (0,1)$, and
a Weyl-type bound only at $\delta=1/2$.
\end{enumerate}

We make a more detailed comparison of the method in this paper with the one used in \cite{HN18} (which extends \cite{feigon_averages_2009} in some sense). The current method  has the following advantages:
\begin{enumerate}
\item[(I)]It made use of the flexibility of Theorem \ref{Theo1.2:PTFaverage}, and the resulting subconvexity bound in Theorem \ref{Theo:subconv} is stronger than both \cite[Theorem 1.8]{HN18} (which obtains Weyl-type strength at $\delta=2$) and the analogue of \cite[Corollary 1]{HT}, allowing Weyl-type subconvexity bound in a wide hybrid range.
\item[(II)] It covers  the principal series case and the supercuspidal case somewhat uniformly.
\item[(III)] It  does not require any $\epsilon-$value condition for the Archimedean components.
\item[(IV)] The refined Petersson/Kuznetsov trace formulae should be applicable to many other problems.
\end{enumerate}
The method in \cite{HN18} involves using the relative trace formula associated to Waldspurger's period integral on some quaternion algebra. This quaternion algebra is assumed to be a division algebra at all Archimedean places (which translates into $\epsilon-$value conditions). The method there has the following advantages:
\begin{enumerate}
\item[(i)] It does not require $M$ to be square-free.
\item[(ii)] It is also used to prove a hybrid subconvexity bound \cite[Theorem 1.10]{HN18} in the joint ramification case.
\item[(iii)] It  works for general number fields, and does not rely on the Ramanujan conjecture.
\end{enumerate}
We do believe that some of the differences are amenable with extra work. For example, (I)-(III) may also be achieved by the method of \cite{HN18}. On the other hand, (i) (ii) may also be recovered using the method in this paper, by employing a more flexible version of the Voronoi summation formula.
\end{rem}

\subsection{Basic strategies}
\subsubsection{Deriving the refined Petersson/Kuznetsov trace formula}
The classical formula \eqref{Eq1.1:Pettracefull} can be obtained by setting the local test function for the relative trace formula to be the characteristic function of the related congruence subgroup for the newform, as is done in \cite{knightly_kuznetsovs_2013,KL06}.
The first idea to derive Theorem \ref{Theo1.2:PTF} is relatively straightforward, that is, to use instead suitable cut-off of the local matrix coefficient for the newform as the test function. 

The matrix coefficient itself however is not very convenient to directly make use of. So far we have some  understandings about its support, level (from \cite[Proposition 2.12]{Hu:17a}) and size (from \cite[Theorem 5.4]{HuSa:19}). 

Our approach in this paper is to make use of the special test vectors, i.e., the minimal vectors for the supercuspidal representations discussed in \cite{HN18,HuNelsonSaha:17a}  and the microlocal lifts for the principal series representations discussed in \cite{nelson_microlocal_2016}. These test vectors have the property that a large compact open subgroup acts on them by a character $\tilde{\theta}$, which can uniquely identify the test vector, while the local representation comes only from $\pi_\theta[l_0]$ (See Proposition  \ref{Prop3.2:B1givespi}/Corollary \ref{Cor:microlocallift} for more details). Using the relation between these special test vectors and the newforms in Corollary \ref{Cor:RelationNewMinimal}/Lemma \ref{Lem3.3:newformasMLL}, we construct test functions in Definition \ref{Def3.2:testfSC}, \ref{Def3.3:testPS} from a linear combination of translates of $\tilde{\theta}$, which exactly pick out the newforms from $\pi_\theta[l_0]$. See Proposition \ref{Prop3.2:testfunaction}, \ref{Prop3.3:testfunaction}.
We believe this idea of constructing test functions will also be useful on higher rank groups.

The second-cell terms from the relative trace formula for the constructed local test function can be reduced to the computations for $\tilde{\theta}$ by a change of variables, giving rise to the generalized Kloosterman sums in  Lemma \ref{Lem4.4:GeneralKLWholerange}/Definition \ref{Def4.4:GpPS}.
The explicit shape of these character sums allows us to prove the square-root cancellation (up to a bounded power of $p$), and also detect cancellations when taking sums in Theorem \ref{Theo1.2:PTFaverage}.

\subsubsection{Alternative description and the character sum after the Voronoi summation formula}\label{Sec1.3.2}
In Lemma \ref{Lem5.1:twoApproach}, we show that the local test function we have constructed and used actually coincides with the matrix coefficient of the newform in the range we are interested in. 
This alternative perspective also turns out to be quite useful. To explain this, we remark that in applications the Petersson/Kuznetsov trace formula is often followed by the use of the Voronoi summation formula. In the classical setting, the Kloosterman sum gives rise to the Ramanujan sum after the Voronoi summation:
\begin{equation*}
\widetilde{\KL}(m_1,m_2,\el,c)=\sum\limits_{x\in \lb \Z/c\Z\rb^\times}e\lb\frac{m_1+m_2\el}{c}x\rb.
\end{equation*}
Here $a$ is an additional parameter, which can be $-1$ for example. The Ramanujan sum has the property that its average size is roughly $1$ when, for example, taking a sum in $m_1$. 

On the other hand for the generalized Kloosterman sum $G(m_1,m_2,\theta,\mu)$, the corresponding character sum, which occurs in the proof of Lemma \ref{Lem6.2:KcZ} and is denoted by $\tilde{G}(m_1,m_2,\el,\theta,\mu)$ in Definition \ref{Defn5.3:Dualsum}, becomes more complicated to analyze. Take $\mu=\frac{1}{c^2}$ and $k=v_p(c)$.
Recall from Remark \ref{Rem1.2:Squarecancel} that when $k\geq \cc$, $G(m_1,m_2,\theta,\mu)$ becomes the classical Kloosterman sum, so $\tilde{G}(m_1,m_2,\el,\theta,\mu)$ becomes the Ramanujan sum. We focus on the case $ v_p(c_0)\leq k< \cc$ now.
Using the alternative description above, we can identify $\tilde{G}(m_1,m_2,\el,\theta,\mu)$ in the range of interest with the value of the matrix coefficient itself as in the proof of Lemma \ref{Lem5.3:dualsum}. Then we apply the known results in \cite[Theorem 5.4]{HuSa:19} on the support and the size of the matrix coefficient for the newform to obtain Lemma \ref{Lem5.3:dualsum}, which says that
$\widetilde{G}_p \lb m_1,m_2,\el,\theta,\mu \rb=0$ unless \begin{equation}\label{Eq1.4:Vcondition}
v\lb m_2\mu+\frac{\el m_1}{p^{2k}}\rb\geq -\cc
\end{equation}
for our particular setting, in which case we have
\begin{equation}\label{Eq1.3:GRSupper}
\widetilde{G}_p \lb m_1,m_2,\el,\theta,\mu \rb\ll_p p^{\frac{3k-\cc}{2}}.
\end{equation}
Note that when $k=v_p(c_0)$, the congruence condition \eqref{Eq1.4:Vcondition} is (almost) automatic, and the upper bound in \eqref{Eq1.3:GRSupper} shows square-root cancellation.
Thus $\widetilde{G}_p \lb m_1,m_2,\el,\theta,\mu \rb$ displays a transition from the Ramanujan-sum-type behavior to the square-root-cancellation behavior when $k$ goes from $\cc$ to roughly $\frac{\cc}{2}$.

\subsubsection{Studying moments and hybrid subconvexity bounds}
The strategy to use the approximate functional equation, the Petersson/Kuznetsov trace formula, and then the Voronoi summation formula etc., is relatively standard. We have taken some arguments and results directly from, for example,  \cite{HT,KMV}. The main new ingredients are the refined Petersson trace formula in Theorem \ref{Theo1.2:PTFaverage} with a flexible parameter $l$, and the study of the character sum $\widetilde{G}_p \lb m_1,m_2,\el,\theta,\mu \rb$. 

By choosing $l$ properly, we can to some extent balance the contributions from the first-cell terms and the second-cell terms, obtaining Weyl-type subconvexity bound in a relatively large hybrid range.
\subsection{The Structure of the paper}
In Section \ref{Sec2} we introduce some basic notations and results.

In Section \ref{Sec3} we review some basic properties for the minimal vectors and the microlocal lifts, discuss their relations with the newforms, and construct test functions which pick out small families of newforms.

In Section \ref{Sec4} we use the relative trace formula for period integrals on unipotent subgroups to derive Theorem \ref{Theo1.2:PTF}, \ref{Theo1.2:PTFaverage} and \ref{Theo:Kuz1}.

In Section \ref{Sec5} we relate the test functions constructed in Section \ref{Sec3} with the matrix coefficient for the newform. Then we prove Lemma \ref{Lem5.3:dualsum} for the character sum  $\widetilde{G}_p \lb m_1,m_2,\el,\theta,\mu \rb$. 

In Section \ref{Sec6} we review a special version of the Voronoi summation formula, and apply the techniques developed so far to prove Theorem \ref{Theo:subconv}.

\subsection{Acknowledgement}
The author would like to thank Ian Petrow for helpful discussions and anonymous referees for suggestions to improve the paper.

\section{Preliminaries}\label{Sec2}
\subsection{Notations}\label{Sec2.1}
\yh{For any $x\in \R$, let $\lceil x\rceil$ be the least integer greater than or equal to $x$, and $\lfloor x \rfloor$ be the greatest integer less than or equal to $x$.}

Globally we shall work with the rational field $\Q$. Many of the discussions also hold for general number fields.

Let $\A$ be the ring of adeles over $\Q$, and $\A_{fin}$ be the finite adeles. We fix an additive character $\psi$ on $\Q\backslash \A$, which is a product of local additive characters $\psi_v$, where $\psi_\infty(x)=e^{-2\pi i x}$, and $\psi_p(x)=e^{2\pi i x'}$ where $x'\in \Q$ and $x'\equiv x\mod \Z_p$.

Let $\F$ denote a $p-$adic local field,  $O_\F$ be its ring of integers and $\varpi_\F$ be a uniformizer with order of residue field $p\neq 2$.  In general the level of an additive character $\psi_v$ is defined to be the smallest integer $\cc(\psi_v)$ such that $\psi_v$ is trivial on $\varpi_\F^{\cc(\psi_v)}O_\F$. Let $U_\F(n)=1+\varpi_\F^n O_\F$ when $n\geq 1$, and $U_\F(0)=O_\F^\times$. The level of a multiplicative character $\chi$ is defined to be the smallest non-negative integer $\cc(\chi)$ such that $\chi$ is trivial on $U_\F(\cc(\chi))$.

Let $\L$ be a  quadratic \'{e}tale algebra over $\F$. When $\L$ is a field, let $e_\L$ be the ramification index of $\L$. Let $O_\L$, $\varpi_\L$ and $U_\L(n)$ be defined similarly as for $\F$.

If $\L= \F\times \F$ splits, let $e_\L=1$. Let $U_\L(0)=O_\L^\times= O_\F^\times \times O_\F^\times$, and $U_\L(n)=1+\varpi_\F^n (O_\F\times O_\F)$.

Let $\theta$ be a character of $\L^\times$ with $\theta|_{\F^\times}=1$. Let $\cc(\theta)$ be the level of $\theta$. If $\L$ splits, then we can write $\theta= \lb \chi,\chi^{-1} \rb $, and define $\cc(\theta)=\cc(\chi)$. 

For $\GL_2$, let $Z$ be its center, $N$ be the unipotent subgroup. Over $\F$, let $K$ be the standard maximal compact open subgroup $\GL_2(O_\F)$.
We also denote $G=\PGL_2$.
 We denote
$$n(x)=\zxz{1}{x}{}{1}, \, a(y)=\zxz{y}{}{}{1}.$$

Let $\pi$ be an irreducible cuspidal automorphic representation of $\GL_2$ with trivial central character. Let $\pi_v$ denote its local component at $v$. Let $\cc(\pi_v)$ be the level of $\pi_v$, which is the smallest integer such that $\pi_v$ contains an element invariant by $$K_0\lb p^{\cc(\pi_v)}\rb=\left\{ g=\zxz{a}{b}{c}{d}\in \GL_2(O_v), c\equiv 0\mod\varpi_v^{\cc(\pi_v)}\right\}.$$

Haar measures are normalized so that
$\Vol \lb \Q\backslash \A \rb =1$, $\Vol(K)=\Vol(Z\backslash ZK)=1$.

For an automorphic cuspidal form $\varphi$, define 
\begin{equation}\label{Eq2.1:AdelicL2}
||\varphi||^2=<\varphi,\varphi>=\int\limits_{Z(\A)\GL_2(\Q)\backslash \GL_2(\A)} |\varphi(g)|^2dg.
\end{equation}
\subsection{A basic result on characters}
\begin{lem}\label{Lem:LiealgChar}
	Suppose that either $p$ is large enough and $i=1$, or $i$ is large enough. Then the $p-$adic logarithm $\log$ is a group isomorphism from $U_\F(i)$ with multiplication  \yh{ to $\varpi_\F^iO_\F$ with addition,
defined by the standard Taylor expansion
	\begin{equation*}
	\log \lb 1+u \rb =u-\frac{u^2}{2}+\frac{u^3}{3}+\cdots.
	\end{equation*}
} For any character $\chi$ with $\cc(\chi)>1$ there exists a unique $\alpha_\chi\in   \varpi^{-\cc \lb \chi \rb +\cc \lb \psi_\F \rb }O_\F^\times$ modulo $ U_\F(c(\chi)-i)$ multiplicatively, such that
	\begin{equation*}
	\chi \lb 1+u \rb =\psi_\F \lb \alpha_\chi \log \lb 1+u \rb  \rb , \forall u\in \varpi_\F^i O_\F.
	\end{equation*}
	On the other hand if $p\neq 2$ and $i\geq \cc(\chi)/2$, we have
\begin{equation}\label{Eq:multitoadditive}
\chi(1+u)=\psi_\F(\alpha_\chi u).
\end{equation}

\end{lem}
Note that we formulate this lemma for general $\cc(\psi_\F)$ because we will also apply it to characters of $\L^\times$ later on.
\subsection{Kirillov model, Whittaker model and unitary pairings}
This subsection is purely local so we skip the subscript $v$ from some of the notations.

For a fixed additive character $\psi$, the Kirillov model of $\pi$ is a unique realization of $\pi$ on a subspace of  $C^\infty (\F^\times)\cap S \lb \F  \rb $ such that
\begin{equation}\label{Kirilmodel}
\pi \lb \zxz{a_1}{m}{0}{a_2} \rb \varphi \lb x \rb =w_{\pi} \lb a_2 \rb \psi \lb ma_2^{-1}x \rb \varphi \lb a_1a_2^{-1}x \rb ,
\end{equation}
where $w_{\pi}$ is the central character for $\pi$. 
Let $W_\varphi$ be the Whittaker function associated to $\varphi$. Then it is related to the Kirillov model by
$$\varphi \lb \alpha \rb =W_\varphi \lb \zxz{\alpha}{0}{0}{1} \rb ,$$
$$W_\varphi \lb g \rb =\pi \lb g \rb \varphi \lb 1 \rb .$$

When $\pi$ is unitary, one can define the $G-$invariant unitary pairing on the Kirillov model by
\begin{equation}\label{Eq2.3:unitarypair}
<\varphi_1,\varphi_2>=\int\limits_{\F ^\times}\varphi_1 \lb x \rb \overline{\varphi_2} \lb x \rb d^\times x.
\end{equation}

On the other hand, if $\pi=\pi(\chi_1,\chi_2)$ is a principal series representation with $\chi_i$ unitary, the unitary pairing can be alternatively defined by
\begin{equation}\label{Eq2.3:UnitaryPricipal}
<f_1,f_2>=\int\limits_{K}f_1(k)\overline{f}_2(k)dk.
\end{equation}
Here $f_i\in \pi$ are elements in the parabolic induction model, and $K$ is a fixed maximal compact open subgroup.

\yh{Note that when we assume $p\neq 2$, $w_\pi$ trivial, and $\cc(\pi)\geq 3$, we do not need to consider complementary series representations or special representations. Though it won't be hard to include these cases under weaker assumptions.}

\subsection{Global Whittaker function}\label{Sec2.4}
Let $W_\varphi$ be now the global Whittaker function associated to a holomorphic newform $\varphi$ and the fixed additive character $\psi$. It can be computed as
\begin{equation*}
W_\varphi(g)= \int\limits_{t\in N \lb \Q \rb \backslash N \lb \A \rb }\varphi(n(t)g)\psi(-t)dt.
\end{equation*}
$W_\varphi$ factorizes into a product of local Whittaker functions
\begin{equation}
W_\varphi(g)=\prod\limits_{v} W_v(g).
\end{equation}
Here $W_\infty$ is the Whittaker function associated to the lowest (non-negative) weight element in $\pi_\infty$. For example for  a discrete series representation of weight $\kappa$ over $\R$, we have the following explicit expression
\begin{equation*}
W_\infty\lb \zxz{y}{x}{0}{1}\rb=\begin{cases}
y^{\kappa/2}e^{-2\pi y}e^{2\pi ix}, &\text{if $y>0$,}\\
0, &\text{otherwise}.
\end{cases} 
\end{equation*}
On the other hand,  $W_v$ is the Whittaker function associated to the local newform at a finite place $v$ with $W_v(1)=1$. They are closely related to the classical Fourier coefficients. More  explicitly for a positive integer $m$,
\begin{equation}
\prod\limits_{v\text{\, finite}} W_v\lb\zxz{m}{}{}{1}\rb= |m|^{-1/2}\lambda_m(\varphi).
\end{equation}
Here $\lambda_m(\varphi)$ is normalized so that $\lambda_1(\varphi)=1$, and $\lambda_m(\varphi)\ll m^\epsilon$ by the Ramanujan Conjecture.
\subsection{Hecke algebra action}
We shall choose a test function $f=f_\infty\times f_{fin}$ on $G(\A)$ (which can be view as a function on $\GL_2(\A)$  invariant by $Z(\A)$), where $f_{fin}$ is smooth on  $ G \lb \A_{fin} \rb $ and compactly supported mod center, and $f_\infty\in C \lb G \lb \R \rb  \rb $  is sufficiently differentiable and with proper decay (the exact requirements depend on whether we are deriving a Petersson trace formula or a Kuznetsov trace formula). 
We define the Hecke algebra action, globally and locally, as follows:
$$\rho \lb f \rb \varphi \lb h \rb =\int\limits_{ G \lb \A \rb }f \lb g \rb \varphi \lb hg \rb dg, \,\, \pi_v \lb f_v \rb \varphi_v=\int\limits_{ G \lb \Q_v \rb } f_v \lb g \rb \pi_v \lb g \rb \varphi_v dg.$$

\section{Minimal vector, microlocal lifts and newforms}\label{Sec3}

This section is purely local so we skip subscript $v$ from all notations. 

\subsection{Small families}

\begin{defn}\label{defn:FamilythetaN}
Let $\L$ be an \'{e}tale quadratic algebra over $\F$. 
	Let $\theta_i,i=1,2$ be characters of $\L^\times$ such that $\theta_i|_{\F^\times}=1$ and $\cc \lb \theta_1 \rb =\cc \lb \theta_2 \rb $. Denote $$i_0=\cc \lb \theta \rb /e_\L,$$ which is always an integer by  $\theta_i|_{\F^\times}=1$.
\yh{
This is because when $\L$ splits, we take $e_\L=1$ by  the convention from Section \ref{Sec2.1}. And when $e_\L=2$, $\cc(\theta)$ is always even. See the list in the beginning of Section \ref{Sec:3.2.1}.
}

	For $0\leq n< i_0$, denote $$\theta_1 \sim_n\theta_2$$ if $\cc \lb \theta_1^{-1}\theta_2 \rb \leq e_\L n$.

	For a fixed character $\theta$ with $\theta|_{\F^\times}=1$, denote $$\theta[n]=\{\theta' \text{ \ over }\L| \cc \lb \theta' \rb =\cc \lb \theta \rb ,\theta'|_{\F^\times}=1,\theta'\sim_n\theta\}.$$
\end{defn}
\begin{defn}\label{Def3.1:pithetan}
	Define $$\pi_\theta[n]=\{
 \pi_{\theta'}|\theta'\in\theta[n]\}.$$
\end{defn}
Here $\pi_{\theta'}$ is the representation associated to $\theta'$ either by the compact induction theory or the parabolic induction theory depending on $\L$ is a field or not. See Section \ref{Sec3.2} \ref{Sec3.3} for more details.
\begin{rem}\label{Remark:subtlefamily}
	When $n< i_0$, there is a bijection between $\theta[n]$ and $\pi_\theta[n]$. This is however not true when $n=i_0$, as $\pi_{\theta}\simeq \pi_{\overline{\theta}}$ while $\theta$ is not necessarily equal to $\overline{\theta}$.
\end{rem}
The following lemma gives an alternative description of the small families in terms of conductors. Though it is not necessary for understanding the remaining part of the paper.
\begin{lem}\label{Lem3.1:Conductorfamily}
	Let $\pi'=\pi_{\theta'}$, where  $\theta$ and $\theta'$ are both characters of $\L^\times$ with $\cc \lb \theta' \rb =\cc \lb \theta \rb\geq 2 $. Then $\pi'\in \pi_\theta[n]$ for $n<i_0$ if and only if $C \lb \pi_{\theta^{-1}}\times \pi_{\theta'} \rb \leq C \lb \pi_\theta \rb  p^{2n+e_\L-1}$.
\end{lem}
\begin{proof}
	As $\theta$ and $\theta'$ are both characters of $\L^\times$, $C \lb \pi_{\theta^{-1}}\times \pi_{\theta'} \rb =C \lb \pi_{\theta^{-1}\theta{}'} \rb C \lb \pi_{\theta^{-1}\overline{\theta}{}'} \rb$. (This equality can be proved by using the local Langlands correspondence, but we shall skip the details here.) 
 Since $p\neq 2$ and $\cc(\theta)\geq 2$, at least one of $\cc \lb \theta^{-1}\theta{}' \rb , \cc \lb \theta^{-1}\overline{\theta}{}' \rb $ is $\cc \lb \theta \rb $. 
	As $\pi_{\theta'}\simeq \pi_{\overline{\theta}'}$, we can assume without loss of generality that $\cc \lb \theta^{-1}\overline{\theta}{}' \rb =\cc \lb \theta \rb $. 
	
	Now $\pi'\in\pi_\theta[n]$ if and only if  $\cc \lb \theta^{-1}\theta{}' \rb \leq e_\L n$. The lemma now follows from Lemma \ref{Eq:cpictheta} below.
\end{proof}
\begin{lem}\label{Lem:Indexoffamily}
For $1\leq n<i_0$,
	$$[\theta[n]: \theta[0]]=p^nL_\F^{-1} \lb 1,\epsilon_{\L/\F} \rb =p^{n-1}\begin{cases}
	p-1, &\text{\ if $\L$ splits,}\\
	p+1, &\text{\ if $\L$ is an inert field extension,}\\
	p, &\text{\ if $e_\L=2$.}
	\end{cases}$$
\end{lem} 
\begin{proof}
In general for a finite abelian group $H$, let $\hat{H}$ denote the group of unitary characters on $H$.

When $\L$ is a field, we have the following bijection:
\begin{align}\label{Eq3.1:M2}
\theta[n]/\theta[0] &\rightarrow \widehat{U_\L(0)/O_\F^\times U_\L(e_\L n)}\\
\theta'&\mapsto \lb\theta^{-1}\theta'\rb|_{U_\L(0)}.\notag
\end{align}

Similarly when $\L$ splits, we have the following bijection:
\begin{align}\label{Eq3.1:M1}
\theta[n]/\theta[0] &\rightarrow \widehat{U_\L(0)/O_\F^\times U_\L(e_\L n)}=\widehat{U_\F(0)/U_\F(n)}\\
\theta'&\mapsto \chi|_{U_\F(0)}\text{\ if } \theta^{-1}\theta'=\lb \chi,\chi^{-1}\rb.\notag
\end{align}
Then the lemma follows directly from the general
Pontryagin duality for finite abelian groups, and counting the sets $U_\L(0)/O_\F^\times U_\L(e_\L n)$ and $
U_\F(0)/U_\F(n)$.

\end{proof}
\begin{rem}
It is also direct to see that $\sharp \theta[0]=1$ when $\L$ is  inert, and $\sharp \theta[0]=2$ when $\L$ is ramified. When $\L$ splits, $\theta[0]$ is however not finite, as any $\theta'=\theta\cdot (\chi_0,\chi_0^{-1})$ for an unramified character $\chi_0$ still belongs to $\theta[0]$.
\end{rem}

%
%
%
%

	For any $\theta'\in \theta[n]$, there is   an element $\alpha_{\theta'}\in   \varpi_\L^{-\cc \lb \theta \rb +\cc \lb \psi_\L \rb }O_\L^\times$   by Lemma \ref{Lem:LiealgChar}
	with 
	\begin{equation}\label{Eq:LiealgcharL}
\theta' \lb 1+u \rb =\psi_\L \lb \alpha_{\theta'}\log  \lb 1+u \rb  \rb ,\forall u\in\varpi_\L^{i}O_\L.
\end{equation}
	$\theta'|_{\F^\times}=1$ implies that $\alpha_\theta'$  ( and also $\alpha_\theta$) can be chosen to be imaginary, i.e. $\overline{\alpha}_\theta'=-\alpha_\theta'$ where $x\mapsto\overline{x}$ is the nontrivial automorphism of $\L/\F$.

\begin{lem}\label{Lem:ThetafamilyPara}
	Fix $n<i_0$. Suppose that $p$ is large enough or $1\leq j<n$ is large enough. For any $\theta'\in \theta[n]$, let $\alpha_{\theta'}$ be an imaginary element associated to $\theta'$ by Lemma \ref{Lem:LiealgChar}. Then we have the following bijection
	\begin{align}\label{Eq:isomap}
	\theta[n]/\sim_j &\rightarrow \alpha_\theta U_\F \lb i_0-n \rb /U_\F \lb i_0-j \rb \\
	\theta' &\mapsto \alpha_{\theta'}\notag
	\end{align}
\end{lem}
Here $j$ being large enough is similar to $i$ large enough in Lemma \ref{Lem:LiealgChar}, depending only on $\L$.
\begin{proof}

 We can write $\alpha_\theta'=\alpha_\theta u$ for $u\in O_\F^\times$, as $\cc \lb \theta \rb =\cc \lb \theta' \rb $ and both $\alpha_\theta,\alpha_\theta'$ are imaginary.
		From $\cc \lb \theta^{-1}\theta' \rb \leq e_\L n$, we get that $\theta^{-1}\theta'$ is trivial on $U_\L \lb e_\L n \rb $, whose image under $\log$ is $\varpi_\L^{e_\L n}O_\L=\varpi^{n} O_\L$. As the associated constant to $\theta^{-1}\theta'$ is $\alpha_\theta'-\alpha_\theta=\alpha_\theta \lb u-1 \rb $, we get that $$\psi_\L \lb \alpha_\theta \lb u-1 \rb x \rb =1, \,\forall x\in \varpi^{n} O_\L.$$ 
This implies that $u\in U_\F \lb i_0-n \rb $.
On the other hand, if $\alpha_\theta'\in \alpha_\theta U_\F \lb i_0-j \rb $, then by \eqref{Eq:LiealgcharL} we get that $\cc \lb \theta^{-1}\theta' \rb \leq e_\L j$. So the map in \eqref{Eq:isomap} is well-defined.
To show that the map is a bijection, it remains to see that the cardinalities of both sides agree using Lemma \ref{Lem:Indexoffamily}, which should be obvious. 
\end{proof}

\subsection{Supercuspidal representation case}\label{Sec3.2}

\subsubsection{Review of compact induction}\label{Sec:3.2.1}
Supercuspidal representations are  associated to characters $\theta$ defined over quadratic field extensions $\L$.
 The detailed construction can be found in, for example, \cite{BushnellHenniart:06a} with some different conventions.

Let $\F$ be a $p-$adic local field, $\L=\F \lb \sqrt{D} \rb $ be a quadratic field extension with ramification index $e_\L$.
We assumed that $v_\F \lb D \rb =e_\L-1$ and $p\neq 2$, and used the following embedding of $\L$ as a standard embedding:
\begin{equation}\label{Eq:standardembeddingL}
x+y\sqrt{D}\mapsto \zxz{x}{y}{yD}{x}.
\end{equation}


We  fix an additive character $\psi$ such that $\cc \lb \psi \rb =0$. Then $\cc \lb \psi_\L \rb =-e_\L+1$. 

The supercuspidal representations are parameterized by characters $\theta$ over some quadratic field extension $\L$  according to the compact induction theory. More specifically we have the following quick guide. \begin{enumerate}
	\item[Case 1.]$\cc \lb \pi \rb =2n+1$ corresponds to $e_\L=2$ and $\cc \lb \theta \rb =2n$ .
	\item[Case 2.] $\cc \lb \pi \rb =4n$ corresponds to $e_\L=1$ and $\cc \lb \theta \rb =2n$.
	\item[Case 3.] $\cc \lb \pi \rb =4n+2$ corresponds to $e_\L=1$ and $\cc \lb \theta \rb =2n+1$ .
\end{enumerate}

\yh{
Combining with the parabolic induction case, we have
\begin{lem}\label{Eq:cpictheta}
Let $\pi$ be an irreducible smooth admissible representation of $\GL_2(\F)$ with trivial central character and $\cc(\pi)\geq 3$, associated to a character $\theta$ of $\L^\times$. Then
\begin{equation}
\cc(\pi)= \frac{2}{e_\L}\cc(\theta) +e_\L-1=2i_0+e_\L-1.
\end{equation}
\end{lem}

}
\begin{defn}\label{defn:specialorder}
	For $e_{\L}=1,2$, define the hereditary orders $$\fA_{e_{\L}}=\begin{cases}
	M_2 \lb \CO_{\F} \rb , \text{\ if }e_\L=1,\\
	\zxz{ \CO_\F}{\CO_\F}{\varpi \CO_\F}{ \CO_\F},\text{\ otherwise}.
	\end{cases}
	$$
	Its Jacobson radical is given by
	$$\CB_{e_\L}=\begin{cases}
	\varpi M_2 \lb \CO_{\F} \rb , \text{\ if }e_\L=1,\\
	\zxz{\varpi \CO_\F}{\CO_\F}{\varpi \CO_\F}{\varpi \CO_\F},\text{\ otherwise}.
	\end{cases}$$
Define a filtration of compact open subgroups as follows:
\begin{equation}
K_{\fA_{e_\L}} \lb n \rb =1+\CB_{e_\L}^n.\ \ \ 
\end{equation}

\end{defn}
\begin{rem}\label{Rem3.2:KLvsL}
Note that each $K_{\fA_{e_\L}} \lb n \rb $ is normalised by $\L^\times$ which is embedded as in \eqref{Eq:standardembeddingL}.
Furthermore $K_{\fA_{e_\L}} \lb n \rb \cap \L=U_\L(n)$.
\end{rem}

Denote 
$$J=\L^\times K_{\fA_{e_\L}} \lb \lfloor \cc \lb \theta \rb /2\rfloor \rb ,\ J^1=U_\L \lb 1 \rb K_{\fA_{e_\L}} \lb \lfloor \cc \lb \theta \rb /2\rfloor \rb ,\ H^1=U_\L \lb 1 \rb K_{\fA_{e_\L}} \lb \lceil \cc \lb \theta \rb /2\rceil \rb .$$ Then $\theta$ on $\L^\times$ can be extended to be a character $\tilde{\theta}$ on $H^1$ by
\begin{equation}\label{Eq:thetatilde}
\tilde{\theta} \lb l  \lb 1+x \rb  \rb =\theta \lb l \rb \psi\circ\Tr  \lb \alpha_\theta x \rb ,
\end{equation}
where $l\in U_\L(1)$, $1+x\in K_{\fA_{e_\L}} \lb \lceil \cc \lb \theta \rb /2\rceil \rb $ and $\alpha_\theta\in \L^\times\subset M_2 \lb \F \rb $ is associated to $\theta$ by Lemma \ref{Lem:LiealgChar} under the fixed embedding.

When $\cc \lb \theta \rb $ is even, then $H^1=J^1$ by definition, and $\tilde{\theta}$ can be further extended to $J$ by the same formula with $l\in\L$. In this case denote $\Lambda=\tilde{\theta}$  and $\pi_\theta=c-\Ind_J^G\Lambda$ is an irreducible supercuspidal representation. $\pi_\theta\simeq \pi_{\theta'}$ if and only if $\theta=\theta'$ or $\overline{\theta}'$.

When $\cc \lb \theta \rb $ is odd, $J^1/H^1$ is a two dimensional vector space over the residue field. This case only occurs when $\cc \lb \pi \rb =4n+2$ as listed above.
Then there exists a $q-$dimensional (or $q-1$ dimensional if $\cc(\pi)=2$, but we will be mainly interested in the case when $\cc(\pi)$ is large enough) irreducible representation $\Lambda$ of $J$ such that
$\Lambda|_{H^1}$ is a multiple of $\tilde{\theta}$, and
\begin{equation}\label{Eq3.2:DecompLambda}
\Lambda|_{\L^\times}=\bigoplus\limits_{\theta'\in\theta[1],\theta'\neq \theta, \overline{\theta}} \theta'
\end{equation}

 More specifically, let $B^1$ be any intermediate group between $J^1$ and $H^1$ such that $B^1/H^1$ gives a polarisation of $J^1/H^1$ under the  pairing given by 
\begin{equation}\label{Eq3.2:antisymPairing}
 \lb 1+x,1+y \rb \mapsto \psi\circ \Tr  \lb \alpha_\theta [x, y] \rb .
\end{equation}

Then $\tilde{\theta}$ can be extended to $B^1$ by the same formula \eqref{Eq:thetatilde} and  
\begin{equation}\label{eq3.2:LambdaJ1}
\Lambda|_{J^1}=\Ind_{B^1}^{J^1}\tilde{\theta}.
\end{equation}
\eqref{Eq3.2:DecompLambda} and \eqref{eq3.2:LambdaJ1} determine the representation $\Lambda$ of $J$, and $\pi_\theta=c-\Ind_{J}^G\Lambda$ is irreducible and supercuspidal in this case, with $\pi_\theta\simeq \pi_{\theta'}$ if and only if $\theta=\theta'$ or $\overline{\theta}'$.  We always have $w_\pi=\theta|_{\F^\times}$.

Note that when $J^1\neq H^1$, any intermediate subgroup $B^1$ works, as the pairing \eqref{Eq3.2:antisymPairing} is skew-symmetric.
It will however be convenient to fix a choice of $B^1$ for later purposes.
\begin{defn}\label{Defn3.2:B1}
When $\L$ is inert and $\cc \lb \theta \rb =2n+1$, let
\begin{equation}
B^1=U_\L \lb 1 \rb K_{\fA_2} \lb 2n+1 \rb .
\end{equation} 
To be uniform, in the case where $\cc \lb \theta \rb $ is even and $J^1=H^1$,  we take $B^1=J^1$.
\end{defn}
\yh{Note that when $e_\L=1$ and $\cc(\theta)$ is odd, we utilized the compact subgroup constructed from the ramification index 2 case to define $B^1$.} In explicit coordinates, 
$$K_{\fA_2} \lb 2n+1 \rb =1+\zxz{\varpi^{n+1}O_\F}{\varpi^{n}O_\F}{\varpi^{n+1}O_\F}{\varpi^{n+1}O_\F},$$
which is  indeed  an intermediate subgroup between  $K_{\fA_1}(n+1)$ and $ K_{\fA_1} \lb n\rb$.

\begin{defn}
	There exists a unique element $\varphi_0\in \pi$ such that $B^1$ acts on it by $\tilde{\theta}$.  
	We  call any single translate $\pi \lb g \rb \varphi_0$ a minimal vector (Type 1 minimal vector in the notation of \cite{HN18}).
\end{defn}
Note that the conjugated group $gB^1g^{-1}$ acts on $\pi \lb g \rb \varphi_0$ by the conjugated character $\tilde{\theta}^g$. The following result is a direct consequence of the compact induction theory. 
\begin{cor}\label{Cor:MCofGeneralMinimalVec}
	Let $\Phi_{\varphi_0}$ be the matrix coefficient associated to a minimal vector $\varphi_0$ as above. Then $\Phi_{\varphi_0}$ is supported on $J$, and
	\begin{equation}
	\Phi_{\varphi_0} \lb bx \rb =\Phi_{\varphi_0} \lb xb \rb =\tilde{\theta} \lb b \rb \Phi_{\varphi_0} \lb x \rb 
	\end{equation}
	for any $b\in B^1$. Furthermore when $\dim \Lambda\neq 1$, $\Phi_{\varphi_0}|_{J^1}$ is supported only on $B^1$.
\end{cor}

Due to the central character, it is clear that $ZB^1$ acts on $\varphi_0$ by a character, which we also denote by $\tilde{\theta}$ without confusion.
We also need a converse result. 
\begin{prop}\label{Prop3.2:B1givespi}
Let $\pi$ be an irreducible smooth  representation of $\GL_2 \lb \F \rb $, with central character $w_\pi=\theta|_{\F^\times}$ and $\cc \lb \pi \rb \geq 3$. Suppose that there exists an element $\varphi\in \pi$ on which $ZB^1$ acts by a given character $\tilde{\theta}$, then $\varphi$ is unique up to a constant. Furthermore we must have $\pi\simeq\pi_{\theta'}$ where $\theta'\in \theta[l_0]$, for $l_0=1$ when $\L$ is inert, \old{ and $\cc \lb \theta \rb $ is odd,} and $l_0=0$ when $\L$ is ramified. \old{otherwise.}
\end{prop}
\begin{proof}

We consider only the case where $\L$ is inert and $\cc \lb \theta \rb $ is odd here, as the other cases are very similar and slightly easier. 

By the condition, $ZB^1$ acts on $\varphi$ by $\tilde{\theta}$. By the Frobenius reciprocity for compact inductions, we have
\begin{equation}\label{Eq3.2:FrobReciInd}
0\neq \Hom_{ZB^1} \lb \tilde{\theta}, \pi|_{ZB^1} \rb =\Hom_G \lb c-\Ind_{ZB^1}^{G}\tilde{\theta}, \pi \rb .
\end{equation}
We study $c-\Ind_{ZB^1}^{G}\tilde{\theta}$ step by step as the induction of representations is transitive.
Since $\Ind_{B^1}^{J_1}\tilde{\theta}=\Lambda|_{J^1}$, we have
$$\Ind_{ZB^1}^{ZJ^1}\tilde{\theta}=\Lambda|_{ZJ^1}.$$
For each $\theta'\in \theta[1]$, let $\Lambda_{\theta'}$ be an irreducible representation of $J$ constructed similarly as $\Lambda$. $\Lambda_{\theta'}$'s are not equivalent to each other by \eqref{Eq3.2:DecompLambda} and $\cc(\theta)\geq 2$. From $\theta'\in \theta[1]$, we get that $\Lambda_{\theta'}|_{ZJ^1}=\Lambda|_{ZJ^1}$, as the construction in \eqref{Eq:thetatilde} and \eqref{eq3.2:LambdaJ1} does not distinguish elements in $\theta[1]$. 
In particular we have $$\Hom_{J} \lb \Ind_{ZB^1}^{J}\tilde{\theta},\Lambda_{\theta'} \rb =\Hom_{ZJ^1} \lb \Ind_{ZB^1}^{ZJ_1}\tilde{\theta}, \Lambda|_{ZJ^1} \rb \neq 0.$$
Then we must have
$$\Ind_{ZB^1}^{J}\tilde{\theta}=\bigoplus_{\theta'\in\theta[1]} \Lambda_{\theta'}$$
as both sides are $(p+1)p$ dimensional.

Then  \eqref{Eq3.2:FrobReciInd} becomes
\begin{equation*}
\Hom_{ZB^1} \lb \tilde{\theta}, \pi|_{ZB^1} \rb =\bigoplus_{\theta'\in\theta[1]}\Hom_G \lb c-\Ind_{J}^{G}\Lambda_{\theta'}, \pi \rb =\bigoplus_{\theta'\in\theta[1]}\Hom_G \lb \pi_{\theta'}, \pi \rb .
\end{equation*}
From this we see that the righ-hand side is either trivial, or 1-dimensional when $\pi\simeq \pi_{\theta'}$ for some $\theta'\in \theta[1]$, as $\pi_{\theta'}$'s are irreducible and not mutually equivalent. The claims in the proposition are clear now.
\end{proof}
\subsubsection{Kirillov model and recovering the newform}
We also need to describe the minimal vectors explicitly in the Kirillov model. 

As we are going to vary $\theta$, we fix a choice of $D$  (unlike \cite{HuYinShu182, HN18}
), and assume that
\begin{equation}\label{Eq:specialAlphaTheta}
\alpha_\theta= \frac{\alpha_0}{\varpi_\L^{\cc \lb \theta \rb }\sqrt{D}} \mapsto \frac{\alpha_0}{\varpi^{\cc \lb \theta \rb /e_\L}}\zxz{0}{\frac{1}{D}}{1}{0}
\end{equation}
for certain $\alpha_0\in O_\F^\times$ by our assumption $\theta|_{\F^\times}=1$.
We define an intertwining operator from $\pi$ to its Whittaker model by
\begin{equation}\label{eq:3.4:IntertwiningtoWhittaker}
\varphi \mapsto W_\varphi \lb g \rb =\int\limits_{\F}\Phi_{\varphi,\varphi_0} \lb \zxz{\frac{\varpi^{\lfloor \cc \lb \pi \rb /2\rfloor}}{\alpha_0}}{0}{0}{1}\zxz{1}{n}{0}{1}g \rb \psi \lb -n \rb dn.
\end{equation}
Its image is in the Whittaker model by a change of variable in $n$.
It is $\GL_2(\Q_p)-$equivalent as $\pi(h)\varphi$ is mapped to the function 
\begin{align*}
g\mapsto &\int\limits_{\F}\Phi_{\pi(h)\varphi,\varphi_0} \lb \zxz{\frac{\varpi^{\lfloor \cc \lb \pi \rb /2\rfloor}}{\alpha_0}}{0}{0}{1}\zxz{1}{n}{0}{1}g \rb \psi \lb -n \rb dn\\
=&\int\limits_{\F}\Phi_{\varphi,\varphi_0} \lb \zxz{\frac{\varpi^{\lfloor \cc \lb \pi \rb /2\rfloor}}{\alpha_0}}{0}{0}{1}\zxz{1}{n}{0}{1}gh \rb \psi \lb -n \rb dn=\pi(h)W_{\varphi}(g).
\end{align*}
The Kirillov model associated to $\varphi$ is given by 
$$\varphi(x)=W_\varphi\lb\zxz{x}{}{}{1}\rb.$$
 The  intertwining operator \eqref{eq:3.4:IntertwiningtoWhittaker} is nontrivial by the following lemma:

\begin{lem}\label{lem:toricnewforminKirillov}
	Up to a constant multiple, a minimal vector $\varphi_0$ is given in the Kirillov model by the following: \begin{enumerate}
		\item When $\cc \lb \pi \rb =4n$, $\varphi_0=\Char \lb \varpi^{-2n}\alpha_0U_\F \lb n \rb  \rb  $.
		\item When $\cc \lb \pi \rb =2n+1$, $\varphi_0=\Char \lb \varpi^{-n}\alpha_0U_\F \lb \lceil n/2 \rceil \rb  \rb $.
		\item When $\cc \lb \pi \rb =4n+2$, $\varphi_0=\Char \lb \varpi^{-2n-1}\alpha_0U_\F \lb n+1 \rb  \rb $.
	\end{enumerate}
Here for any set $U$, $\Char(U)$ denotes the function given by
$$\Char\lb U\rb (x)=\begin{cases}
1, &\text{\ if }x\in U;\\
0, &\text{otherwise}.
\end{cases}$$
\end{lem}
The computations are essentially the same as in \cite[Lemma A.7]{HN18}. Using the notation $
	i_0=\frac{\cc \lb \theta \rb }{e_\L},
$
one can uniformly write \begin{equation*}
\varphi_0=\sqrt{ \lb p-1 \rb p^{\lceil i_0/2\rceil -1}}\Char \lb \varpi^{-i_0}\alpha_0 U_\F \lb \lceil i_0/2\rceil \rb  \rb .
\end{equation*}
Note here we have $L^2-$normalized $\varphi_0$. $i_0$ is roughly $\frac{\cc \lb \pi \rb }{2}$. 


\begin{rem}\label{Rem3.2:ONB}
From the explicit Kirillov model, 
and the local unitary pairing given by $$<\varphi_1,\varphi_2>=\int\limits_{x\in\F^\times}\varphi_1 \lb x \rb \overline{\varphi_2 \lb x \rb }d^\times x,$$
one can  see that the set $$B_\pi=\lB \pi\lb \zxz{a}{0}{0}{1}\zxz{1}{n}{0}{1}\rb\varphi_0 \mid a\in \F^\times/U_\F \lb \lceil i_0/2\rceil \rb , n\in \F/\varpi^{\lfloor i_0/2\rfloor}O_\F\rB$$
forms an orthogonal basis for $\pi$, and is invariant by any diagonal translation. 
\end{rem}
\begin{cor}\label{Cor:RelationNewMinimal}
For $a\in  \lb \CO_\F/\varpi^{\lceil i_0/2 \rceil} \CO_\F \rb ^\times$, let $\varphi_a=\pi \lb \zxz{\varpi^{-i_0}a^{-1}}{0}{0}{1} \rb \varphi_0$. 
	Then we have for  $\varphi_{new}=\Char \lb \O_\F^\times \rb $
	\begin{equation*}
	\varphi_{new}=\frac{1}{\sqrt{ \lb p-1 \rb p^{\lceil i_0/2 \rceil-1}}}
	\sum\limits_{a\in  \lb \CO_\F/\varpi^{\lceil i_0/2 \rceil} \CO_\F \rb ^\times} \varphi_a.
	\end{equation*}
	
\end{cor}



Note that $\varphi_a$ can be viewed as the minimal vector associated to the embedding 
\begin{equation}\label{Eq:conjugatedembedding}
x+y\sqrt{D}\mapsto \zxz{x}{\frac{y}{a\varpi^{i_0}}}{yDa\varpi^{i_0}}{x}.
\end{equation}
\begin{defn}\label{Defn3.1:Phiaa'}
	Define
	$
	\Phi_{0,0} \lb g \rb =<\pi \lb g \rb \varphi_0,\varphi_{0}> \text{\ with normalisation } \Phi_{0,0} \lb 1 \rb =1, $
	and define
	$
	\tilde{\Phi}_{0,0}=\Phi_{0,0}|_{Z B^1}$. Define in general for $a,a'\in  \lb \CO_\F/\varpi^{\lceil i_0/2 \rceil} \CO_\F \rb ^\times$
	$${\Phi}_{a,a'} \lb g \rb = {\Phi}_{0,0} \lb  \zxz{\varpi^{i_0}a'}{0}{0}{1}g\zxz{\varpi^{-i_0}a^{-1}}{0}{0}{1} \rb,\  \tilde{\Phi}_{a,a'} \lb g \rb = \tilde{\Phi}_{0,0} \lb  \zxz{\varpi^{i_0}a'}{0}{0}{1}g\zxz{\varpi^{-i_0}a^{-1}}{0}{0}{1} \rb .$$
\end{defn}


\begin{cor}\label{Cor:conjugatedMC}
	$\tilde{\Phi}_{0,0} \lb g \rb =0$ unless we can write $g=e\zxz{1+x}{m}{0}{1}$ or $\zxz{1+x}{m}{0}{1} e$ for some $e\in Z U_\L \lb 1 \rb $ 
with embedding as in \eqref{Eq:standardembeddingL}, $x\in \varpi^{\lceil i_0/2\rceil}\CO_\F$ and $m\in \varpi^{\lfloor i_0/2\rfloor}\CO_\F$. In that case, we have
	\begin{equation*}
	\tilde{\Phi}_{0,0} \lb g \rb 
	=\theta \lb e \rb \psi \lb \varpi^{-i_0} \alpha_0m \rb .
	\end{equation*}
Let $a\in  \lb \CO_\F/\varpi^{\lceil i_0/2 \rceil} \CO_\F \rb ^\times$. Then
	$\tilde{\Phi}_{a,a} \lb g \rb =0$ unless $g=e\zxz{1+x}{m}{0}{1}$ or $\zxz{1+x}{m}{0}{1} e$ for $e\in Z U_\L \lb 1 \rb $ 
with embedding as in \eqref{Eq:conjugatedembedding}, $x\in \varpi^{\lceil i_0/2\rceil}\CO_\F$ and $m\in \varpi^{-\lceil i_0/2\rceil}\CO_\F$. In that case, we have
	\begin{equation*}
	\tilde{\Phi}_{a,a} \lb g \rb 
	=\theta \lb e \rb \psi \lb \alpha_0am \rb .
	\end{equation*}
\end{cor}
\begin{proof}
Note that any $g\in \GL_2$ can be written as a product of $e\in \L$ with some element in the Borel subgroup.
	Then the first part of the corollary  follows from \eqref{Eq:thetatilde}, Corollary \ref{Cor:MCofGeneralMinimalVec}, 
 the explicit shape of $\alpha_\theta$ as in \eqref{Eq:specialAlphaTheta} and the definition of $\tilde{\Phi}_{0,0}$. The second part follows  from the first part and the definition of  $\tilde{\Phi}_{a,a}$ in terms of $\tilde{\Phi}_{0,0}$.
\end{proof}

\begin{defn}\label{Def3.2:testfSC}
For  a quadratic field extension $\L$ and a character $\theta$ on it, choose the local test function to be
\begin{equation}\label{Eq3.2:TestfpSC}
f \lb g \rb =\frac{1}{ \lb p-1 \rb p^{\lceil i_0/2\rceil-1}\Vol \lb Z\backslash Z B^1 \rb }\sum\limits_{a,a'\in  \lb \CO_\F/\varpi^{\lceil i_0/2 \rceil} \CO_\F \rb ^\times} \overline{\tilde{\Phi}}_{a,a'} \lb g \rb .
\end{equation}
\end{defn}

\begin{prop}\label{Prop3.2:testfunaction}
Let $f$ be as in \eqref{Eq3.2:TestfpSC}, and let $\pi$ be an irreducible smooth representation of $\GL_2 \lb \F \rb $ with trivial central character. Then
 $\pi \lb f \rb $ is zero unless $\pi\simeq\pi_{\theta'}$ where $\theta'\in \theta[l_0]$, in which case $\pi \lb f \rb $ is the projection to the line generated by the newform.
%
\end{prop}
\begin{proof}
We first discuss $\pi \lb \overline{\tilde{\Phi}}_{0,0} \rb $. 
If $\pi \lb \overline{\tilde{\Phi}}_{0,0} \rb \varphi \neq 0$, then by a change of variable the element $\varphi'=\pi \lb \overline{\tilde{\Phi}}_{0,0} \rb \varphi$ has the property  that $B^1$ acts on it by $\tilde{\Phi}_{0,0}=\tilde{\theta}$.
According to Proposition \ref{Prop3.2:B1givespi}, $\pi\simeq\pi_{\theta'}$ for $\theta'\in \theta[l_0]$.

 In that case, we also know that $\varphi'$ must be a multiple of $\varphi_0$. 
We choose the orthonormal basis $B_\pi$ as in Remark \ref{Rem3.2:ONB}. 
Then we have
$$<\pi \lb \overline{\tilde{\Phi}}_{0,0} \rb \varphi,\varphi_0>=<\varphi,\pi \lb {\tilde{\Phi}}^{-1}_{0,0} \rb \varphi_0>=\Vol \lb Z\backslash Z B^1 \rb <\varphi,\varphi_0>,$$
which implies that if $\varphi\in B_\pi$, then $\pi \lb \overline{\tilde{\Phi}}_{0,0} \rb \varphi=0$ unless $\varphi=\varphi_0$. Thus $\pi\lb\frac{1}{\Vol \lb Z\backslash Z B^1 \rb }\overline{\tilde{\Phi}}_{0,0}\rb$ is the projection onto the line spanned by $\varphi_0$.

Now for any $a,a' \in  \lb \CO_\F/\varpi^{\lceil i_0/2 \rceil} \CO_\F \rb ^\times$, $\varphi\in B_\pi$, we have by definition
\begin{align}
\pi \lb \overline{\tilde{\Phi}}_{a,a'} \rb \varphi&=\int\limits_{g\in Z\backslash Z B^1}{\tilde{\theta}}^{-1} \lb g \rb \pi\lb\zxz{\varpi^{-i_0}a'{}^{-1}}{}{}{1} g \zxz{\varpi^{i_0}a}{}{}{1}\rb\varphi dg\\
&=\pi\lb \zxz{\varpi^{-i_0}a'{}^{-1}}{}{}{1}\rb \pi \lb \overline{\tilde{\Phi}}_{0,0} \rb \pi\lb \zxz{\varpi^{i_0}a}{}{}{1}\rb\varphi.\notag
\end{align}
As $B_\pi$ is invariant by diagonal translates (up to constants), we see from the previous discussion that $$\pi\lb \frac{1}{\Vol \lb Z\backslash Z B^1 \rb }\overline{\tilde{\Phi}}_{a,a'}\rb\varphi=0$$ unless $\varphi=\pi\lb \zxz{\varpi^{-i_0}a^{-1}}{}{}{1}\rb\varphi_0=\varphi_a$, in which case $$\pi\lb \frac{1}{\Vol \lb Z\backslash Z B^1 \rb }\overline{\tilde{\Phi}}_{a,a'}\rb\varphi=\varphi_{a'}.$$
By Corollary \ref{Cor:RelationNewMinimal} and Definition \ref{Def3.2:testfSC}, we get that
\begin{align*}
\pi \lb f \rb \varphi_{new}&=\frac{1}{ \lb p-1 \rb p^{\lceil i_0/2\rceil-1}}\sum\limits_{a,a'}\pi\lb\frac{1}{\Vol \lb Z\backslash Z B^1 \rb }\overline{\tilde{\Phi}}_{a,a'}\rb\frac{1}{\sqrt{ \lb p-1 \rb p^{\lceil i_0/2 \rceil-1}}}\sum\limits_{b}\varphi_{b}\\
&=\frac{1}{ \lb p-1 \rb p^{\lceil i_0/2\rceil-1}}\frac{1}{\sqrt{ \lb p-1 \rb p^{\lceil i_0/2 \rceil-1}}}\sum\limits_{a',b}\varphi_{a'}=\varphi_{new}
\end{align*}
as $\sharp \lb \CO_\F/\varpi^{\lceil i_0/2 \rceil} \CO_\F \rb ^\times= \lb p-1 \rb p^{\lceil i_0/2\rceil-1}.$
\end{proof}

\begin{rem}\label{Rem3.2:Issuesforsmallerfamily}
It may seem possible and desirable to devise $f$ so that one can take $l_0=0$ also for the $e_\L=1$ case. We work with $l_0=1$ in this paper because of the following reasons: 
\begin{enumerate}
\item 
When $\cc \lb \pi_\theta \rb =4n+2$, it is still complicated (though not impossible) to write down and make use of the matrix coefficients on the whole  group $J$, compared to its restriction to $ZB^1$.
\item When $\cc \lb \pi_\theta \rb =4n$, one can easily start from $l_0=0$ and $k\geq i_0$. 
One small benefit to start with $l_0=1$ is that the formulations in Theorem \ref{Theo1.2:PTF} \ref{Theo1.2:PTFaverage} are relatively more uniform for the supercuspidal representation cases. 
The proof of Lemma\ref{Lem:AverGeneralKL} in Section \ref{Sec4.6.1} also becomes slightly easier when $k>i_0$ holds.
\item For applications, considering the family $\theta[1]$ instead of $\theta[0]$ would affect asymptotic bounds by a fixed power of $p$, which is negligible for depth aspect problem considered in this paper.
\end{enumerate}

\end{rem}

\subsection{Principal series representation case}\label{Sec3.3}
We remark that when the central character $w_\pi$ is trivial, $p\neq 2$ and $\cc \lb \pi \rb \geq 3$ , $\pi$ can neither be a Steinberg representation nor a twisted complementary series representation. Henceforth we assume $\pi$ to be parabolically induced from two unitary characters.
\subsubsection{Microlocal lifts}
Here we review the microlocal lifts from \cite{nelson_microlocal_2016}. For convenience, we mainly restrict ourselves to the case where the central character is trivial, but the approach can be easily extended to more general cases.

We start with slightly more general situations. Let $\pi=\pi \lb \chi_1,\chi_2 \rb $ be a principal series representation, whose elements $\varphi\in\pi$ satisfy
$$\varphi\lb\zxz{a}{n}{0}{b}g\rb=\chi_1(a)\chi_2(b)|\frac{a}{b}|^{1/2}\varphi(g).$$
Let $\pi_1=\pi \lb 1,\chi_1^{-1}\chi_2 \rb =\pi\otimes\chi_1^{-1}$, so that $\pi=\pi_1\otimes \chi_1$. In this case denote $i_0=\cc \lb \chi_1^{-1}\chi_2 \rb $. Let $$K_0 \lb \varpi^{i_0} \rb =\left\{\zxz{a}{b}{c}{d}\equiv \zxz{*}{*}{0}{*}\mod\varpi^{i_0}\right\}$$ be the usual congruence subgroup.

\begin{lem}\label{Lem:twistnewform}
	The exists a unique  (up to constant) element $\varphi_1\in \pi_1$  such that $K_0 \lb \varpi^{i_0} \rb $ acts on $\varphi_1$ by $\chi_1^{-1}\chi_2 \lb d \rb $. The normalised Whittaker function associated to $\varphi_1$ is given by
	$$W_{\varphi_1}\lb\zxz{\alpha}{0}{0}{1}\rb=\sqrt{1-p^{-1}}\begin{cases}
	p^{-v \lb \alpha \rb /2}, &\text{\ if $v \lb \alpha \rb \geq 0$,}\\
	0, &\text{\ otherwise.}
	\end{cases}$$
	Furthermore if there exists an element $\varphi'$ from an irreducible smooth admissible representation $\pi'$ such that $K_0 \lb \varpi^{i_0} \rb $ acts on $\varphi'$ by $\chi_1^{-1}\chi_2 \lb d \rb $, then $\pi'\simeq \pi \lb \nu_1,\nu_2 \chi_1^{-1}\chi_2 \rb $ for some unramified characters $\nu_1,\nu_2$.
\end{lem}
\begin{proof}
	The existence/uniqueness of $\varphi_1$ follows simply from the newform theory in \cite{Casselman73}. In the parabolic induction model, $\varphi_1$ is supported only on $BK_0 \lb \varpi^{i_0} \rb $. Furthermore, for any $\varphi'\in\pi'$ with the same equivalent property, $\varphi'$ is in particular invariant by $$K_1 \lb \varpi^{i_0} \rb =\left\{\zxz{a}{b}{c}{d}\equiv \zxz{*}{*}{0}{1}\mod\varpi^{i_0}\right\},$$
	so $\cc \lb \pi' \rb \leq i_0$. On the other hand the equivalent property implies that $w_{\pi'}|_{\CO^\times}=\chi_1^{-1}\chi_2$. Then $\cc \lb \pi' \rb \geq \cc \lb w_{\pi'} \rb  =i_0$. This forces $\pi'$ to be in the specified shape.
	
	The expression for $W_{\varphi_1}$ follows immediately from, for example, \cite[Lemma 2.13]{hu_triple_2017}.
\end{proof}


For uniformity, let $\L$ denote the diagonal torus, and let $\theta$ be the character $ \lb \chi_1,\chi_2 \rb $.  We associate the pair $ \lb \L,\theta \rb $ to  the principal series representation $\pi=\pi \lb \chi_1,\chi_2 \rb $, and simply write $\pi=\pi_\theta$. 

Let $\tilde{\theta}$ be the character on $ZK_0 \lb \varpi^{i_0} \rb $ defined by 
\begin{equation}
\tilde{\theta} \lb zg \rb =\chi_1\chi_2 \lb z \rb \chi_1 \lb \det g \rb  \chi_1^{-1}\chi_2 \lb d \rb ,
\end{equation}
where $z\in Z$, $g\in K_0 \lb \varpi^{i_0} \rb $. \yh{The microlocal lift we need is an element in the following sense.}

\begin{cor}\label{Cor:microlocallift}
	There exists a unique element (up to a constant) $\varphi_\theta\in \pi=\pi \lb \chi_1,\chi_2 \rb $ such that $ZK_0 \lb \varpi^{i_0} \rb $ acts on $\varphi_\theta$ by $\tilde{\theta}$. 
	The associated Whittaker function for $\varphi_\theta$ is given by
	\begin{equation}
	W_{\varphi_\theta}\lb\zxz{\alpha}{0}{0}{1}\rb=\sqrt{1-p^{-1}}\begin{cases}
	p^{-v \lb \alpha \rb /2}\chi_1 \lb \alpha \rb , &\text{\ if $v \lb \alpha \rb \geq 0$},\\
	0, &\text{\ otherwise.}
	\end{cases}
	\end{equation}
Conversely, if there is an element $\varphi\in \pi'$ such that $ZK_0 \lb \varpi^{i_0} \rb $ acts on it by $\tilde{\theta}$, then $\pi'\simeq\pi(\nu\chi_1,\nu^{-1}\chi_2)$ for some unramified character $\nu$.
\end{cor}
\begin{proof}
	The results follow directly from Lemma \ref{Lem:twistnewform} by a twist, and the requirement for the central character to be $\chi_1\chi_2$.
\end{proof}
In particular if we assume the central character to be trivial, we get $\pi'=\pi_{\theta'}$ for some $\theta'\in\theta[0]$ as in Definition \ref{defn:FamilythetaN}.
\subsubsection{Recovering the newform}

\begin{lem}\label{Lem3.3:newformasMLL}
	Denote $c_1=\cc \lb \chi_1 \rb $,  $\varphi_a=\pi\lb \zxz{1}{\frac{a}{\varpi^{c_1}}}{0}{1}\rb \varphi_\theta$, and  $$C_0= \lb 1-p^{-1} \rb^{3/2} p^{c_1}\int\limits_{x\in \CO^\times}\chi_1 \lb x \rb \psi \lb \varpi^{-c_1}x \rb d^\times x.$$ Then the newform can be written as
	
	$$\varphi_{new}=\Char \lb O_\F^\times \rb=\frac{1}{C_0}\sum\limits_{a\in  \lb \CO/\varpi^{c_1}\CO \rb ^\times}  \chi_1 \lb a \rb \varphi_a.
	$$
\end{lem}
\begin{proof}
	Note that one can alternatively write $$C_0=\sqrt{1-p^{-1}}\sum\limits_{x\in  \lb \CO/\varpi^{c_1}\CO \rb ^\times}\chi_1 \lb x \rb \psi \lb \varpi^{-c_1}x \rb .$$ In the Kirillov/Whittaker model, we have for $v(x)\geq 0$,
	
	\begin{align*}
	\sum\limits_{a\in  \lb \CO/\varpi^{c_1}\CO \rb ^\times}\chi_1 \lb a \rb W_{\varphi_a}\lb\zxz{ x }{0}{0}{1}\rb&=\sum\limits_{a\in  \lb \CO/\varpi^{c_1}\CO \rb ^\times}\chi_1 \lb a \rb \pi\lb\zxz{1}{\frac{a}{\varpi^{c_1}}}{0}{1}\rb W_{\varphi_\theta}\lb\zxz{ x }{0}{0}{1}\rb\\
	&=\sum\limits_{a\in  \lb \CO/\varpi^{c_1}\CO \rb ^\times}\chi_1 \lb a \rb \psi\lb\frac{a x }{\varpi^{c_1}}\rb W_{\varphi_\theta}\lb\zxz{ x }{0}{0}{1}\rb\\
	&=\sqrt{1-p^{-1}}p^{-v \lb  x  \rb /2}\sum\limits_{a\in  \lb \CO/\varpi^{c_1}\CO \rb ^\times}\psi\lb\frac{a x }{\varpi^{c_1}}\rb \chi_1 \lb a x  \rb .
	\end{align*}
	Here we used Corollary \ref{Cor:microlocallift} for the third line. The sum is automatically 0 when $v(x)<0$.
	Note that when $v \lb  x  \rb >0$, the sum in $a$ in the last line will be vanishing as the levels do not match. Thus by a change of variable, we have
	$$\frac{1}{C_0}\sum\limits_{a\in  \lb \CO/\varpi^{c_1}\CO \rb ^\times}\chi_1 \lb a \rb W_{\varphi_a}\lb\zxz{x}{0}{0}{1}\rb=\Char  \lb \CO^\times \rb =W_{\varphi_{new}}.$$
\end{proof}

From now on we assume that $\pi=\pi \lb \chi_1,\chi_1^{-1} \rb $, $p\neq 2$ and $\cc(\chi_1)\geq 2$, so that \begin{equation}\label{Eq3.3:i0PS}
i_0=\cc(\chi_1)=c_1.
\end{equation} Then the character $\tilde{\theta}$ can be rewritten as
\begin{equation}\label{Eq3.2:newtheta}
\tilde{\theta}\lb z\zxz{a}{b}{c}{d}\rb=\chi_1 \lb a \rb \chi_1^{-1} \lb d \rb , \text{\ for }\forall \zxz{a}{b}{c}{d}\in K_0 \lb \varpi^{i_0} \rb .
\end{equation}

\begin{defn}\label{Defn3.2:Phiaa'Prinicipal}
	Define $\Phi_{0,0} \lb g \rb =<\pi \lb g \rb \varphi_\theta,\varphi_\theta> \text{\ with normalisation } \Phi_{0,0} \lb 1 \rb =1$, $\tilde{\Phi}_{0,0}=\Phi_{0,0}|_{ZK_0 \lb \varpi^{i_0} \rb }$, and define for $a,a'\in  \lb \CO/\varpi^{i_0}\CO \rb ^\times$
	\begin{equation}\label{Eq:Phiaa'Princial0}
	{\Phi}_{a,a'} \lb g \rb =\chi_1 \lb a \rb \chi_1^{-1} \lb a' \rb {\Phi}_{0,0}\lb\zxz{1}{-a'\varpi^{-i_0}}{0}{1} g\zxz{1}{a\varpi^{-i_0}}{0}{1} \rb
	\end{equation}	
\begin{equation}\label{Eq:Phiaa'Princial}
	\tilde{\Phi}_{a,a'} \lb g \rb =\chi_1 \lb a \rb \chi_1^{-1} \lb a' \rb \tilde{\Phi}_{0,0}\lb\zxz{1}{-a'\varpi^{-i_0}}{0}{1} g\zxz{1}{a\varpi^{-i_0}}{0}{1} \rb
	\end{equation}
\end{defn}

\begin{defn}\label{Def3.3:testPS}
Define the following test function
\begin{equation}\label{Eq3.3:testfun}
f \lb g \rb =\frac{1}{ \lb p-1 \rb p^{i_0-1} \Vol \lb Z\backslash ZK_0 \lb \varpi^{i_0} \rb  \rb }\sum\limits_{a,a'\in  \lb \CO_\F/\varpi^{i_0} \CO_\F \rb ^\times} \overline{\tilde{\Phi}}_{a,a'} \lb g \rb .
\end{equation}
\end{defn}

\begin{prop}\label{Prop3.3:testfunaction}
\yh{
For $\L$ split, $\theta=(\chi_1,\chi_1^{-1})$, $f$ defined in \eqref{Eq3.3:testfun}, and  $\pi$ an irreducible smooth representation of $\GL_2 \lb \F \rb $ with trivial central character, we have that
 $\pi \lb f \rb $ is zero unless $\pi\simeq\pi_{\theta'}$ where $\theta'\in \theta[0]$, in which case $\pi \lb f \rb $ is the projection to the line generated by the newform.
}
\end{prop}
\begin{proof}
The proof is parallel to that of Proposition \ref{Prop3.2:testfunaction}. We first specify the orthonormal basis we are going to work with. First of all, the elements in the set
\begin{equation}\label{Eq3.3:Orthogonalset}
\lB \pi\lb\zxz{1}{n}{0}{1}\rb \varphi_\theta\,\middle\vert\, n\in \F/O_\F \rB
\end{equation}
are orthogonal to each other. This will be shown in the proof of Lemma \ref{Lem:PrincipalONB} below.
Then we  complete an orthonormal basis $B_\pi$ from \eqref{Eq3.3:Orthogonalset}. 

As in the proof of Proposition \ref{Prop3.2:testfunaction}, we get that $\pi\lb\frac{1}{\Vol \lb Z\backslash ZK_0 \lb \varpi^{i_0} \rb  \rb }\overline{\tilde{\Phi}}_{0,0}\rb$ is the projection onto the line spanned by $\varphi_\theta$
by Corollary \ref{Cor:microlocallift}.
Then as $\overline{\chi_1}=\chi_1^{-1}$, $$\pi\lb\frac{1}{\Vol \lb Z\backslash ZK_0 \lb \varpi^{i_0} \rb  \rb }\overline{\tilde{\Phi}}_{a,a'}\rb\varphi=\begin{cases}
0, &\text{\ if $\varphi\in B_\pi, \varphi\neq \varphi_a$,}\\
\chi^{-1}_1 \lb a \rb \chi_1 \lb a' \rb \varphi_{a'}, &\text{\ if $\varphi=\varphi_a$}
\end{cases}$$ 
Using Lemma \ref{Lem3.3:newformasMLL}, we get that
\begin{align*}
\pi \lb f \rb \varphi_{new}&=\frac{1}{C_0 \lb p-1 \rb p^{i_0-1}}\sum\limits_{a,a'}\pi\lb\frac{1}{\Vol \lb Z\backslash ZK_0 \lb\varpi^{i_0} \rb  \rb }\overline{\tilde{\Phi}}_{a,a'}\rb\sum\limits_{b}\chi_1 \lb b \rb \varphi_{b}\\
&=\frac{1}{C_0 \lb p-1 \rb p^{i_0-1}}\sum\limits_{a',b}\chi_1 \lb a' \rb  \varphi_{a'}=\varphi_{new}.
\end{align*}

\end{proof}

\subsubsection{$K-$type generated by $\varphi_\theta$}

Let $K'=\lB g\in \GL_2 \lb\F \rb \cap \zxz{\CO_\F}{\varpi^{-i_0}\CO_\F}{\varpi^{i_0}\CO_\F}{\CO_\F} \rB$. We shall discuss the representation $\sigma$ of $K'$ generated by $\varphi_\theta$ here, which might have independent interest. It will also be used in Lemma \ref{Lem5.1:twoApproach}. 
\begin{lem}\label{Lem:PrincipalONB}
	Let $\pi=\pi \lb\chi_1,\chi_1^{-1} \rb $ be a unitary principal series representation with $i_0=\cc(\chi_1)\geq 1$, and $\varphi_\theta\in\pi$ be as in Corollary \ref{Cor:microlocallift}. Let $\sigma$ be the representation of $K'$ generated by $\varphi_\theta$.
	The set $\{ \pi \lb g \rb  \varphi_\theta\}_{g\in K'/K_0 \lb\varpi^{i_0} \rb }$ provides an orthonormal basis for
	the representation $\sigma$, which is dimension $[K':K_0 \lb\varpi^{i_0} \rb ]= \lb p+1 \rb p^{i_0-1}$.
\end{lem}
Note that $\chi_1$ is automatically a unitary character by the setting.
\begin{proof}
	It is straightforward to verify that we can choose the coset representatives as follows:
	\begin{equation}\label{Eq3.2:cosetrepresentative}
	K'/K_0 \lb\varpi^{i_0} \rb =\coprod\limits_{x\in \varpi^{-i_0}\CO_\F/\CO_\F} \zxz{1}{x}{0}{1}\cup \coprod\limits_{x\in \varpi^{-i_0+1}\CO_\F/\CO_\F} \zxz{0}{\varpi^{-i_0}}{\varpi^{i_0}}{0}\zxz{1}{x}{0}{1}
	\end{equation}
	where the RHS has exactly $ \lb p+1 \rb p^{i_0-1}$ elements.
	
	Let $g,g'$ be any two different elements from the right-hand side of \eqref{Eq3.2:cosetrepresentative}.
	By the invariance of the unitary pairing, we have $$<\pi \lb g \rb \varphi_\theta,\pi \lb g' \rb \varphi_\theta>=<\pi \lb{g'}^{-1}g \rb \varphi_\theta,\varphi_\theta>,$$ where ${g'}^{-1}g \in K'-K_0 \lb\varpi^{i_0} \rb $.
	
	Thus for the orthorgonality, it suffices to show that for any coset representative $g\neq 1$, $$<\pi \lb g \rb \varphi_\theta,\varphi_\theta>=0.$$
	
	Let $g=\zxz{1}{x}{0}{1}$ for $x\notin \CO_\F$ first. Then using Corollary \ref{Cor:microlocallift},
	\begin{align*}
	<\pi \lb g \rb \varphi_\theta,\varphi_\theta>&=\int\limits_{\alpha\in \F^\times} W_{\varphi_\theta}\lb\zxz{\alpha}{0}{0}{1}\zxz{1}{x}{0}{1}\rb\overline{W_{\varphi_\theta}\lb\zxz{\alpha}{0}{0}{1}\rb}d^\times \alpha=\int\limits_{v \lb\alpha \rb \geq 0}p^{-v \lb\alpha \rb }\psi \lb\alpha x \rb d^\times\alpha\\
	&=\frac{1}{1-p^{-1}}\int\limits_{v \lb\alpha \rb \geq 0}\psi \lb\alpha x \rb d\alpha=0.\notag
	\end{align*}
	
	Now let $g=\zxz{0}{\varpi^{-i_0}}{\varpi^{i_0}}{0}\zxz{1}{x}{0}{1}$ with $v \lb x \rb \geq -i_0+1$, let $K$ be the standard maximal compact open subgroup. Then up to a constant multiple, we have by \eqref{Eq2.3:UnitaryPricipal}
	\begin{align*}
	<\pi \lb g \rb \varphi_\theta,\varphi_\theta>=\int\limits_{k\in K}\varphi_\theta \lb k\zxz{0}{\varpi^{-i_0}}{\varpi^{i_0}}{0}\zxz{1}{x}{0}{1}\rb \overline{\varphi_\theta} \lb k \rb dk\sim \int\limits_{k\in K_0 \lb\varpi^{i_0} \rb }\varphi_\theta \lb k\zxz{0}{\varpi^{-i_0}}{\varpi^{i_0}}{0}\zxz{1}{x}{0}{1}\rb \tilde{\theta}^{-1} \lb k \rb dk.
	\end{align*}
\yh{Here $\sim$ means equality up to a nonzero constant.}
	We have also used that $\varphi_\theta$ in the parabolic induction model is only supported on $BK_0 \lb\varpi^{i_0} \rb $.
	Writing $$k=\zxz{k_1}{k_2}{\varpi^{i_0}k_3}{k_4},$$ for $k_1,k_4\in O_\F^\times$ and $k_2,k_3\in O_\F,$ we have
	$$k\zxz{0}{\varpi^{-i_0}}{\varpi^{i_0}}{0}\zxz{1}{x}{0}{1}=\zxz{k_2\varpi^{i_0}}{k_1\varpi^{-i_0}+k_2x\varpi^{i_0}}{k_4\varpi^{i_0}}{k_3+k_4x\varpi^{i_0}}.$$
	As $v \lb k_4x\varpi^{i_0} \rb >0$, we need $v \lb k_3 \rb =0$ for the matrix above to land in the support of $\varphi_\theta$, which is $BK_0 \lb\varpi^{i_0} \rb $.
	In that case we can write the matrix above as
	$$\zxz{-\frac{\det k}{k_3+k_4x\varpi^{i_0}}}{k_1\varpi^{-i_0}+k_2x\varpi^{i_0}}{0}{k_3+k_4x\varpi^{i_0}}\zxz{1}{0}{\frac{k_4\varpi^{i_0}}{k_3+k_4x\varpi^{i_0}}}{1},$$
	thus
	\begin{align*}
	<\pi \lb g \rb \varphi_\theta,\varphi_\theta>\sim &\int\limits_{k\in K_0 \lb\varpi^{i_0} \rb }\chi_1\lb\frac{\det k}{k_3+k_4x\varpi^{i_0}}\rb\chi_1^{-1} \lb k_3+k_4x\varpi^{i_0} \rb \chi_1^{-1} \lb k_1 \rb \chi_1 \lb k_4 \rb dk\\
	=&	\int\limits_{k\in K_0 \lb\varpi^{i_0} \rb }\chi_1^{-2}\lb \frac{k_3}{k_4}+x\varpi^{i_0}  \rb dk=0.	\notag
	\end{align*}
	Here we have used  \eqref{Eq3.2:newtheta} for $\tilde{\theta}(k)$, and that $\chi_1(\det k)=\chi_1(k_1k_4)$ as $\cc \lb\chi_1 \rb =i_0$.
\end{proof}

%

\section{A refined Petersson trace formula}\label{Sec4}
Fix an \'etale quadratic algebra $\L$ over $\F=\Q_p$ at a fixed place $p$, and a character $\theta$ on $\L^\times$. 
Let $\cc=\cc \lb \pi_\theta \rb $ be the level of $\pi_\theta$. Fix an even weight $\kappa\geq 4$. Let $n, i_0$ be as in Definition \ref{defn:FamilythetaN}. Define
\begin{align}\label{Eq4.0:family}
\mathcal{F}_\theta[n]=\{&\text{holomorphic newforms $\varphi$ of weight $\kappa$, level $N=p^\cc$ and trivial nebentypus} \\
& \text{ such that $\pi_p\in \pi_\theta[n]$ where $\pi_p$ is the local representation  associated to $\varphi$}\}. \notag
\end{align}

We shall develop refined Petersson trace formula where only the members of $\mathcal{F}_\theta[n]$ appear on the spectral side. We shall start with smaller families and get the larger families by summation.
\subsection{Test function}\label{Sec4.1:testfun}
We shall make the standard choice for the local test functions when $v\neq p$. In particular $f_v=\Char \lb Z\GL_2 \lb \CO_v \rb  \rb $ for any non-archimedean place $v\neq p$. $f_\infty$ is the conjugate of the matrix coefficient for the lowest weight element of $\pi_\infty$, normalized to be an idempotent element under convolution. Explicitly one can take
\begin{equation}\label{Eq4.1:finfty}
f_\infty \lb g \rb =\begin{cases}
\frac{\kappa-1}{4\pi}\frac{\det \lb g \rb ^{\kappa/2} \lb 2i \rb ^\kappa}{ \lb -b+c+ \lb a+d \rb i \rb ^\kappa}, &\text{\ if $g=\zxz{a}{b}{c}{d}$ with  }\det \lb g \rb >0,\\
0, &\text{\ otherwise}.
\end{cases}
\end{equation}

At the place $p$, $f_p$ is chosen to be \eqref{Eq3.2:TestfpSC} or \eqref{Eq3.3:testfun}, depending on $\L$ or the local representations we are interested in. 

Let now $f=\otimes f_v$. Note that $f$ is $Z-$invariant by our choices.

\subsection{Relative trace formula for integrals along unipotent orbits}\label{Sec4.2:RTFsetup}
For more details relevant to this section one can see \cite{KL06, knightly_kuznetsovs_2013}.
Let $\psi$ be a fixed additive character of $\Q\backslash \A$. 
Recall from Definition \ref{defn:FamilythetaN} that

	 \begin{equation}\label{Eq:i0}
	i_0=\frac{\cc \lb \theta \rb }{e_\L}.
	\end{equation}	
Recall that when $\L\simeq \F\times \F$, we use the convention that $\cc(\theta)=\cc(\chi)$ if $\theta=(\chi,\chi^{-1})$, and $e_\L=1$.

Alternatively one can define
\begin{equation*}
i_0=\lfloor \frac{\cc(\pi_\theta)}{2}\rfloor.
\end{equation*}

To get the relative trace formula associated to the unipotent period integrals, we start with a pretrace formula for  $f$ specified above:
\begin{equation}\label{Eq4.2:pretrace}
\sum\limits_{\varphi}\frac{1}{||\varphi||^2}\rho \lb f \rb \varphi \lb x \rb \overline{\varphi} \lb y \rb=\sum\limits_{\gamma\in G(\Q)} f(x^{-1}\gamma y)
\end{equation}
Here $||\cdot||$ denotes the $L^2-$norm, $G=\text{PGL}_2$, and 
$$\rho(f)\varphi=\int\limits_{G(\A)}f(g)\rho(g)\varphi.$$
The sum in $\varphi$ is over some orthogonal basis for holomorphic automorphic forms with trivial central character, extended from the holomorphic newforms. 

Then by the choice of $f$ specified in Section \ref{Sec4.1:testfun}, and Proposition \ref{Prop3.2:testfunaction} \ref{Prop3.3:testfunaction}, the sum for $\varphi$ is actually over $\varphi\in \mathcal{F}_\theta[l_0]$ as in \eqref{Eq4.0:family}, with 
\begin{equation}\label{Eq4.5:l0}
l_0=\begin{cases}
1, &\text{\ if $\L/\F$ is an inert quadratic field extension, \old{and $i_0$ is odd;}}\\
0, &\text{\ otherwise}.
\end{cases}
\end{equation}

Integrating $x,y$ in \eqref{Eq4.2:pretrace} along unipotent subgroups against additive characters, we obtain that that 
\begin{align}\label{Eq4.2:pre-relativeTF}
\sum\limits_{\varphi\in \mathcal{F}_\theta[l_0]}\frac{1}{||\varphi||^2}\iint\limits_{t_1,t_2\in  \Q  \backslash  \A  }\rho \lb f \rb \varphi \lb n(t_1) \rb \overline{\varphi} \lb n(t_2) \rb \psi \lb -m_1t_1+m_2t_2 \rb dt_1dt_2
=\sum\limits_{\gamma\in N \lb \Q \rb \backslash G \lb \Q \rb /N \lb \Q \rb }I(\gamma,f,m_1,m_2) ,
\end{align}
where 
$$I(\gamma,f,m_1,m_2)=\int\limits_{(n(t_1),n(t_2))\in H_\gamma\backslash H \lb \A \rb }f\lb\zxz{1}{t_1}{0}{1}^{-1}\gamma \zxz{1}{t_2}{0}{1}\rb\psi \lb -m_1t_1+m_2t_2 \rb d \lb t_1,t_2 \rb,$$
$H= N\times N$, $H_\gamma$ is the stabiliser of $\gamma $ in $H(\Q)$, and $d(t_1,t_2)$ is the Haar measure on $H_\gamma\backslash H \lb \A \rb$. 

The period integrals on the left-hand side of \eqref{Eq4.2:pre-relativeTF} is directly related to the global Whittaker function:  $$\int\limits_{t\in N \lb \Q \rb \backslash N \lb \A \rb }\varphi(n(t))\psi(-mt)dt=W_\varphi\lb\zxz{m}{}{}{1}\rb.$$
Using the discussions in Section \ref{Sec2.4} we can rewrite the spectral side of \eqref{Eq4.2:pre-relativeTF} as
\begin{equation}\label{Eq4.2:Spectralside}
 \lb m_1m_2 \rb ^{\kappa/2-1/2}e^{-2\pi  \lb m_1+m_2 \rb }\sum\limits_{\varphi\in \mathcal{F}_\theta[l_0]}\frac{1}{||\varphi||^2}\lambda_{m_1} \lb \varphi \rb \overline{\lambda}_{m_2} \lb \varphi \rb.
\end{equation}
The main task is, of course, to analyze the geometric side of \eqref{Eq4.2:pre-relativeTF}.
\begin{defn}\label{Defn4.2:faa'}
For convenience, 
denote $f_{a,a'}$ to be the test function which agrees with $f$ at all other places, and at $p$ equals $ \overline{\tilde{\Phi}}_{a,a'}$. 
\end{defn} 

Note that using the same computations as in \cite[Corollary A.6]{HN18}, together with that $$[\L^\times:\F^\times U_\L(1)]=\begin{cases}
p+1, &\text{\, if $e_\L=1$;}\\
2, &\text{\, if $e_\L=2$.}
\end{cases}$$ we have for the supercuspidal representation case
\begin{equation}\label{Eq4.2:VolZB}
\Vol \lb Z\backslash ZB^1 \rb =
\frac{1}{ \lb p^2-1 \rb p^{i_0-1}}.
\end{equation}
\old{ $$\Vol \lb Z\backslash ZB^1 \rb =\begin{cases}
\frac{1}{ \lb p-1 \rb p^{i_0-1}}, &\text{ if }e_\L=1, \cc \lb \theta \rb \text{\ even},\\
\frac{1}{ \lb p^2-1 \rb p^{i_0-1}}, &\text{\ if }e_\L=1, \cc \lb \theta \rb \text{\ odd},\\
\frac{1}{ \lb p^2-1 \rb p^{i_0-1}}, &\text{\ if }e_\L=2.
\end{cases}$$}

On the other hand for the principal series representation case, it is also straightforward to check that \begin{equation}\label{Eq4.2:VolZKi0}
\Vol \lb Z\backslash ZK_0 \lb \varpi^{i_0} \rb  \rb =\frac{1}{ \lb p+1 \rb p^{i_0-1}}.
\end{equation}

Denote by $D_\mathcal{F}$ the constant multiple appearing in $f_p$, which is
\begin{equation}\label{Eq4.2:DF1}
D_\mathcal{F}=\frac{1}{ \lb p-1 \rb p^{\lceil i_0/2\rceil-1}\Vol \lb Z\backslash ZB^1 \rb }=(p+1)p^{\lfloor i_0/2\rfloor}
\asymp_p p^{\cc \lb \pi \rb /4}
\end{equation} 
when $\pi_\theta$ is a supercuspidal representation, and
\begin{equation}\label{Eq4.2:DF2}
D_\mathcal{F}=\frac{1}{ \lb p-1 \rb p^{i_0-1} \Vol \lb Z\backslash ZK_0 \lb \varpi^{i_0} \rb  \rb }=\frac{p+1}{p-1} \asymp 1
\end{equation}
when $\pi_\theta$ is a principal series representation.

Define
\begin{equation*}
I \lb \gamma, a,a',m_1,m_2 \rb =\iint\limits_{H_\gamma \backslash H_\A}f_{a,a'} \lb n \lb t_1 \rb ^{-1}\gamma n \lb t_2 \rb  \rb \psi \lb -m_1t_1+m_2t_2 \rb d \lb t_1,t_2 \rb .
\end{equation*}

Now the geometric side of \eqref{Eq4.2:pre-relativeTF} becomes
\begin{align}\label{Eq:PreKuz2}
D_\mathcal{F}\sum\limits_{a,a'\in   \lb \CO_\F/\varpi^{\lceil i_0/d_\L \rceil} \CO_\F \rb ^\times}\sum\limits_{\gamma \in N\backslash G \lb \Q \rb /N} I \lb \gamma, a,a',m_1,m_2 \rb .
\end{align}
Here $d_\L=2$ when $\L$ is a field, and $d_\L=1$ when $\L$ splits.

Also recall that by the Bruhat decomposition, $N\backslash G \lb \Q \rb /N$ consists of first-cell terms $\zxz{\mu}{}{}{1}$ for $\mu\in\Q^\times$, as well as second-cell terms $\zxz{}{-\mu}{1}{}$, $\mu\in \Q^\times$. We shall discuss the corresponding orbit integrals $I \lb \gamma, a,a',m_1,m_2 \rb $ in the next two subsections.

\subsection{Geometric side: first-cell terms}\label{Sec:KutGeoComp}
The manipulations and the local factors at $v\neq p$ for the first-cell terms and second-cell terms 
are the same as in \cite[Section 3]{KL06} or \cite[Section 7]{knightly_kuznetsovs_2013}. 
When $\gamma=\zxz{\mu}{}{}{1}$, $H_\gamma=\{ \lb n \lb \mu t \rb , n \lb t \rb  \rb \in N \lb \Q \rb ^2\}$.
We get that
\begin{align*}
I \lb \gamma, a,a',m_1,m_2 \rb &=\iint\limits_{\{ \lb \mu t, t \rb \in \Q^2\}\backslash \A^2}f_{a,a'}\lb \zxz{\mu}{\mu t_2-t_1}{0}{1}\rb\psi \lb -m_1t_1+m_2t_2 \rb d(t_1, t_2)\\
&=\int\limits_{x\in \A}\int\limits_{t_2\in \Q\backslash \A}f_{a,a'}\lb\zxz{\mu}{x}{0}{1}\rb\psi \lb m_1x \rb \psi \lb  \lb m_2-\mu m_1 \rb t_2 \rb dxdt_2.
\end{align*}
Here we made a change of variable $x=\mu t_2-t_1$. As $\psi$ is nontrivial, the integral in $t_2$ is nontrivial only when $\mu=\frac{m_2}{m_1}$. In that case, we write $m_1x =t$ and get that 
\begin{equation*}
I \lb \gamma, a,a',m_1,m_2 \rb =\int\limits_{t\in \A} f_{a,a'} \lb \zxz{m_2}{t}{0}{m_1} \rb \psi \lb t \rb  dt,
\end{equation*}
which is factorisable. 
At all finite places, we need $ v \lb m_1 \rb =v \lb m_2 \rb \geq 0$ for the local factor to be nonvanishing. 
At $\infty$, we get $m_1m_2>0$. So $I \lb \gamma, a,a',m_1,m_2 \rb $ is non-vanishing only when $m_1=m_2$. 

For any finite place $v\neq p$, we have
$$\int\limits_{t\in\Q_v}f_v\lb\zxz{m_1}{t}{0}{m_1}\rb\psi_v \lb t \rb  dt=\begin{cases}
||m_1||_v , &\text{\ if } v \lb m_1 \rb \geq 0,\\
0, &\text{\ otherwise}.
\end{cases} $$

For $v=\infty$, $m_1,m_2>0$, we have according to \cite[Proposition 3.4]{KL06}
$$\int\limits_{t\in\Q_v}f_v\lb\zxz{m_2}{t}{0}{m_1}\rb\psi_v \lb t \rb  dt=\frac{ \lb 4\pi \rb ^{\kappa-1}}{ \lb \kappa-2 \rb !}  \lb m_1m_2 \rb ^{\kappa/2} e^{-2\pi  \lb m_1+m_2 \rb }.$$

When $v=p$ and $\L$ is a field, we have by Definition \ref{Defn3.1:Phiaa'} and \ref{Defn4.2:faa'},
\begin{align*}
\int\limits_{t\in\Q_p}f_{a,a',p}\lb\zxz{m_1}{t}{0}{m_1}\rb\psi_p \lb t \rb  dt&=\int\limits_{t\in\Q_p}\overline{\tilde{\Phi}}_{a,a'}\lb\zxz{m_1}{t}{0}{m_1}\rb\psi_p \lb t \rb  dt\\
&=\int\limits_{t\in\Q_p}\overline{\tilde{\Phi}}_{0,0}\lb\zxz{\frac{a'}{a}m_1}{\varpi^{i_0} a' t}{0}{m_1}\rb\psi_p \lb t \rb  dt. \notag
\end{align*}
By Corollary \ref{Cor:conjugatedMC}, $\overline{\tilde{\Phi}}_{0,0}\lb\zxz{\frac{a'}{a}m_1}{\varpi^{i_0} a' t}{0}{m_1}\rb\neq 0$ if and only if $a'\equiv a\mod \varpi_p^{\lceil i_0/2\rceil}$ and $v \lb t \rb -v \lb m_1 \rb \geq -\lceil i_0/2\rceil$, in which case
\begin{equation*}
\int\limits_{t\in\Q_p}f_{a,a',p}\lb\zxz{m_1}{t}{0}{m_1}\rb\psi_p \lb t \rb  dt=\int\limits_{v \lb t \rb -v \lb m_1 \rb \geq -\lceil i_0/2\rceil} \psi_p \lb -\alpha_0 a \frac{t}{m_1} \rb \psi_p \lb t \rb dt,
\end{equation*}
which is nonzero if and only if $v \lb m_1 \rb =0$ and $a\equiv \frac{m_1}{\alpha_0 } \mod \varpi_p^{\lceil i_0/2\rceil}$, in which case the integral is $p^{\lceil i_0/2\rceil}$.

In this case we obtain that when $m_1,m_2\in \Z_{>0}$, $ \lb m_i,p \rb =1$,
\begin{align}\label{Eq4.3.1:firstcellSC}
D_\mathcal{F}\sum\limits_{a,a'\in   \lb \CO_\F/\varpi^{\lceil i_0/2 \rceil} \CO_\F \rb ^\times} I \lb \gamma, a,a',m_1,m_2 \rb =\delta_{m_1=m_2}\frac{ \lb 4\pi \rb ^{\kappa-1}}{ \lb \kappa-2 \rb !} m_1^{\kappa-1} e^{-4\pi m_1}D_\mathcal{F}p^{\lceil i_0/2\rceil}\asymp _p \delta_{m_1=m_2}p^{\cc \lb \pi \rb /2}.
\end{align}

When $v=p$ and $\L$ splits, we have by \eqref{Eq3.2:newtheta} and Definition \ref{Defn3.2:Phiaa'Prinicipal}
\begin{align*}
\int\limits_{t\in\Q_p}f_{a,a',p}\lb\zxz{m_1}{t}{0}{m_1}\rb\psi_p \lb t \rb  dt&=\int\limits_{t\in\Q_p}\overline{\tilde{\Phi}}_{a,a'}\lb\zxz{m_1}{t}{0}{m_1}\rb\psi_p \lb t \rb  dt\\
&=\chi_1^{-1} \lb a \rb \chi_1 \lb a' \rb \int\limits_{t\in\Q_p}\overline{\tilde{\Phi}}_{0,0}\lb\zxz{m_1}{\varpi^{-i_0}m_1 \lb a-a' \rb +t}{0}{m_1}\rb\psi_p \lb t \rb  dt \notag\\
&=\chi_1^{-1} \lb a \rb \chi_1 \lb a' \rb \int\limits_{t\in -\varpi^{-i_0}m_1 \lb a-a' \rb +m_1O_\F}\psi_p \lb t \rb dt\notag\\
&=\delta_{v \lb m_1 \rb \geq 0}\chi_1^{-1} \lb a \rb \chi_1 \lb a' \rb ||m_1||_v\psi_p \lb -\varpi^{-i_0}m_1 \lb a-a' \rb  \rb .
\end{align*}
The sum over $a,a'$ would now be vanishing unless $v \lb m_1 \rb =0$. In that case we obtain that
\begin{align}\label{Eq4.3.1:firstcellPS}
D_\mathcal{F}\sum\limits_{a,a'\in   \lb \CO_\F/\varpi^{i_0} \CO_\F \rb ^\times} I \lb \gamma, a,a',m_1,m_2 \rb &=\delta_{m_1=m_2}\frac{ \lb 4\pi \rb ^{\kappa-1}}{ \lb \kappa-2 \rb !} m_1^{\kappa-1} e^{-4\pi m_1}D_\mathcal{F}|\sum\limits_{a'}\chi_1 \lb a' \rb \psi_p \lb \varpi^{-i_0}m_1a' \rb |^2\\
&=\delta_{m_1=m_2}\frac{ \lb 4\pi \rb ^{\kappa-1}}{ \lb \kappa-2 \rb !} m_1^{\kappa-1} e^{-4\pi m_1}D_\mathcal{F}   p^{i_0}\asymp \delta_{m_1=m_2}p^{\cc \lb \pi \rb /2}. \notag
\end{align}

%
%
%
%

\subsection{Geometric side: second-cell term}\label{Sec4.4:2ndcell}
This is probably the most technical part of the paper, requiring more careful computations for the test function $f_p$.

For $v \lb \mu \rb \leq 0$ even, denote the classical Kloosterman sum \begin{equation}\label{Eq4.4:ClassicalKloosterman}
\KLv{a,b, \mu}=\sum\limits_{t_1, t_2\in  \lb \varpi_v^{v \lb \mu \rb /2}\O_\F/\O_\F \rb , t_1t_2\equiv \mu \mod\O_\F}\psi_v \lb at_1+bt_2 \rb 
\end{equation}
where the additive character $\psi_v$ is assumed to be unramified.

First of all, as in the standard situation, we have for $\gamma=\zxz{0}{-\mu}{1}{0}$, $H_\gamma=1$ and
\begin{equation}
I \lb \gamma, a,a',m_1,m_1 \rb =\int\limits_{\A^2}f_{a,a'} \lb \zxz{1}{t_1}{0}{1}^{-1}\zxz{0}{-\mu}{1}{0}\zxz{1}{t_2}{0}{1} \rb \psi \lb -m_1t_1+m_2t_2 \rb dt_1dt_2,
\end{equation}
which is factorisable. The computation at the archimedean place and unramified places are the same as in \cite{KL06}. 
At the unramified places, the local factor is nonvanishing only if $v \lb \mu \rb \leq 0$ is even. Then 
\begin{align}\label{Eq4.4:Ifin}
I_v \lb \gamma, a,a',m_1,m_2 \rb &=\int\limits_{\Q_v^2}f_v \lb \zxz{-t_1}{-\mu-t_1t_2}{1}{t_2} \rb \psi \lb -m_1t_1+m_2t_2 \rb dt_1dt_2\\
&=\KLv{m_1, m_2,\mu}.\notag
\end{align}
At infinity, the local factor is nonvanishing if and only if $m_i, \mu>0$, in which case
\begin{align}\label{Eq4.4:Iinf}
I_\infty \lb \gamma, a,a',m_1,m_2 \rb =\frac{e^{-2\pi \lb m_1+m_2 \rb } \lb 4\pi i \rb ^\kappa\sqrt{m_1m_2}^{\kappa-1}}{2 \lb \kappa-2 \rb !}\mu^{1/2}J_{\kappa-1} \lb 4\pi \sqrt{\mu m_1m_2} \rb .
\end{align}

At the  place $p$, the computations are more complicated. The basic strategy is to compute first $$I_p \lb \zxz{0}{-\mu}{1}{0},a,a',m_1,m_2 \rb $$ for a single pair of $ \lb a,a' \rb $, and then relate to others by a simple change of variable.
\subsubsection{Supercuspidal representation case}

\begin{lem}\label{Lem4.4:micong}
	Suppose $I_p \lb \zxz{0}{-\mu}{1}{0},1,1,m_1,m_2 \rb \neq 0$. Then we must have $ m_1\equiv m_2 \equiv \alpha_0 \mod\varpi^{\lceil i_0/2 \rceil}$, where $\alpha_0$ is as in \eqref{Eq:specialAlphaTheta}.
	
\end{lem}
\begin{proof}
	By making a change of variable $t_2\rightarrow t_2+\Delta t_2$ for $\Delta t_2\in \varpi^{-\lceil i_0/2 \rceil} \O_\F$, and noting that $\zxz{1}{\Delta t_2}{0}{1}\in \Supp \tilde{\Phi}_{1,1}$, we get by Corollary \ref{Cor:conjugatedMC}  that the integral is non-vanishing only if $$\psi \lb -\alpha_0\Delta t_2 \rb \psi \lb m_2\Delta t_2 \rb =1,\, i.e. \, m_2\equiv \alpha_0 \mod\varpi^{\lceil i_0/2 \rceil}.$$ Similarly by a change of variable for $t_1$, we get that $m_1\equiv \alpha_0 \mod \varpi^{\lceil i_0/2 \rceil}$. 
\end{proof}

To compute $I_p \lb \zxz{0}{-\mu}{1}{0},1,1,m_1,m_2 \rb $ explicitly when $m_1\equiv m_2\equiv \alpha_0\mod \varpi^{\lceil i_0/2 \rceil}$, we care about when $\zxz{-t_1}{-\mu-t_1t_2}{1}{t_2}\in \Supp \tilde{\Phi}_{1,1}$. By considering the determinant, we see that $v \lb \mu \rb =-2k$ must be even  (including the $e_\L=2$ case, by the choice of $f_{1,1}$). In that case we have the following lemma:
\begin{lem}\label{Lem4.4:tidomains}
	$\zxz{-t_1}{-\mu-t_1t_2}{1}{t_2}\in  \Supp \tilde{\Phi}_{1,1}$ if and only if all the followings hold
	\begin{enumerate}
		\item[(i)] $ \lb t_2+\frac{1}{\sqrt{D}\varpi^{i_0}} \rb \in ZU_\L \lb 1 \rb $. \old{ $ZO_\L^\times$ in the case $\cc \lb \pi \rb =4n$  or $2n+1$, and $\in ZU_\L \lb 1 \rb $ in the case $\cc \lb \pi \rb =4n+2$, }
		\item[(ii)] $t_2^2-\frac{1}{D\varpi^{2i_0}}\equiv \mu \mod\varpi^{v \lb \mu \rb +\lceil i_0/2\rceil}$.
		\item[(iii)] $t_1\equiv -\frac{\mu}{t_2^2-\frac{1}{D\varpi^{2i_0}}}t_2\mod \varpi^{-\lceil i_0/2\rceil}$.
	\end{enumerate}
	In that case, we have
	\begin{equation}\label{Eq4.4:DecompBE}
	\zxz{-t_1}{-\mu-t_1t_2}{1}{t_2}=\zxz{\frac{\mu}{t_2^2-\frac{1}{D\varpi^{2i_0}}}}{-t_1-\frac{\mu}{t_2^2-\frac{1}{D\varpi^{2i_0}}} t_2}{0}{1} 
	\zxz{t_2}{\frac{1}{D\varpi^{2i_0}}}{1}{t_2},
	\end{equation}
	and\begin{equation*}
	f_{1,1,p} \lb \zxz{-t_1}{-\mu-t_1t_2}{1}{t_2} \rb =\theta^{-1}\lb t_2+\frac{1}{\sqrt{D}\varpi^{i_0}}\rb\psi\lb\alpha_0 \lb t_1+\frac{\mu}{ t_2^2-\frac{1}{D\varpi^{2i_0}}} t_2\rb\rb.
	\end{equation*}
\end{lem}

\begin{proof}
Note that $	\zxz{t_2}{\frac{1}{D\varpi^{2i_0}}}{1}{t_2} $ is an element of $\L$ with the embedding in \eqref{Eq:conjugatedembedding}.
	The matrix decomposition \eqref{Eq4.4:DecompBE} is direct to check, while the remaining statements follow directly from the definition  $f_{1,1,p}=\overline{\tilde{\Phi}}_{1,1}$  and Corollary \ref{Cor:conjugatedMC}.
\end{proof}

We make an explicit description of the admissible values for $v \lb \mu \rb $ and  $v \lb t_2 \rb $.

\begin{cor}\label{Cor4.4:vt2SC}
 When the set of $t_1,t_2$ satisfying (i)-(iii) is non-empty, we must have $v \lb \mu \rb =-2k<-\cc \lb \pi_\theta \rb $, and $v \lb t_2 \rb =-k<-i_0$.
\end{cor}
\begin{proof}
Consider the case $e_\L=1$ first.
From Lemma \ref{Lem4.4:tidomains}(i), we get that $v \lb t_2 \rb <-i_0$. From (ii), we get $v \lb \mu \rb =2v \lb t_2 \rb <-2i_0=-\cc \lb \pi_\theta \rb $. 
When $e_\L=2$, we also get $v \lb t_2 \rb <-i_0$ from (i), and $v \lb \mu \rb =2v \lb t_2 \rb <-2i_0-1=-\cc \lb \pi_\theta \rb $ from (ii).
\end{proof}
\old{\begin{example}\label{Example:admissibleMu}
	\begin{enumerate}
		\item When $\cc \lb \pi \rb =4n$, we have $v \lb \mu \rb =-2k\leq -2i_0$,  $v \lb t_2 \rb =-k$ in case $k> i_0$, $v \lb t_2 \rb \geq -i_0$ in case $k=i_0$.
		\item When $\cc \lb \pi \rb =2n+1$, we have $v \lb \mu \rb =-2k<-2i_0-1$, and $v \lb t_2 \rb =-k$.
		\item When $\cc \lb \pi \rb =4n+2$, we have $v \lb \mu \rb =-2k<-2i_0$, and  $v \lb t_2 \rb =-k$.
	\end{enumerate}
	
\end{example}}

Under the conditions in Lemma \ref{Lem4.4:micong}, \ref{Lem4.4:tidomains}, we have
\begin{align*}
&\,I_p \lb \zxz{0}{-\mu}{1}{0},1,1,m_1,m_2 \rb \\
&=\int\limits_{t_i \text{\ satisfying (i)-(iii)}}\theta^{-1}\lb t_2+\frac{1}{\sqrt{D}\varpi^{i_0}}\rb\psi\lb\alpha_0 \lb t_1+\frac{\mu}{ t_2^2-\frac{1}{D\varpi^{2i_0}}} t_2\rb\rb\psi \lb -m_1t_1+m_2t_2 \rb dt_1dt_2\notag\\
&=\int\limits_{t_2 \text{\ satisfying (i)-(ii)}}\theta^{-1}\lb t_2+\frac{1}{\sqrt{D}\varpi^{i_0}}\rb\psi\lb\frac{\alpha_0 \mu}{ t_2^2-\frac{1}{D\varpi^{2i_0}}} t_2+m_2t_2\rb \int\limits_{t_1 \text{\ satisfying (iii)}}\psi\lb  \lb  \alpha_0-m_1 \rb t_1\rb dt_1dt_2\notag\\
&=p^{\lceil i_0/2\rceil}\int\limits_{t_2 \text{\ satisfying (i)-(ii)}}\theta^{-1}\lb t_2+\frac{1}{\sqrt{D}\varpi^{i_0}}\rb\psi\lb\frac{\alpha_0 \mu}{ t_2^2-\frac{1}{D\varpi^{2i_0}}} t_2+m_2t_2\rb \psi\lb-\lb \alpha_0-m_1\rb  \frac{\mu t_2}{t_2^2-\frac{1}{D\varpi^{2i_0}}}\rb dt_2
\notag\\
&=p^{\lceil i_0/2\rceil}\int\limits_{t_2 \text{\ satisfying (i)-(ii)}}\theta^{-1} \lb t_2+\frac{1}{\sqrt{D}\varpi^{i_0}} \rb \psi \lb \frac{m_1\mu}{t_2^2-\frac{1}{D\varpi^{2i_0}}} t_2+m_2t_2 \rb dt_2.\notag
\end{align*}
Here in the third equality, we have used Lemma \ref{Lem4.4:micong}, so that the integrand is constant for the integral in $t_1$ with the domain given in (iii).

For a general pair $ \lb a,a' \rb $, we have

\begin{lem}
	$$
	I _p\lb \zxz{0}{-\mu}{1}{0}, a,a',m_1,m_2 \rb =I_p \lb \zxz{0}{-\mu a a'}{1}{0},1,1,a'^{-1}m_1,a^{-1}m_2 \rb .$$
\end{lem}
\begin{proof}
	By definition, 
	\begin{align*}
	I _p\lb \zxz{0}{-\mu}{1}{0}, a,a',m_1,m_2 \rb &=\int\limits_{\F ^2}\overline{\tilde{\Phi}}_{a,a'} \lb \zxz{1}{t_1}{0}{1}^{-1}\zxz{0}{-\mu}{1}{0}\zxz{1}{t_2}{0}{1} \rb \psi \lb -m_1t_1+m_2t_2 \rb dt_1dt_2\\
	&=\int\limits_{\F ^2}\overline{\tilde{\Phi}}_{1,1'} \lb \zxz{a'}{0}{0}{1}\zxz{1}{t_1}{0}{1}^{-1}\zxz{0}{-\mu}{1}{0}\zxz{1}{t_2}{0}{1}\zxz{a^{-1}}{0}{0}{1} \rb \psi \lb -m_1t_1+m_2t_2 \rb dt_1dt_2\notag\\
	&=\int\limits_{\F ^2}\overline{\tilde{\Phi}}_{1,1'} \lb \zxz{1}{a't_1}{0}{1}^{-1}\zxz{0}{-\mu  a'}{a^{-1}}{0}\zxz{1}{at_2}{0}{1} \rb \psi \lb -m_1t_1+m_2t_2 \rb dt_1dt_2\notag\\
	&=\int\limits_{\F ^2}\overline{\tilde{\Phi}}_{1,1'} \lb \zxz{1}{t_1}{0}{1}^{-1}\zxz{0}{-\mu a a'}{1}{0}\zxz{1}{t_2}{0}{1} \rb \psi \lb -a'^{-1}m_1t_1+a^{-1}m_2t_2 \rb dt_1dt_2\notag\\
	&=I_p \lb \zxz{0}{-\mu a a'}{1}{0},1,1,a'^{-1}m_1,a^{-1}m_2 \rb .
	\end{align*}
	
\end{proof}

Note that $a,a'$ are defined $\mod \varpi^{\lceil i_0/2\rceil}$, and the local integral should be independent of the choice of representatives. Combining the previous lemmas, we get that $I_p \lb \zxz{0}{-\mu}{1}{0}, a,a',m_1,m_2 \rb $ is non-vanishing if and only if $a'\alpha_0\equiv m_1\mod\varpi^{\lceil i_0/2\rceil}$,  $a\alpha_0\equiv m_2\mod\varpi^{\lceil i_0/2\rceil}$, in which case we simply choose $a',a$ such that $a'\alpha_0=m_1$, $a\alpha_0=m_2$.

As a result, we have for fixed $m_1,m_2$, 
\begin{align}\label{Eq4.4:GpSC}
& 
\sum\limits_{a,a'}I_p \lb \zxz{0}{-\mu}{1}{0}, a,a',m_1,m_2 \rb =I _p \lb \zxz{0}{-\mu}{1}{0}, \alpha_0^{-1}m_2,\alpha_0^{-1}m_1,m_1,m_2 \rb=I _p\lb \zxz{0}{-\alpha_0^{-2}\mu m_1m_2}{1}{0},1,1,\alpha_0,\alpha_0 \rb \\
=&p^{\lceil i_0/2\rceil}\int\limits_{t_2 \text{\ satisfying (i)}, t_2^2-\frac{1}{D\varpi^{2i_0}} \equiv \frac{m_1m_2\mu}{\alpha_0^{2}} \mod\varpi^{v \lb \mu \rb +\lceil i_0/2\rceil}}\theta^{-1}\lb t_2+\frac{1}{\sqrt{D}\varpi^{i_0}}\rb\psi\lb\frac{m_1m_2\mu}{\alpha_0\lb t_2^2-\frac{1}{D\varpi^{2i_0}}\rb} t_2+\alpha_0t_2\rb dt_2\notag\\
=&p^{\lceil i_0/2\rceil}\int\limits_{t_2+\alpha_\theta\in ZU_\L \lb 1 \rb , \Nm \lb t_2+\alpha_\theta \rb \equiv {m_1m_2\mu} \mod\varpi^{v \lb \mu \rb +\lceil i_0/2\rceil}}
\theta^{-1}\lb t_2+\alpha_\theta\rb\psi\lb\frac{m_1m_2\mu}{ t_2^2-\frac{\alpha_0^{2}}{D\varpi^{2i_0}}} t_2+t_2\rb dt_2.\notag
\end{align}
In the last line we have made a change of variable $\alpha_0t_2\rightarrow t_2$, and used that $\alpha_\theta=\frac{\alpha_0}{\sqrt{D}\varpi^{i_0}} $, $\theta|_{\F^\times}=1$.

\begin{defn}\label{Def4.4:KloostermanSC}
Denote by $G_p \lb m_1,m_2,\theta,\mu \rb$  the following integral, which we call the generalized Kloosterman sum:
\begin{align}\label{Eq4.4.1:Gp1}
G_p \lb m_1,m_2,\theta,\mu \rb =\int\limits_{t_2+\alpha_\theta\in ZU_\L \lb 1 \rb , \Nm \lb t_2+\alpha_\theta \rb \equiv {m_1m_2\mu} \mod\varpi^{v \lb \mu \rb +\lceil i_0/2\rceil}}
\theta^{-1}\lb t_2+\alpha_\theta\rb\psi\lb\frac{m_1m_2\mu}{ t_2^2-\frac{\alpha_0^{2}}{D\varpi^{2i_0}}} t_2+t_2\rb dt_2.
\end{align}
\end{defn}
Note that we can alternatively write
\begin{align}\label{Eq4.4.1:Gp1'}
G_p \lb m_1,m_2,\theta,\mu \rb =\int\limits_{e=t_2+\alpha_\theta\in ZU_\L \lb 1 \rb , \Nm \lb e \rb  \equiv m_1m_2\mu \mod\varpi^{v \lb \mu \rb +\lceil i_0/2\rceil}}\theta^{-1} \lb e \rb \psi\circ\Tr\lb\frac{1}{2}\lb\frac{m_1m_2\mu}{e} +e\rb\rb de. \notag
\end{align}

\begin{lem}\label{Lem4.4:GeneralKLWholerange}
When $k>i_0$, we can adjust the congruence requirement for $t_2$, that is,
	\begin{align*}
	G_p \lb m_1,m_2,\theta,\mu \rb &=\int\limits_{v \lb t_2 \rb =-k}
\theta^{-1}\lb t_2+\alpha_\theta\rb\psi\lb\frac{m_1m_2\mu}{ t_2^2-\frac{\alpha_0^{2}}{D\varpi^{2i_0}}} t_2+t_2\rb dt_2 \\
&=\int\limits_{ \Nm \lb t_2+\alpha_\theta \rb \equiv {m_1m_2\mu} \mod\varpi^{v \lb \mu \rb +i}}
\theta^{-1}\lb t_2+\frac{\alpha_0}{\sqrt{D}\varpi^{i_0}}\rb\psi\lb\frac{m_1m_2\mu}{ t_2^2-\frac{\alpha_0^{2}}{D\varpi^{2i_0}}} t_2+t_2\rb dt_2								
	\end{align*}
for any $0< i\leq \lfloor k/2\rfloor$. In particular the generalized Kloosterman sum satisfies the   square-root cancellation:
$$G_p \lb m_1,m_2,\theta,\mu \rb \ll_p p^{k/2}.$$
\end{lem}
\begin{proof}
When $k>i_0$, $t_2+\alpha_\theta\in ZU_\L \lb 1 \rb $ follows directly from $v \lb t_2 \rb =-k$.
We apply now the $p-$adic analogue of the stationary phase analysis for any alternative integral expression in Lemma \ref{Lem4.4:GeneralKLWholerange}.	Writing $t_2=t_0 \lb 1+dt \rb $,  with $t_0\in \varpi^{-k}O_\F^\times/U_\F\lb \lceil k/2\rceil\rb$,  $v \lb dt \rb \geq \lceil k/2\rceil$, we have by Lemma \ref{Lem:LiealgChar}
	\begin{equation*}
	\theta^{-1}\lb t_2+\frac{\alpha_0}{\sqrt{D}\varpi^{i_0}}\rb=\theta^{-1}\lb t_0+\frac{\alpha_0}{\sqrt{D}\varpi^{i_0}}\rb\psi \lb \frac{\frac{2\alpha_0^2t_0dt}{D\varpi^{2i_0}}}{t_0^2-\frac{\alpha_0^2}{D\varpi^{2i_0}}} \rb,
	\end{equation*}
	
	\begin{equation*}
	\psi\lb\frac{m_1m_2\mu}{ t_2^2-\frac{\alpha_0^2}{D\varpi^{2i_0}}} t_2+t_2\rb=\psi\lb \frac{m_1m_2\mu}{ t_0^2-\frac{\alpha_0^2}{D\varpi^{2i_0}}} t_0+t_0\rb\psi\lb -\frac{m_1m_2\mu \lb t_0^2+\frac{\alpha_0^2}{D\varpi^{2i_0}} \rb }{ \lb t_0^2-\frac{\alpha_0^2}{D\varpi^{2i_0}} \rb ^2}t_0dt+t_0dt \rb.
	\end{equation*}
	
	For the integral in $dt$ to be non-zero, the stationary point $t_0$ has to satisfy
	\begin{equation}\label{Eq:removecong1}
	\frac{\frac{2\alpha_0^2}{D\varpi^{2i_0}}}{t_0^2-\frac{\alpha_0^2}{D\varpi^{2i_0}}}-\frac{m_1m_2\mu \lb t_0^2+\frac{\alpha_0^2}{D\varpi^{2i_0}} \rb }{ \lb t_0^2-\frac{\alpha_0^2}{D\varpi^{2i_0}} \rb ^2}+1\equiv 0\mod\varpi^{\lfloor k/2\rfloor}.
	\end{equation}
	This equation factorizes as
	\begin{equation*}
	\lb 1-\frac{m_1m_2\mu}{ t_0^2-\frac{\alpha_0^2}{D\varpi^{2i_0}}}\rb \frac{t_0^2+\frac{\alpha_0^2}{D\varpi^{2i_0}}}{t_0^2-\frac{\alpha_0^2}{D\varpi^{2i_0}}}\equiv 0\mod\varpi^{\lfloor k/2\rfloor}.
	\end{equation*}
	When $k>i_0$, we have  $\frac{t_0^2+\frac{\alpha_0^2}{D\varpi^{2i_0}}}{t_0^2-\frac{\alpha_0^2}{D\varpi^{2i_0}}}\not\equiv 0\mod \varpi$ and $\lfloor k/2\rfloor\geq \lceil i_0/2\rceil$. Thus the stationary points in particular satisfy the congruence condition imposed in \eqref{Eq4.4.1:Gp1}, 
and the nonzero contribution comes only from those $t_0$s satisfying
$$\Nm \lb t_0+\alpha_\theta \rb \equiv {m_1m_2\mu} \mod\varpi^{v \lb \mu \rb +\lfloor k/2\rfloor}.$$  
This  congruence equation is non-degenerate and has at most two solutions of $t_0\mod{U_\F(\lfloor k/2\rfloor)}$. 
The square-root cancellation then follows directly. 
\end{proof}
\begin{rem}\label{Rem4.4:Issueforsmallerfamily}
The freedom to adjust the congruence condition for $t_2$ is later used in the proof of Lemma \ref{Lem:AverGeneralKL} to obtain cancellations among second-cell terms for different $\theta$s.
\end{rem}
\begin{rem}\label{Rem:Gpbecomeskloosterman}
	As a sanity check, we show that when $k\geq \cc \lb \pi \rb $, the local integral $G_p \lb m_1,m_2,\theta,\mu \rb $ reduces to the usual Kloosterman sum. Indeed in that case, we have $\theta^{-1} \lb t_2+\alpha_\theta \rb =1$ by the level of $\theta$, and $$\psi \lb \frac{m_1m_2\mu}{t_2^2-\frac{\alpha_0^2}{D\varpi^{2i_0}}} t_2+t_2 \rb =\psi \lb t_2+\frac{m_1m_2\mu}{t_2} \lb 1+\frac{\alpha_0^2}{Dt_2^2\varpi^{2i_0}}+\cdots \rb  \rb =\psi \lb t_2+\frac{m_1m_2\mu}{t_2} \rb .$$
\end{rem}

\subsubsection{Principal series representation case}
In this case, it is easier to compute 
$I_p \lb \zxz{0}{-\mu}{1}{0}, 0,0,m_1,m_2 \rb $ first, that is, to use $\tilde{\Phi}_{0,0}$ as the test function.
\begin{lem}\label{Lem4.4:tidomainsPS}
	$\zxz{-t_1}{-\mu-t_1t_2}{1}{t_2}\in ZK_0 \lb \varpi^{i_0} \rb $ if and only if all the followings hold
	\begin{enumerate}
\item $v \lb \mu \rb =-2k$, $v \lb t_1 \rb =v \lb t_2 \rb =-k\leq -i_0$;
\item $t_1t_2\equiv -\mu \mod \varpi^{-k}$.
\end{enumerate}
In that case, we have
\begin{align*}
I_p \lb \zxz{0}{-\mu}{1}{0}, 0,0,m_1,m_2 \rb =\int\limits_{v \lb t_2 \rb =-k}\chi_1^{-1} \lb \mu \rb \chi_1^2 \lb t_2 \rb \psi \lb \frac{m_1\mu}{t_2}+m_2t_2 \rb dt_2.
\end{align*}
\end{lem}
\begin{proof}
Note that this case is very similar to the classical case where $f_p$ is the characteristic function of a congruence subgroup. By considering the determinant, we get that $v \lb \mu \rb =-2k$ for some $k\in \Z$. Thus $\varpi^{k}\zxz{-t_1}{-\mu-t_1t_2}{1}{t_2}\in K_0 \lb \varpi^{i_0} \rb $, giving rise to all the conditions for $\mu$, $t_i$ and $k$. Then by \eqref{Eq3.2:newtheta}, Definition \ref{Defn3.2:Phiaa'Prinicipal} and $t_1\equiv -\frac{\mu}{t_2} \mod O_\F$,
\begin{align*}
I_p \lb \zxz{0}{-\mu}{1}{0}, 0,0,m_1,m_2 \rb &=\int\limits_{v \lb t_2 \rb =-k}\chi_1^{-1} \lb \varpi^k\frac{\mu}{t_2} \rb \chi_1 \lb \varpi^kt_2 \rb \psi \lb \frac{m_1\mu}{t_2}+m_2t_2 \rb dt_2\\
&=\int\limits_{v \lb t_2 \rb =-k}\chi_1^{-1} \lb \mu \rb \chi_1^2 \lb t_2 \rb \psi \lb \frac{m_1\mu}{t_2}+m_2t_2 \rb dt_2. \notag
\end{align*}
\end{proof}

For a general pair $ \lb a,a' \rb $, we have

\begin{lem}
	$$
	I _p\lb \zxz{0}{-\mu}{1}{0}, a,a',m_1,m_2 \rb =\chi_1^{-1} \lb a \rb \psi \lb -m_2a\varpi^{-i_0} \rb \chi_1 \lb a' \rb \psi \lb m_1a'\varpi^{-i_0} \rb I_p \lb \zxz{0}{-\mu}{1}{0}, 0,0,m_1,m_2 \rb .$$
\end{lem}
\begin{proof}
	By Definition \ref{Defn3.2:Phiaa'Prinicipal}, 
	\begin{align*}
	&\,\, I_p \lb \zxz{0}{-\mu}{1}{0}, a,a',m_1,m_2 \rb \\
&=\int\limits_{\F ^2}\overline{\tilde{\Phi}}_{a,a'} \lb \zxz{1}{t_1}{0}{1}^{-1}\zxz{0}{-\mu}{1}{0}\zxz{1}{t_2}{0}{1} \rb \psi \lb -m_1t_1+m_2t_2 \rb dt_1dt_2\notag\\
	&=\chi_1^{-1} \lb a \rb \chi_1 \lb a' \rb \int\limits_{\F ^2}\overline{\tilde{\Phi}}_{0,0} \lb \zxz{1}{-a'\varpi^{-i_0}}{0}{1}\zxz{1}{t_1}{0}{1}^{-1}\zxz{0}{-\mu}{1}{0}\zxz{1}{t_2}{0}{1}\zxz{1}{a\varpi^{-i_0}}{0}{1} \rb \psi \lb -m_1t_1+m_2t_2 \rb dt_1dt_2\notag\\
&=\chi_1^{-1} \lb a \rb \psi \lb -m_2a\varpi^{-i_0} \rb \chi_1 \lb a' \rb \psi \lb m_1a'\varpi^{-i_0} \rb I_p \lb \zxz{0}{-\mu}{1}{0}, 0,0,m_1,m_2 \rb. \notag
	\end{align*}
	
\end{proof}
\begin{cor}\label{Cor4.4:IpPS}
$\sum\limits_{a,a'}I_p \lb \zxz{0}{-\mu}{1}{0}, a,a',m_1,m_2 \rb $ is nonzero only when $v_p \lb m_i \rb =0$, in which case
\begin{align*}
\sum\limits_{a,a'}I_p \lb \zxz{0}{-\mu}{1}{0}, a,a',m_1,m_2 \rb = p^{i_0}\int\limits_{v \lb t_2 \rb =-k}	\chi_1^{-1} \lb m_1m_2\mu  \rb \chi_1^2 \lb t_2 \rb \psi \lb \frac{m_1m_2\mu}{t_2}+t_2 \rb dt_2.
\end{align*}
\end{cor}
\begin{proof}
By the previous discussions, we see indeed that 
$$\sum\limits_{a,a'}\chi_1^{-1} \lb a \rb \psi \lb -m_2a\varpi^{-i_0} \rb \chi_1 \lb a' \rb \psi \lb m_1a'\varpi^{-i_0} \rb \neq 0$$
if and only if $v_p \lb m_i \rb =0$. In that case,
 we get by a change of variable
\begin{align*}
\sum\limits_{a,a'}I_p \lb \zxz{0}{-\mu}{1}{0}, a,a',m_1,m_2 \rb &=\chi_1 \lb m_1^{-1}m_2 \rb \left|\sum\limits_{a'}\chi_1 \lb a' \rb \psi \lb a'\varpi^{-i_0} \rb \right|^2I_p \lb \zxz{0}{-\mu}{1}{0}, 0,0,m_1,m_2 \rb \\
&=\chi_1 \lb m_1^{-1}m_2 \rb {  p^{i_0}} \int\limits_{v \lb t_2 \rb =-k}\chi_1^{-1} \lb \mu \rb \chi_1^2 \lb t_2 \rb \psi \lb \frac{m_1\mu}{t_2}+m_2t_2 \rb dt_2				\notag\\
&= p^{i_0} \int\limits_{v \lb t_2 \rb =-k}	\chi_1^{-1} \lb m_1m_2\mu  \rb \chi_1^2 \lb t_2 \rb \psi \lb \frac{m_1m_2\mu}{t_2}+t_2 \rb dt_2.	\notag
\end{align*}
\end{proof}
\begin{defn}\label{Def4.4:GpPS}
When $\L$ splits, denote by $G_p \lb m_1,m_2,\theta,\mu \rb$ the following generalized Kloosterman sum
$$G_p \lb m_1,m_2,\theta,\mu \rb =	\chi_1^{-1} \lb m_1m_2\mu  \rb 	\int\limits_{v \lb t_2 \rb =-k}\chi_1^2 \lb t_2 \rb \psi \lb \frac{m_1m_2\mu}{t_2}+t_2 \rb dt_2.$$
\end{defn}
We also have an analogue of Lemma \ref{Lem4.4:GeneralKLWholerange}.
\begin{lem}\label{Lem4.4:StationaryPS}
$G_p \lb m_1,m_2,\theta,\mu \rb $ is vanishing unless there exists $t_2$ such that $$v_p \lb t_2 \rb =-k, \ t_2^2+2\alpha_{\chi_1}t_2\equiv m_1m_2\mu\mod \varpi^{-\lceil 3k/2\rceil}.$$ In that case, we have
\begin{align*}
G_p \lb m_1,m_2,\theta,\mu \rb &=
	\chi_1^{-1} \lb m_1m_2\mu  \rb 	\int\limits_{v(t_2)=-k} \chi_1^2 \lb t_2 \rb \psi \lb \frac{m_1m_2\mu}{t_2}+t_2 \rb dt_2		\\
&=	\chi_1^{-1} \lb m_1m_2\mu  \rb 	\int\limits_{t_2^2+2\alpha_{\chi_1}t_2\equiv m_1m_2\mu\mod \varpi^{v(\mu)+i
} }\chi_1^2 \lb t_2 \rb \psi \lb \frac{m_1m_2\mu}{t_2}+t_2 \rb dt_2.			\notag
\end{align*}
Here $0<i<\lfloor k/2\rfloor$.
Furthermore when $k>i_0$, we have
$$|G_p \lb m_1,m_2,\theta,\mu \rb |\ll_p p^{k/2} .$$
\end{lem}
\begin{proof}
Let $t_2=t_0 \lb 1+dt \rb $ for 
 $t_0\in \varpi^{-k}O_\F^\times/U_\F\lb \lceil k/2\rceil\rb$,  $v_p \lb dt \rb \geq \lceil k/2\rceil$. Then 
\begin{equation*}
G_p \lb m_1,m_2,\theta,\mu \rb =	\chi_1^{-1} \lb m_1m_2\mu  \rb \sum\limits_{t_0}\chi_1 \lb t_0^2 \rb \psi \lb \frac{m_1m_2\mu}{t_0}+t_0 \rb \int\limits_{dt\in \varpi^{\lceil k/2\rceil}O_\F}\psi \lb 2\alpha_{\chi_1}dt \rb \psi \lb -\frac{m_1m_2\mu}{t_0}dt+t_0dt \rb .
\end{equation*}

The integral in $dt$ is nonvanishing only if 
\begin{equation}\label{Eq4.4:PrincipalStationary}
2\alpha_{\chi_1}-\frac{m_1m_2\mu}{t_0}+t_0\equiv 0\mod \varpi^{-\lceil k/2\rceil}
\end{equation}
for some $t_0$. It is also straightforward to check that this congruence equation is non-degenerate when $k>i_0$. The claims  now follow easily.

\end{proof}
\begin{rem}\label{Rem4.4.2:isstandard}
Note that when $k=i_0$, it is possible that \eqref{Eq4.4:PrincipalStationary} is degenerate, so there can be more solutions for $t_0$ and the square-root cancellation does not necessarily hold. One way to avoid this problem is to sum over a slightly larger family on the spectral side, so that we get $k>i_0$ automatically, as we shall see in Section \ref{Sec:KuzAverg}.

On the other hand when $k\geq 2i_0=\cc(\pi_\theta)$, we get that the stationary points satisfy $$t_2^2\equiv m_1m_2\mu\mod\varpi^{-k-i_0},\text{\ \  so }\chi_1\lb \frac{t_2^2}{m_1m_2\mu}\rb=\chi_1(1)=1.$$ Then the generalized Kloosterman sum becomes the classical Kloosterman sum.
\end{rem}
\subsection{Petersson trace formula for small families}
\begin{defn}\label{Def4.5:GlobalGKloosterman}
Globally define the generalized Kloosterman sum to be
$$G \lb m_1,m_2,\theta,\mu \rb =G_p \lb m_1,m_2,\theta,\mu \rb \times\prod\limits_{v\neq p\text{\ finite}} \KLv{m_1,m_2,\mu}$$
where $G_p \lb m_1,m_2,\theta,\mu \rb $ is given in Definition \ref{Def4.4:KloostermanSC}/Definition \ref{Def4.4:GpPS} according to whether $\pi_\theta$ is a supercuspidal/principal series representation, and $\KLv{m_1,m_2,\mu}$ is as in \eqref{Eq4.4:ClassicalKloosterman}.
\end{defn}
Recall $i_0$ from \eqref{Eq:i0} and $l_0$ from \eqref{Eq4.5:l0}.

\begin{defn}\label{Def4.5:c0}
\yh{Denote  $$c_0=\begin{cases}
p^{i_0+1}, &\text{if $\pi_\theta$ is supercuspidal,}\\
p^{i_0}, &\text{otherwise.}
\end{cases}  $$}
\end{defn}

Recall $D_\mathcal{F}$ is given in \eqref{Eq4.2:DF1}/\eqref{Eq4.2:DF2}. Denote
\begin{equation*}
C_\mathcal{F}[l_0]=D_\mathcal{F}\times\begin{cases}
p^{\lceil i_0/2\rceil},\text{\ if $\pi_\theta$ is a supercuspidal representation,}\\
 p^{i_0}, \text{\ if $\pi_\theta$ is a principal series representation.}
\end{cases}
\end{equation*}
Then in either case, we have \yh{
\begin{equation}\label{Eq4.5:CF}
C_\mathcal{F}[l_0]
\asymp (1+p^{-1})c_0
\asymp_p \sqrt{C \lb \pi \rb },
\end{equation}
}
 and
\begin{equation}\label{Eq4.5:IpGp}
I_p(\gamma,f, m_1,m_2)=C_\mathcal{F}[l_0] G_p(m_1,m_2,\theta,\mu)
\end{equation}
for the second-cell term $\gamma=\zxz{}{-\mu}{1}{}$ by \eqref{Eq:PreKuz2} \eqref{Eq4.4:GpSC} and Corollary \ref{Cor4.4:IpPS}.

\begin{theo}\label{Theo4.5:smallfamilyPTF}
For fixed even weight $\kappa \geq 4$ and the family of newforms $\mathcal{F}_\theta[l_0]$,
we have the following refined Petersson trace formula:
\begin{align}
\sum\limits_{\varphi\in \mathcal{F}_\theta[l_0]}\frac{1}{||\varphi||^2}\lambda_{m_1} \lb \varphi \rb \overline{\lambda}_{m_2} \lb \varphi \rb 
= 		C_\mathcal{F}[l_0]\frac{ \lb 4\pi \rb ^{\kappa-1}}{ \lb \kappa-2 \rb !}		\lb\delta_{m_1=m_2}+2\pi i^\kappa\sum\limits_{c\equiv 0 \mod{c_0}, c>0}\frac{G \lb m_1,m_2,\theta,c^{-2} \rb }{c}J_{\kappa-1} \lb \frac{4\pi\sqrt{m_1m_2}}{c} \rb \rb.	\notag
\end{align}
\end{theo}

\begin{proof}
Here we collect all the calculations we have done in the last three subsections. We start with the relative trace formula in \eqref{Eq4.2:pre-relativeTF}.
The spectral side is given in \eqref{Eq4.2:Spectralside}, while the geometric side is set up in \eqref{Eq:PreKuz2}. The first order terms on the geometric side are given in \eqref{Eq4.3.1:firstcellSC}/ \eqref{Eq4.3.1:firstcellPS}. 

The second-cell terms are given in \eqref{Eq4.4:GpSC}/Corollary \ref{Cor4.4:IpPS} at $p$, and in \eqref{Eq4.4:Ifin}\eqref{Eq4.4:Iinf} at other places. Note that the local requirements for $\mu$ implies that $\mu=\frac{1}{c^2}$ for $c_0|c$.

We have also canceled $ \lb m_1m_2 \rb ^{k/2-1/2}e^{-2\pi  \lb m_1+m_2 \rb }$ from both sides for the final formula.
\end{proof}

\subsection{Spectral average}\label{Sec:KuzAverg}
For applications, it is helpful to be able to sum over a larger family than $\theta[l_0]$ on the spectral side, in order to reach a balance between the main terms and the complicated analysis of the error terms. The main idea is that with longer sum on the spectral side, the sum of the generalized Kloosterman sum should be shorter. 

Fix an integer $l$ such that $l_0< l< i_0$.
 For any $\theta'\in \theta[l]$,  we apply Theorem \ref{Theo4.5:smallfamilyPTF} and get

\begin{align}\label{Eq:Kuztrace2}
\sum\limits_{\varphi\in \mathcal{F}_{\theta'}[l_0]}\frac{1}{||\varphi||^2}\lambda_{m_1} \lb \varphi \rb \overline{\lambda}_{m_2} \lb \varphi \rb 
= 		C_\mathcal{F}[l_0]\frac{ \lb 4\pi \rb ^{\kappa-1}}{ \lb \kappa-2 \rb !}		\lb\delta_{m_1=m_2}+2\pi i^\kappa \sum\limits_{c\equiv 0 \mod{c_0}, c>0}\frac{G \lb m_1,m_2,\theta',c^{-2} \rb }{c}J_{\kappa-1} \lb \frac{4\pi\sqrt{m_1m_2}}{c} \rb \rb.	
\end{align}
Note that $C_\mathcal{F}[l_0]$ depends only on $\L$ and $\cc \lb \theta \rb $.

We now take
a sum of \eqref{Eq:Kuztrace2} over $\theta'\in \theta[l]/\sim_{l_0}$. The non-trivial observation is that there are further cancellations for the second-cell terms on the geometric side as follows:
\begin{lem}\label{Lem:AverGeneralKL}For  $v \lb \mu \rb =-2k<-2i_0$
,  we have
	\begin{equation}\label{Eq:AvergQuadKL}
	\frac{1}{ [\theta[l]:\theta[l_0]]}\sum\limits_{\theta'\in \theta[l]/\sim_{l_0}} G_p \lb m_1,m_2,\theta',\mu \rb =\begin{cases}
	G_p \lb m_1,m_2,\theta,\mu \rb ,&\text{\ if $k\geq v_p \lb c_0 \rb +l-l_0 $,}\\
	0, &\text{\ otherwise}.
	\end{cases}
	\end{equation}
\end{lem}
\begin{defn}\label{Eq4.6:CFl}
Define 
\begin{equation*}
C_\mathcal{F}[l]=C_\mathcal{F}[l_0][\theta[l]:\theta[l_0]],
\end{equation*} 
\begin{equation*}
c_l=c_0p^{l-l_0}.
\end{equation*}
\end{defn}
It is clear from Lemma \ref{Lem:Indexoffamily} that \begin{equation}\label{Eq4.6:CFl2}
C_\mathcal{F}[l] \asymp p^{l-l_0}C_\mathcal{F}[l_0].
\end{equation}

From Lemma \ref{Lem:AverGeneralKL},  we immediately obtain the following result:
\begin{theo}\label{Theo4:spectralaverg}
For fixed even weight $\kappa \geq 4$ and the family of newforms $\mathcal{F}_\theta[l]$, we have the following:
\begin{align}\label{Eq:Kuztraceaverg}
\sum\limits_{\varphi\in \mathcal{F}_{\theta}[l]}\frac{1}{||\varphi||^2}\lambda_{m_1} \lb \varphi \rb \overline{\lambda}_{m_2} \lb \varphi \rb 
= 		C_\mathcal{F}[l]\frac{ \lb 4\pi \rb ^{\kappa-1}}{ \lb \kappa-2 \rb !}	\lb\delta_{m_1=m_2}+2\pi i^\kappa\sum\limits_{c\equiv 0 \mod{c_l}, c>0}\frac{G \lb m_1,m_2,\theta,c^{-2} \rb }{c}J_{\kappa-1} \lb \frac{4\pi\sqrt{m_1m_2}}{c} \rb \rb	.
\end{align}

\end{theo}


\subsubsection{Proof of Lemma \ref{Lem:AverGeneralKL}: supercuspidal representation case}\label{Sec4.6.1}

Consider first the case where $\pi_\theta$ is a supercuspidal representation. Note that $v \lb \Nm \lb \alpha_{\theta'} \rb  \rb =-\cc \lb \pi_\theta \rb $, and $v_p \lb c_0 \rb =i_0+1$ in this case. 
Suppose  $k\geq v_p \lb c_0 \rb +l-l_0$ first.
For any $\theta'\in \theta[l]$, we have $\alpha_{\theta'}\in \alpha_{\theta}U_\F \lb i_0-l \rb $ by Lemma \ref{Lem:ThetafamilyPara}. Then  we claim that
\begin{align}
G_p \lb m_1, m_2, \theta',\frac{1}{c^2} \rb &=\int\limits_{ v \lb t_2 \rb =-k}
{\theta'}^{-1}\lb t_2+\alpha_{\theta'}\rb\psi\lb\frac{m_1m_2\mu}{ \Nm \lb t_2+\alpha_{\theta'} \rb } t_2+t_2\rb dt_2 	\notag	\\
&= 	\int\limits_{ v \lb t_2 \rb =-k}
{\theta}^{-1}\lb t_2+\alpha_{\theta}\rb\psi\lb\frac{m_1m_2\mu}{ \Nm \lb t_2+\alpha_{\theta} \rb } t_2+t_2\rb dt_2		\label{Eq4.6:nearbyfamilyGp}.
\end{align}
Here the first equality is Lemma \ref{Lem4.4:GeneralKLWholerange}. 
By the condition $\alpha_{\theta'}\in \alpha_{\theta}U_\F \lb i_0-l \rb $, we have $$t_2+\alpha_{\theta'}\in  \lb t_2+\alpha_{\theta} \rb U_\L \lb e_\L\lb k-\cc(\pi_\theta)/2+i_0-l\rb \rb \subset  \lb t_2+\alpha_{\theta} \rb U_\L \lb e_\L i_0\rb .$$
Here we have used that $\cc \lb \pi_\theta \rb =2i_0+e_\L-1$
. Thus ${\theta'}^{-1}\lb t_2+\alpha_{\theta'}\rb={\theta'}^{-1}\lb t_2+\alpha_{\theta}\rb$ as $\cc \lb \theta' \rb =i_0e_\L$; Similarly 
we have
$$\Nm(t_2+\alpha_{\theta'})=t_2^2+\Nm(\alpha_{\theta'})\in \lb t_2^2+\Nm(\alpha_{\theta})\rb U_\F(2k-\cc(\pi_\theta)+i_0-l)\subset \lb t_2^2+\Nm(\alpha_{\theta})\rb U_\F(k).$$
Thus by the Taylor expansion, $v(\mu)=-2k<-2i_0$ and $v(t_2)=-k$, we have $$\frac{m_1m_2\mu}{ \Nm \lb t_2+\alpha_{\theta'} \rb } t_2\in \frac{m_1m_2\mu}{ \Nm \lb t_2+\alpha_{\theta} \rb } t_2+\CO_\F,$$  so
$$\psi\lb\frac{m_1m_2\mu}{ \Nm \lb t_2+\alpha_{\theta'} \rb } t_2\rb=\psi\lb\frac{m_1m_2\mu}{ \Nm \lb t_2+\alpha_{\theta} \rb } t_2\rb.$$
Lastly $${\theta'}^{-1} \lb t_2+\alpha_{\theta} \rb =\theta^{-1} \lb t_2+\alpha_\theta \rb ,$$
as $\cc \lb \theta^{-1}\theta' \rb \leq e_\L l$ 
while$$t_2+\alpha_\theta\in ZU_\L \lb \frac{e_\L}{2} ( 2k-\cc \lb \pi_\theta )  \rb \rb\subset ZU_\L \lb e_\L l \rb .$$
Thus 
$$	\frac{1}{ [\theta[l]:\theta[l_0]]}\sum\limits_{\theta'\in \theta[l]/\sim_{l_0}} G_p \lb m_1,m_2,\theta',\mu \rb =G_p \lb m_1,m_2,\theta,\mu \rb .$$

Consider now the case $v_p \lb c_0 \rb \leq k<v_p \lb c_0 \rb +l-l_0$. By the same argument as above, it is clear that for any $\theta_1\in \theta[l]$, and $\theta'\in \theta_1[k+l_0-v_p(c_0)]$, we have $G_p\lb m_1,m_2,\theta',\mu \rb =G_p \lb m_1,m_2,\theta_1,\mu \rb$. We shall average over slightly larger family $\theta'\in \theta_1[j]$ for $j=k+l_0-v_p(c_0)+1$, so that we will see the cancellation while only have to deal with the first order terms and first digits for the $p-$adic stationary phase analysis. Note that $j\leq l$ by the condition on $k$. Then we claim that for any $\theta_1\in \theta[l]$,
\begin{equation}\label{Eq4.6:cancelSC0}
\sum\limits_{\theta'\in \theta_1[j]}G_p\lb m_1,m_2,\theta',\mu \rb=\sum\limits_{\theta'\in \theta_1[j]}\int\limits_{t_2^2\equiv m_1m_2\mu \mod\varpi^{v(\mu)+1}}
{\theta'}^{-1}\lb t_2+\alpha_{\theta'}\rb\psi\lb\frac{m_1m_2\mu}{ \Nm \lb t_2+\alpha_{\theta'} \rb } t_2+t_2\rb dt_2 	=0.
\end{equation}

Then a further sum over $\theta_1\in \theta[l]/\sim_{j}$ would also be vanishing. 

For the first equality in \eqref{Eq4.6:cancelSC0}, we apply  
Lemma \ref{Lem4.4:GeneralKLWholerange} for $i=1$. Note that $v_p(t_2^2)<v_p(\Nm(\alpha_{\theta_1}))$ as $k\geq i_0+1$ in the supercuspidal representation case, the congruence requirement $\Nm \lb t_2+\alpha_\theta \rb \equiv {m_1m_2\mu} \mod\varpi^{v \lb \mu \rb +1}$ is the same as $t_2^2\equiv m_1m_2\mu \mod\varpi^{v(\mu)+1}$, which is independent of $\theta'$.

For the second equality of \eqref{Eq4.6:cancelSC0}
, we write $\alpha_{\theta'}=\alpha_{\theta_1}+\alpha_{\theta_1} u$ for $u\in \varpi^{i_0-j}O_\F$. Then by Lemma \ref{Lem:ThetafamilyPara}, the sum over $\theta_1[j]/\sim_{j-1}$ is parameterized by the sum over $u\in \varpi^{i_0-j}O_\F/\varpi^{i_0-j+1}O_\F$. By the same argument as above, we have
\begin{equation*}
t_2+\alpha_{\theta'}\in (t_2+\alpha_{\theta_1})U_\L(e_\L\lb k-\cc(\pi_\theta)/2+i_0-j\rb)=(t_2+\alpha_{\theta_1})U_\L(e_\L i_0-1).
\end{equation*}

Then by Lemma \ref{Lem:LiealgChar}, 
\begin{align}\label{Eq4.6:cancelSC1}
{\theta'}^{-1}(t_2+\alpha_{\theta'})
&={\theta'}^{-1}(t_2+\alpha_{\theta_1}+\alpha_{\theta_1 }u)
={\theta'}^{-1}(t_2+\alpha_{\theta_1})\psi_\L\lb-\alpha_{\theta'}\frac{\alpha_{\theta_1}u}{t_2+\alpha_{\theta_1}}\rb\\
&={\theta'}^{-1}(t_2+\alpha_{\theta_1})\psi\lb-\frac{2\alpha_{\theta_1}^2t_2u}{\Nm(t_2+\alpha_{\theta_1})}\rb\notag\\
&={\theta'}^{-1}(t_2+\alpha_{\theta_1})\psi\lb-\frac{2\alpha_{\theta_1}^2u}{t_2}\rb. \notag
\end{align}
Here in the last line we have used again that $v_p(t_2^2)<v_p(\Nm(\alpha_{\theta_1}))$, and that $v_p\lb\frac{2\alpha_{\theta_1}^2u}{t_2} \rb\geq -1$ by our choice of $j$.

Furthermore as $t_2+\alpha_\theta\in ZU_\L \lb \frac{e_\L}{2} ( 2k-\cc \lb \pi_\theta \rb)\rb$ with $\frac{e_\L}{2} ( 2k-\cc \lb \pi_\theta \rb )\geq \frac{e_\L j}{2}$, we have
\begin{equation}\label{Eq4.6:cancelSC2}
{\theta'}^{-1}(t_2+\alpha_{\theta_1})={\theta_1}^{-1}(t_2+\alpha_{\theta_1})(\theta_1{\theta'}^{-1})(t_2+\alpha_{\theta_1})={\theta_1}^{-1}(t_2+\alpha_{\theta_1})\psi_\L\lb-\alpha_{\theta_1} u \frac{\alpha_{\theta_1}}{t_2}\rb
={\theta_1}^{-1}(t_2+\alpha_{\theta_1})\psi\lb- \frac{2\alpha_{\theta_1}^2 u}{t_2}\rb.
\end{equation}
Lastly we have
$$\Nm(t_2+\alpha_{\theta'})=t_2^2+\Nm(\alpha_{\theta'})\in \lb t_2^2+\Nm(\alpha_{\theta_1})\rb U_\F(2k-\cc(\pi_\theta)+i_0-j)\subset \lb t_2^2+\Nm(\alpha_{\theta_1})\rb U_\F(k-1).$$
 Then one can compute that 
\begin{equation}\label{Eq4.6:cancelSC3}
\psi\lb\frac{m_1m_2\mu}{ \Nm \lb t_2+\alpha_{\theta'} \rb } t_2\rb=\psi\lb\frac{m_1m_2\mu}{ \Nm \lb t_2+\alpha_{\theta_1} \rb } t_2\rb \psi\lb\frac{2m_1m_2\mu\alpha_{\theta_1}^2 u}{ t_2^3} \rb. 
\end{equation}
Piecing together \eqref{Eq4.6:cancelSC0}-\eqref{Eq4.6:cancelSC3}, we get that
\begin{align*}
&\sum\limits_{\theta'\in \theta_1[j]}G_p\lb m_1,m_2,\theta',\mu \rb\\
=&\int\limits_{t_2^2\equiv m_1m_2\mu \mod\varpi^{v(\mu)+1}}
{\theta_1}^{-1}(t_2+\alpha_{\theta_1})
\psi\lb\frac{m_1m_2\mu}{ \Nm \lb t_2+\alpha_{\theta_1} \rb } t_2\rb
\sum\limits_{u\in \varpi^{i_0-j}O_\F/\varpi^{i_0-j+1}O_\F}
\psi\lb\frac{2(m_1m_2\mu-2t_2^2)\alpha_{\theta_1}^2 u}{ t_2^3} \rb dt_2\notag\\
=&0.\notag
\end{align*}
In the last equality we have used that $v_p(m_1m_2\mu-2t_2^2)=-2k$ as $t_2^2\equiv m_1m_2\mu\mod\varpi^{v(\mu)+1}$, and $v_p\lb\frac{2(m_1m_2\mu-2t_2^2)\alpha_{\theta_1}^2 }{ t_2^3}\rb=-i_0+j-1$, thus the sum in $u$ first gives $0$.

%
\subsubsection{Proof of Lemma \ref{Lem:AverGeneralKL}: principal series representation case}

Consider now the case where $\pi_\theta$ is a principal series representation. This case is easier than the supercuspidal representation case. In this case,
$\theta'= \lb \chi',\chi'{}^{-1} \rb \in \theta[j]$ if and only if $\cc \lb \chi_1^{-1}\chi' \rb \leq j$.
Recall that by Lemma \ref{Lem4.4:StationaryPS},

\begin{align}\label{Eq4.6:GpCancelPS}
G_p \lb m_1,m_2,\theta',\mu \rb &=	\int\limits_{v \lb t_2 \rb =-k}	\chi' \lb \frac{t_2^2}{m_1m_2\mu} \rb \psi \lb \frac{m_1m_2\mu}{t_2}+t_2 \rb dt_2\notag\\
&=\int\limits_{t_2^2+2\alpha_{\chi'}t_2\equiv m_1m_2\mu\mod \varpi^{v(\mu)+i}}	\chi' \lb \frac{t_2^2}{m_1m_2\mu} \rb \psi \lb \frac{m_1m_2\mu}{t_2}+t_2 \rb dt_2.
\end{align}
Recall that in this case $v_p \lb c_0 \rb =i_0$ and $l_0=0$. $0<i\leq \lfloor k/2\rfloor$. Note that $v \lb m_1m_2\mu \rb =v \lb t_2^2 \rb $. When $k\geq i_0+l$, 
choose now $i=\min\{\lfloor k/2 \rfloor, k-i_0\}$ for Lemma \ref{Lem4.4:StationaryPS}. Then the points in the integral domain in \eqref{Eq4.6:GpCancelPS} satisfy $$t_2^2+2\alpha_{\chi'}t_2-m_1m_2\mu\equiv t_2^2-m_1m_2\mu\equiv 0\mod \varpi^{v(\mu)+i},$$
as $v_p \lb \alpha_{\chi_1}t_2 \rb =-i_0-k$. Equivalently we have $\frac{t_2^2}{m_1m_2\mu}\equiv 1\mod\varpi^{i }.$

 For such $t_2$, it is clear that $$\chi' \lb \frac{t_2^2}{m_1m_2\mu} \rb =\chi_1 \lb \frac{t_2^2}{m_1m_2\mu} \rb \chi_1^{-1}\chi' \lb \frac{t_2^2}{m_1m_2\mu} \rb =\chi_1 \lb \frac{t_2^2}{m_1m_2\mu} \rb $$
as $\cc \lb \chi_1^{-1}\chi' \rb \leq l\leq \min\lB\lfloor k/2\rfloor, k-i_0\rB $. Here we have used that 
either $\lfloor k/2\rfloor \geq k-i_0\geq l$, or
 $\lfloor k/2\rfloor < k-i_0$, in which case we have $k\geq 2i_0+1$, and thus  $l<i_0\leq \lfloor k/2\rfloor$. 
Thus when $k\geq l+i_0$,
\begin{align*}
\frac{1}{ [\theta[l]:\theta[l_0]]}\sum\limits_{\theta'\in \theta[l]/\sim_{l_0}} G_p \lb m_1,m_2,\theta',\mu \rb&=\int\limits_{t_2^2\equiv m_1m_2\mu\mod \varpi^{v(\mu)+i}}	\chi_1 \lb \frac{t_2^2}{m_1m_2\mu} \rb \psi \lb \frac{m_1m_2\mu}{t_2}+t_2 \rb dt_2 \\& =G_p \lb m_1,m_2,\theta,\mu \rb .\notag
\end{align*}	

On the other hand when $i_0\leq k< i_0+l<2i_0$, we have $\lfloor k/2\rfloor > k-i_0$. 
Choose now $i=k-i_0+1$. The domain of the integral in \eqref{Eq4.6:GpCancelPS} becomes
$$t_2^2-m_1m_2\mu\equiv 2\alpha_{\chi'}t_2\equiv 2\alpha_{\chi_1}t_2\not\equiv 0\mod\varpi^{v(\mu)+i}.$$
Here we have used that when $\theta'\in \theta[l]$, $\alpha_{\chi'}\in \alpha_{\chi_1}U_\F(i_0-l)$.
As $i=k-i_0+1\leq l$, we have
$$\frac{t_2^2}{m_1m_2\mu}\not\equiv 1\mod \varpi^{l}.$$
Then by the orthogonality of characters, we have
\begin{align}
&	\frac{1}{ [\theta[l]:\theta[l_0]]}\sum\limits_{\theta'\in \theta[l]/\sim_{l_0}} G_p \lb m_1,m_2,\theta',\mu \rb \\
=&	\frac{1}{ [\theta[l]:\theta[l_0]]} \int\limits_{t_2^2-m_1m_2\mu\equiv  2\alpha_{\chi_1}t_2\mod\varpi^{v(\mu)+i}}	\sum\limits_{\cc \lb \chi_1^{-1}\chi' \rb \leq l}\chi' \lb \frac{t_2^2}{m_1m_2\mu} \rb \psi \lb \frac{m_1m_2\mu}{t_2}+t_2 \rb dt_2=0.\notag
\end{align}

\subsection{The refined Kuznetsov trace formula}
The discussions so far also allow us to derive the refined Kuznetsov trace formula in Theorem \ref{Theo:Kuz1} without additional difficulty.
Note that the only difference for this case and the Petersson trace formula case is the Archimedean computation, which has already been done in, for example, \cite{knightly_kuznetsovs_2013}.

We shall skip the details here, leaving them to interested readers.

%

\section{An alternative description and the compatibility with the Voronoi formula}\label{Sec5}
Again this section is purely local, so we skip subscript $v$ whenever possible.
\subsection{The relation between the test function and the local matrix coefficient}
The construction of the test function $f_p$ is closely related to the restriction of the matrix coefficient of the newform to proper subgroups. We make this relation  explicit here for later discussions. 
\begin{defn}
Let $K'$ be the maximal compact open subgroup whose elements lie in $$\zxz{O_\F}{\varpi^{-i_0}O_\F}{\varpi^{i_0}O_\F}{O_\F}.$$
\end{defn}
\begin{lem}\label{Lem5.1:twoApproach}
For $\pi=\pi_\theta$, suppose that $\cc(\pi)\geq 3$, $\varphi_{\text{new}}\in \pi$ is an $L^2-$normalized newform, and $\Phi_{\varphi_{\text{new}}}$ is the associated matrix coefficient.
Suppose that $v(\mu)=-2k<-2i_0$ and $v(t_1)=v(t_2)=-k$. Then for the test function $f_p$ as specified in Section \ref{Sec4.1:testfun} and some positive constant $a_\pi\asymp_p p^{\cc(\pi)/2}\asymp_p C_\mathcal{F}[l_0]$, we have
$$f_p\lb\zxz{-t_1}{-\mu-t_1t_2}{1}{t_2}\rb=a_\pi\overline{\Phi}_{\varphi_{\text{new}}}|_{ZK'}\lb\zxz{-t_1}{-\mu-t_1t_2}{1}{t_2}\rb.$$
\end{lem}
\begin{proof}
Denote $g=\zxz{-t_1}{-\mu-t_1t_2}{1}{t_2}$.
Consider the supercuspidal representation case first.
By   Corollary \ref{Cor:RelationNewMinimal}, 
$${\Phi}_{\varphi_{\text{new}}}=\frac{1}{ \lb p-1 \rb p^{\lceil i_0/2 \rceil-1}}
	\sum\limits_{a, a'\in  \lb \CO_\F/\varpi^{\lceil i_0/2 \rceil}\CO_\F \rb ^\times}\Phi_{a,a'}.$$
Comparing with Definition \ref{Def3.2:testfSC}, 
we get that $$a_{\pi}=\frac{1}{\Vol(Z\backslash ZB^1)}\asymp_p p^{\cc(\pi)/2}$$
by \eqref{Eq4.2:VolZB}, and it suffices to check 
by  Definition \ref{Defn3.1:Phiaa'} that,
$$\Phi_{0,0}|_{ZB^1}\lb  g_{a,a'} \rb=\Phi_{0,0}|_{ZK}\lb  g_{a,a'} \rb.$$
Here $g_{a,a'}=\zxz{\varpi^{i_0}a'}{0}{0}{1}g\zxz{\varpi^{-i_0}a^{-1}}{0}{0}{1}$, and we have used that
$$\zxz{\varpi^{i_0}a'}{0}{0}{1} K'\zxz{\varpi^{-i_0}a^{-1}}{0}{0}{1}=K.$$
 Note that $ZB^1\subset ZK$. Thus it suffices to show that $g_{a,a'} \in \Supp \Phi_{0,0}\cap ZK$ implies $g\in ZB^1$. Indeed in that case, we have $v_p(\det(g_{a,a'}))=v_p(\mu)=-2k$, so
$$\varpi^kg_{a,a'}=\zxz{-\varpi^ka^{-1}a't_1}{-(\mu+t_1t_2)\varpi^{i_0+k}a'}{\varpi^{k-i_0}a^{-1}}{\varpi^kt_2}\in K.$$
Note that the lower left element satisfies $v_p(\varpi^{k-i_0}a^{-1})\geq 1$. 
Recall that $$\Supp \Phi_{0,0}\subset J=\L^\times K_{\fA_{e_\L}}(\lfloor \cc(\theta)/2\rfloor),$$ and when $\cc(\pi)\geq 3$, the lower left entry of any element in $K_{\fA_{e_\L}}(\lfloor \cc(\theta)/2\rfloor)$ also satisfies $v_p\geq 1$.
Then  $\varpi^kg_{a,a'}\in \Supp \Phi_{0,0}\cap K$ implies that $\varpi^kg_{a,a'}\in ZJ^1=ZU_\L(1)K_{\fA_{e_\L}}(\lfloor \cc(\theta)/2\rfloor)$. The claim now follows from the last part of  Corollary \ref{Cor:MCofGeneralMinimalVec}.

The principal series representation case is mostly parallel. By Lemma \ref{Lem3.3:newformasMLL} and Definition \ref{Defn3.2:Phiaa'Prinicipal}, we have
$${\Phi}_{\varphi_{\text{new}}}=\frac{1}{|C_0|^2}\sum\limits_{a,a'\in  \lb \CO_\F/\varpi^{i_0} \CO_\F \rb ^\times} {{\Phi}}_{a,a'} \lb g \rb.$$
Comparing with  \eqref{Eq3.3:testfun}, we get that
$$a_\pi=|C_0|^2   \frac{1}{ \lb p-1 \rb p^{i_0-1} \Vol \lb Z\backslash ZK_0 \lb \varpi^{i_0} \rb  \rb }\asymp_p p^{\cc(\pi)/2},$$
and the lemma is reduced to check that
$\Supp \Phi_{0,0}\cap ZK'=ZK_0(\varpi^{i_0})$. This follows immediately from Lemma \ref{Lem:PrincipalONB}.
\end{proof}
\begin{rem}
 $a_\pi$  only depends on $\L$ and $\cc(\pi)$, and actually $a_\pi=(1-p^{-1})C_\mathcal{F}[l_0]$ for our choice of $f_p$ using a case by case check. But we do not need this property here. The condition $v(\mu)=-2k<-2i_0$ can be easily achieved by using the Petersson trace formula for slightly larger family according to Theorem \ref{Theo4:spectralaverg}.
\end{rem}
For later applications, we also prove the following lemma:
\begin{lem}\label{Lem5.1:vt1}
Let $\mu$ and $\pi$ be as in Lemma \ref{Lem5.1:twoApproach}, and $v(t_1)=-k,\ v(t_2)>-k$. Then both $f_p$ and $\Phi_{\varphi_{\text{new}}}$ are vanishing.
\end{lem}
\begin{proof}
We first show that $f_p$ is vanishing on the given domain. In the supercuspidal representation case, this follows immediately from 
 Corollary \ref{Cor4.4:vt2SC}.
Consider the principal series representation case now.
By Lemma \ref{Lem4.4:tidomainsPS}, $g=\zxz{-t_1}{-\mu-t_1t_2}{1}{t_2}\in \Supp \tilde{\Phi}_{0,0}=ZK_0(\varpi^{i_0})$ only if $v(t_1)=v(t_2)=-k$. For general $\tilde{\Phi}_{a,a'}$, we do  translations by $\zxz{1}{\pm\varpi^{-i_0}a}{}{1}$ on the left or right, which however does not change   the valuation of the upper left or lower right entries as $k>i_0$.

We discuss $\Phi_{\varphi_{\text{new}}}$ now. Suppose that $v(t_2)=-j>-k$. By the extended Iwasawa decomposition in the sense of \cite[Lemma 2.1]{hu_triple_2017}, we compute that
$$g=\begin{cases}
\varpi^{-j}\zxz{\mu\varpi^{2j}}{\varpi^j(-\mu-t_1t_2)}{}{\varpi^jt_2}\zxz{1}{}{\varpi^j}{1}\zxz{t_2^{-1}\varpi^{-j}}{}{}{1}, &\text{\ if $j\geq 0$},\\
\zxz{\mu}{-t_1}{}{1}\zxz{1}{}{1}{1}\zxz{}{-1}{1}{1+t_2}		,&\text{\ otherwise}.
\end{cases}$$
One can now check case by case that $g$ is not in the support using
\cite[Proposition 2.19]{hu_triple_2017}. For example when $0\leq j<k$, we have $v(a)=2j-2k$ for $a=\mu\varpi^{2j}$, while \cite[Proposition 2.19]{hu_triple_2017} requires $v(a)\geq \min\{0, 2j-\cc(\pi)\}>2j-2k$.
\end{proof}
\begin{rem}
With a little extra work, it is possible to show that $\Phi_{\varphi_{\text{new}}}$ is vanishing on the given $g$ when $k=i_0$. We skip that here.
\end{rem}
\subsection{Alternative approach to the second-cell terms}

\begin{cor}\label{Cor5.2:Alt2ndcell}
Let the test function $f$ be as in Section \ref{Sec4.1:testfun}. Suppose that $v(\mu)=-2k<-2i_0$.  Then the  second-cell terms can be alternatively written as
\begin{align*}
I_p(\gamma,f,m_1,m_2)&=\frac{a_\pi}{1-p^{-1}}\int\limits_{v(t_1)=-k}W_{\varphi_{\text{new}}}\lb\zxz{m_2}{0}{0}{1}\zxz{0}{-\mu}{1}{0}\zxz{1}{t_1}{0}{1}\rb\psi(-m_1t_1)  dt_1\\
&=\frac{a_\pi p^k}{1-p^{-1}}\int\limits_{v(u)=0}W_{\varphi_{\text{new}}}\lb\zxz{m_2}{0}{0}{1}\zxz{0}{-\mu}{1}{0}\zxz{1}{\frac{u}{p^k}}{0}{1}\rb\psi\lb-\frac{m_1 u}{p^k}\rb  du		.\notag
\end{align*}
\end{cor}
\begin{proof}
By Lemma \ref{Lem5.1:twoApproach}, Lemma \ref{Lem5.1:vt1} and \eqref{Eq2.3:unitarypair}, we can rewrite
\begin{align}
I_p(\gamma, f,m_1,m_2)&=\int\limits_{\F^2}f_p \lb\zxz{1}{t_1}{0}{1}^{-1}\zxz{0}{-\mu}{1}{0}\zxz{1}{t_2}{0}{1}\rb\psi(-m_1t_1+m_2t_2)dt_1dt_2\\
&=a_\pi\int\limits_{v(t_1)= -k,v(t_2)\geq -k}\overline{\Phi}_{\varphi_{\text{new}}}\lb\zxz{1}{t_1}{0}{1}^{-1}\zxz{0}{-\mu}{1}{0}\zxz{1}{t_2}{0}{1}\rb\psi(-m_1t_1+m_2t_2)dt_1dt_2\notag\\
&=a_\pi\int\limits_{t_1,t_2}\overline{w_\pi(-\mu)}\overline{<\pi\lb\zxz{1}{t_2}{0}{1}\rb\varphi_{\text{new}},\pi\lb\zxz{0}{-\mu}{1}{0}\zxz{1}{t_1}{0}{1}\rb\varphi_{\text{new}}>}\psi(-m_1t_1+m_2t_2)dt_1dt_2\notag\\
&=a_\pi\int\limits_{t_1,t_2}\int\limits_{x\in\F ^\times}\overline{W_{\varphi_{\text{new}}}}\lb\zxz{x}{0}{0}{1}\zxz{1}{t_2}{0}{1}\rb W_{\varphi_{\text{new}}}\lb\zxz{x}{0}{0}{1}\zxz{0}{-\mu}{1}{0}\zxz{1}{t_1}{0}{1}\rb d^\times x \psi(-m_1t_1+m_2t_2)dt_1dt_2.\notag
\end{align}
Here we have used the assumption that $w_\pi$ is trivial.
Now we swap the order and integrate in $v(t_2)\geq -k$ first. Using that $$W_\varphi\lb\zxz{x}{0}{0}{1}\zxz{1}{t_2}{0}{1}\rb=\psi(xt_2)W_\varphi\lb\zxz{x}{0}{0}{1}\rb,$$
we get that the integral in $t_2$ is nonvanishing if and only if $x\equiv m_2\mod{\varpi^k}$. As $W_{\varphi_{\text{new}}}(a(x))=\Char \lb O_\F^\times \rb(x)$, we get
\begin{equation*}
I_p(\gamma, f,m_1,m_2)=a_\pi p^k \int\limits_{v(t_1)=-k}\int\limits_{x\equiv m_2\mod\varpi^k} W_{\varphi_{\text{new}}}\lb\zxz{x}{0}{0}{1}\zxz{0}{-\mu}{1}{0}\zxz{1}{t_1}{0}{1}\rb\psi(-m_1t_1) d^\times x dt_1.
\end{equation*}
We show now that the integrand is a constant function in $x$ when $x\equiv m_2\mod\varpi^k$.
Note that in the extended Iwasawa decomposition, we can write
$$\zxz{x}{}{}{1}\zxz{0}{-\mu}{1}{0}\zxz{1}{t_1}{0}{1}=t_1\zxz{\frac{x\mu\varpi^{k}}{t_1}}{-\frac{x\mu}{t_1}}{}{1}\zxz{1}{}{\varpi^k}{1}\zxz{t_1^{-1}\varpi^{-k}}{}{}{1}.$$
Thus
\begin{equation*}
W_{\varphi_{\text{new}}}\lb\zxz{x}{0}{0}{1}\zxz{0}{-\mu}{1}{0}\zxz{1}{t_1}{0}{1}\rb=\psi\lb-\frac{x\mu}{t_1}\rb W_{\varphi_{\text{new}}}\lb\zxz{\frac{x\mu\varpi^{k}}{t_1}}{}{}{1}\zxz{1}{}{\varpi^k}{1}\rb,
\end{equation*}
which is of level $\leq k$ in $x$ by \cite[Proposition 2.12]{Hu:17a}. Thus the integrand is constant for $x\equiv m_2\mod\varpi^k$. The corollary is now clear.
\end{proof}
\subsection{Compatibility with the Voronoi formula}
The alternative description Corollary \ref{Cor5.2:Alt2ndcell} for the second-cell terms for the refined Petersson/Kuznetsov trace formula allows us to analyze the character sum after applying the Voronoi formula more easily and to reduce the problem to the existing works.

\begin{defn}\label{Defn5.3:Dualsum}
For some integer $\el $ with $(\el ,p)=1$, define the accompanying character sum/integral of $G_p\lb m_1,m_2,\theta,\mu \rb$ to be
$$\widetilde{G}_p \lb m_1,m_2,\el ,\theta,\mu \rb=\frac{ p^k}{1-p^{-1}}\int\limits_{v(u)=0}W_{\varphi_{\text{new}}}\lb\zxz{m_2}{0}{0}{1}\zxz{0}{-\mu}{1}{0}\zxz{1}{\frac{u}{p^k}}{0}{1}\rb\psi\lb-\frac{\el m_1}{ u p^k}\rb du	.$$
\end{defn}

The reason we make this definition will be clear in Section \ref{Sec6}.

\begin{lem}\label{Lem5.3:dualsum}
$\widetilde{G}_p \lb m_1,m_2,\el,\theta,\mu \rb=0$ unless $v\lb m_2\mu+\frac{\el m_1}{p^{2k}}\rb\geq \min\{-\cc(\pi),-k-1\}$, 
in which case we have
$$\widetilde{G}_p \lb m_1,m_2,\el,\theta,\mu \rb\ll_p \min\{p^{\frac{3k-\cc(\pi)}{2}},p^k\}.$$
\end{lem}
\begin{proof}
Our strategy is to reinterpret the integral as the value of the matrix coefficient for the newform. 
By a change of variable and the invariance of the newform, we get
\begin{align*}
\widetilde{G}_p \lb m_1,m_2,\el,\theta,\mu \rb& =p^k\int\limits_{v(u)=0}W_{\varphi_{\text{new}}}\lb\zxz{m_2}{0}{0}{1}\zxz{0}{-\mu}{1}{0}\zxz{1}{\frac{1}{up^k}}{0}{1}\rb\psi\lb-\frac{\el m_1 u}{  p^k}\rb d^\times u\\
&=p^k\int\limits_{v(u)=0}W_{\varphi_{\text{new}}}\lb\zxz{0}{-m_2\mu}{1}{0}\zxz{\frac{1}{u}}{}{}{1}\zxz{1}{\frac{1}{p^k}}{0}{1}\rb\psi\lb-\frac{\el m_1 u}{ p^k}\rb d^\times u					\notag\\
&=p^k\int\limits_{v(u)=0}w_\pi^{-1}(u)W_{\varphi_{\text{new}}}\lb\zxz{u}{}{}{1}\zxz{0}{-m_2\mu}{1}{0}\zxz{1}{\frac{1}{p^k}}{0}{1}\rb\psi\lb-\frac{\el m_1 u}{  p^k}\rb d^\times u					\notag\\
&= p^k\int\limits_{v(u)=0}W_{\varphi_{\text{new}}}\lb\zxz{u}{}{}{1}\zxz{0}{-m_2\mu}{1}{0}\zxz{1}{\frac{1}{p^k}}{0}{1}\rb
\overline{W_{\varphi_{\text{new}}}\lb\zxz{u}{}{}{1}\zxz{1}{\frac{\el m_1}{p^k}}{}{1}\rb			}
d^\times u					\notag\\
&= p^k\Phi_{\varphi_{\text{new}}}\lb\zxz{1}{-\frac{\el m_1}{p^k}}{}{1}\zxz{0}{-m_2\mu}{1}{0}\zxz{1}{\frac{1}{p^k}}{0}{1}	\rb. \notag
\end{align*}

Note that
$$\zxz{1}{-\frac{\el m_1}{p^k}}{}{1}\zxz{0}{-m_2\mu}{1}{0}\zxz{1}{\frac{1}{p^k}}{0}{1}
=p^{-k}\zxz{-\el m_1}{-\lb m_2\mu+\frac{\el m_1}{p^{2k}}\rb p^k}{p^k}{1}.
$$
By \cite[Theorem 5.4]{HuSa:19}, this matrix is not in the support of the matrix coefficient of the newform unless $v\lb\lb m_2\mu+\frac{\el m_1}{p^{2k}}\rb p^k\rb\geq \min\{k-\cc(\pi),-1\}$, in which case $
|\Phi_{\varphi_{\text{new}}}|\ll_p \min\{p^{\frac{k-\cc(\pi)}{2}},1\}.
$
The lemma follows easily now.
\end{proof}

\section{Application to the first moment of the Rankin--Selberg $L-$function}\label{Sec6}
\subsection{Preparations}
We take a special version of the Voronoi formula from \cite[Lemma 2.6]{HT} or \cite[Theorem A.4]{KMV}, though a more flexible version would be helpful to  extend our main result to more general situations.
\begin{theo}\label{Theo6.1:Voronoi}
Suppose that $a\in \Z$ is coprime to $c$, and $h$ is a smooth compactly-supported function on $(0,\infty)$.
Let $g$ be a holomorphic modular form of weight $\kappa_g$, square-free level $M$ and  nebentypus $\chi$.  	We factorize $M$ as $M=M_1M_2$ with $M_1=(M,c)$. Then there exists a newform $g^*$ of the same level $M$ and weight $\kappa_g$ such that
\begin{align}\label{Eq6.1:Voronoiformula}
\sum\limits_{n}\lambda_g(n)e\lb\frac{an}{c}\rb h(n)
=\frac{2\pi \eta}{c\sqrt{M_2}}
\sum\limits_{n}\lambda_{g^*}(n)e\lb-\frac{\overline{aM_2} n}{c}\rb\int\limits_{0}^{\infty} h(\xi)J_{\kappa_g-1}\lb\frac{4\pi\sqrt{n\xi}}{c\sqrt{M_2}}\rb d\xi.
\end{align}
Here $\overline{x}$ denotes the multiplicative inverse of $x\mod c$, and $\eta$ is a complex number of modulus 1 depending  on $a,c,g$. $J_{\kappa_g-1}$ is the J-Bessel function.
\end{theo}

The following lemma is straightforward to check using the Chinese remainder theorem:

\begin{lem}\label{Lem:compositeinverse0}
	Suppose $ \lb n_1,n_2 \rb =1$, $a_i\overline{a_i}\equiv 1\mod{n_i}$, $i=1,2$, $\overline{n_1}n_1\equiv 1\mod{n_2}$, $\overline{n_2}n_2\equiv 1\mod{n_1}$. Then $$ \lb a_1n_2+a_2n_1 \rb   \lb \overline{a_1}n_2 \overline{n_2}^2+\overline{a_2}n_1 \overline{n_1}^2 \rb \equiv 1\mod {n_1n_2}.$$
\end{lem}
\subsection{The first moment of  the Rankin--Selberg $L-$function and a hybrid subconvexity bound}\label{Sec6.2:1stMoment}
Recall that $\mathcal{F}_\theta[l]$ is the set of holomorphic newforms of weight $\kappa \geq 4$, level $N=p^\cc$ with $\cc\geq 3$, and trivial nebentypus, whose associated local representation $\pi_p$ belongs to the family $\pi_\theta[l]$. 
Let $g$ be a fixed self-dual holomorphic cusp form of weight $\kappa_g$, level $\D$ and nebentypus $\chi$. We assume that $\D$ is square-free and coprime to $N$. $\chi$ will be quadratic as $g$ is self-dual.
The implied constants for the bounds $\ll$ are always allowed to depend on 
$\epsilon$, which we omit from notations.

Let $M_g$ be the first moment of the Rankin--Selberg $L-$functions
\begin{equation}\label{Eq6.2:Mg=LM}
M_g={\sum\limits_{f\in  \mathcal{F}_{\theta}[l]}}^hL \lb f\times g,1/2 \rb .
\end{equation}
Here ${\sum\limits_f}^h$ is the  harmonic average as in \cite{KMV}:
\begin{equation*}
{\sum\limits_f}^h \alpha_f:=\frac{\Gamma \lb \kappa-1 \rb }{ \lb 4\pi \rb ^{\kappa-1}}\sum\limits_{f}\frac{\alpha_f}{||f ||^2 }.
\end{equation*}
 $f$ is normalized so that $\lambda_f \lb 1 \rb =1$. 

The first few steps are standard, and we follow \cite{HT} closely.
\yh{
By the approximate functional equation as in, for example, \cite[(2.19)]{HT}, we get
\begin{equation}\label{Eq:Mg}
M_{g}=\sum\limits_{n\geq 1}\frac{\lambda_g \lb n \rb }{\sqrt{n}}V \lb \frac{n}{N\D} \rb {\sum\limits_{f\in  \mathcal{F}_{\theta}[l]}}^h\lambda_f \lb n \rb+ \sum\limits_{n\geq 1}\frac{\overline{\lambda_g \lb n \rb} }{\sqrt{n}}\tilde{V} \lb \frac{n}{N\D} \rb {\sum\limits_{f\in  \mathcal{F}_{\theta}[l]}}^h\epsilon(f\times g,1/2)\lambda_f \lb n \rb.
\end{equation}
Here $V, \tilde{V}$ are  fixed smooth functions on $(0,\infty)$, rapidly decaying as $x\rightarrow \infty$.
 In the above equality we have  used that $\lambda_f(n)\in \R$ as the central character for $f$ is trivial. We also note  that $\epsilon(f\times g,1/2)$ is actually the same for any $f\in \mathcal{F}_{\theta}[l]$ due to, for example, \cite[(1.1.2),(1.1.5), (1.1.6)]{Tunnell} and similar computations as in \cite[Section 2.3]{HT}.
}

We  assume from now on that $\epsilon(f\times g,1/2)=1$, since if it is $-1$, $L(f\times g,1/2)=0$ and Theorem \ref{Theo:subconv} is automatic. \yh{In the following we focus on the first term in \eqref{Eq:Mg}, as the second term with epsilon value can be analyzed similarly.
}
Multiplying with $\overline{\lambda_f \lb 1 \rb }=1$ and applying the refined Petersson trace formula in Theorem \ref{Theo4:spectralaverg}, we get that
\begin{equation}\label{Eq:d+od}
M_g\sim M_g^d+M_g^{od}, 
\end{equation}
where  $M_{g}^d$ involves diagonal terms coming from $\delta_{m_1=m_2}$ in Theorem \ref{Theo4:spectralaverg}:
\begin{equation}\label{Eq6.2:Md}
M_{g}^d=	C_\mathcal{F}[l]V \lb \frac{1}{N\D} \rb , 
\end{equation}
and $M_g^{od}$ involves the off-diagonal terms:
\begin{equation}\label{Eq6.1:Moffdiag}
M^{od}_{g}=2\pi i^{\kappa}C_\mathcal{F}[l]\sum\limits_{c\equiv 0 \mod{c_l}, c>0}\frac{1}{c}\sum\limits_{n} \frac{\lambda_g \lb n \rb }{\sqrt{n}} V \lb \frac{n}{N\D} \rb G \lb n,1, \theta, \frac{1}{c^2} \rb J_{\kappa-1} \lb \frac{4{\pi}\sqrt{n}}{c} \rb .
\end{equation}
To analyze the off-diagonal term $M_g^{od}$, we break the sum in $n$ into dyadic ranges as usual by multiplying with a bump function $\eta_Z$, where the size of the sum in $n$ is $Z\ll  \lb N\D \rb ^{1+\epsilon}$. Up to a small error, we may also assume that $c\ll  \lb MN \rb ^A$ for some  fixed large $A$. This is because for the complementary range, one can easily control the sum by using Lemma \ref{Lem4.4:GeneralKLWholerange}, \ref{Lem4.4:StationaryPS}, and that when $\kappa\geq 2$,
$$J_{\kappa-1}(x)\ll x \text{\ as }x\rightarrow 0. $$

Furthermore, we write $c=\aa_p  p^k$ 
for $k\geq v_p(c_l)$ and $ \lb \aa_p ,p \rb =1$, and organize the sum in $c$ according to $\aa_p $ and $k$. 
We shall however be mainly interested in the case where $k<\cc(\pi)$, as the complementary case is much easier to deal with by Remark \ref{Rem:Gpbecomeskloosterman}, \ref{Rem4.4.2:isstandard}.
By Definition \ref{Def4.5:GlobalGKloosterman}, \eqref{Eq4.5:IpGp} and Corollary \ref{Cor5.2:Alt2ndcell},

\begin{align*}
G \lb n,1,\theta, \frac{1}{c^2} \rb &=\frac{1}{C_\mathcal{F}[l_0]}\sum\limits_{y\in  \lb \Z/{\aa_p }\Z \rb ^\times} e \lb \frac{\overline{p}^{2k}y}{\aa_p }+\frac{n\overline{y}}{\aa_p } \rb I_p(\gamma, f,n,1) \notag\\
&=\frac{a_\pi p^k}{C_\mathcal{F}[l_0](1-p^{-1})}\sum\limits_{y\in  \lb \Z/{\aa_p }\Z \rb ^\times} e \lb \frac{\overline{p}^{2k}y}{\aa_p }+\frac{n\overline{y}}{\aa_p } \rb\int\limits_{v(u)=0}W_{\varphi_{\text{new}}}\lb\zxz{0}{-\mu}{1}{0}\zxz{1}{\frac{u}{p^k}}{0}{1}\rb e\lb-\frac{n u}{p^k}\rb  du.
\end{align*}
Because of this, we write
\begin{equation}\label{Eq6.2:Modg}
M^{od}_{g}=2\pi i^{\kappa}a_\pi \frac{C_\mathcal{F}[l]}{C_\mathcal{F}[l_0]}\sum\limits_{c_l|c}
\sum\limits_{Z\ll (MN)^{1+\epsilon}}
\frac{1}{c}K_{c,Z},
\end{equation}
where
\begin{align*}
K_{c,Z}=&\frac{p^{k}}{1-p^{-1}}\sum\limits_{n} \frac{\lambda_g \lb n \rb \eta_Z(n)}{\sqrt{n}} V \lb \frac{n}{N\D} \rb J_{\kappa-1} \lb \frac{4{\pi}\sqrt{n}}{c} \rb	\\
&\times\sum\limits_{y\in  \lb \Z/{\aa_p }\Z \rb ^\times} e \lb \frac{\overline{p}^{2k}y}{\aa_p }+\frac{n\overline{y}}{\aa_p } \rb\int\limits_{v(u)=0}W_{\varphi_{\text{new}}}\lb\zxz{0}{-\mu}{1}{0}\zxz{1}{\frac{u}{p^k}}{0}{1}\rb e\lb-\frac{n u}{p^k}\rb  du \notag \\
=&\frac{p^{k}}{1-p^{-1}}\sum\limits_{y\in  \lb \Z/{\aa_p }\Z \rb ^\times} e \lb \frac{\overline{p}^{2k}y}{\aa_p }\rb   \int\limits_{v(u)=0}W_{\varphi_{\text{new}}}\lb\zxz{0}{-\mu}{1}{0}\zxz{1}{\frac{u}{p^k}}{0}{1}\rb\notag\\
&\times\left[\sum\limits_{n} \frac{\lambda_g \lb n \rb \eta_Z(n)}{\sqrt{n}} V \lb \frac{n}{N\D} \rb J_{\kappa-1} \lb \frac{4{\pi}\sqrt{n}}{c} \rb e\lb\frac{n\overline{y}}{\aa_p } \rb e\lb-\frac{n u}{p^k}\rb  \right]du\notag
\end{align*}
Here in the second equality we have swapped the order  of the sum in $n$ and the sum/integral in $y$/$u$, as the integral in $u$ is essentially a finite sum. 

\begin{lem}\label{Lem6.2:KcZ}
For $L=\frac{\sqrt{Z}}{c}$, we have

$$K_{c,Z}\ll (cMZ)^\epsilon\begin{cases}
\lb 1+\frac{M_2}{\aa_pp^{2k-\cc(\pi)}}\rb  p^{\frac{k-\cc(\pi)}{2}}\sqrt{\frac{Z}{M_2}}\frac{1}{L}	&\text{when $L\gg 1$};\\
\lb 1+\frac{M_2}{\aa_pp^{2k-\cc(\pi)}}\frac{1}{L^2}\rb  p^{\frac{k-\cc(\pi)}{2}}\sqrt{\frac{Z}{M_2}}L	&\text{when $L\ll 1$}.
\end{cases}$$
\end{lem}
\begin{proof}
Denote by $K_{c,Z}(y,u) $ the following expression: $$K_{c,Z}(y,u)=\sum\limits_{n} \frac{\lambda_g \lb n \rb \eta_Z(n)}{\sqrt{n}} V \lb \frac{n}{N\D} \rb J_{\kappa-1} \lb \frac{4{\pi}\sqrt{n}}{c} \rb e\lb\frac{n\overline{y}}{\aa_p } \rb e\lb-\frac{n u}{p^k}\rb.$$
 We shall apply the Voronoi summation formula in Theorem \ref{Theo6.1:Voronoi} for $K_{c,Z}(y,u)$. In particular
Lemma \ref{Lem:compositeinverse0} implies that the part
$$e\lb\frac{n\overline{y}}{\aa_p } \rb e\lb-\frac{n u}{p^k}\rb=e\lb \frac{\overline{y}p^k-u\aa_p }{c}n\rb$$
on the left-hand side of \eqref{Eq6.1:Voronoiformula} with $a=\overline{y}p^k-u\aa_p $ becomes the following on the right-hand side of \eqref{Eq6.1:Voronoiformula}:
$$e\lb-\frac{\overline{aM_2}n}{c}\rb=e\lb-\frac{\overline{M_2}y\overline{p}^{2k}}{\aa_p }n\rb e\lb\frac{\overline{M_2u\aa_p^2 }}{p^k}n\rb.$$
 Thus
\begin{align*}
K_{c,Z}(y,u)= 	\frac{2\pi \eta}{c\sqrt{M_2}}
\sum\limits_{n}\lambda_{g^*}(n)e\lb-\frac{\overline{M_2}y\overline{p}^{2k}}{\aa_p }n\rb e\lb\frac{\overline{M_2u\aa_p }^2}{p^k}n\rb I(n)	,					
\end{align*}
where
\begin{equation*}
I \lb n \rb =\int\limits_{0}^\infty \frac{V \lb \frac{x}{NM} \rb \eta_Z \lb x \rb }{\sqrt{x}}J_{\kappa-1} \lb \frac{4\pi\sqrt{x}}{c} \rb J_{\kappa_g-1} \lb \frac{4\pi\sqrt{nx}}{c\sqrt{M_2}} \rb dx.
\end{equation*}
Then 
\begin{equation}\label{Eq6.2:KcZnew}
K_{c,Z}=	\frac{2\pi \eta}{c\sqrt{M_2}}
\sum\limits_{n}\lambda_{g^*}(n)\widetilde{\KL}(1-\overline{M_2}n, \aa_p ) \widetilde{G}_p\lb n,1,-\overline{M_2\aa_p }^2,  \theta,\frac{1}{c^2} \rb  I(n).	
\end{equation}
Here $$\widetilde{\KL}(1-\overline{M_2}n,\aa_p )=\sum\limits_{y\in  \lb \Z/{\aa_p }\Z \rb ^\times} e \lb \frac{\overline{p}^{2k}y(1-\overline{M_2}n)}{\aa_p }\rb  =\sum\limits_{y\in  \lb \Z/{\aa_p }\Z \rb ^\times} e \lb \frac{y(1-\overline{M_2}n)}{\aa_p }\rb$$ is the Ramanujan sum. If \begin{equation}\label{Eq6.2:Congcp}
(1-\overline{M_2}n, \aa_p )=\aa_{p,n},
\end{equation} then \begin{equation}\label{Eq6.2:RamaSum}
|\widetilde{\KL}(1-\overline{M_2}n,\aa_p )|\ll \aa_{p,n}c^\epsilon.
\end{equation}

 $\widetilde{G}_p\lb n,1,-\overline{M_2\aa_p^2 },  \theta,\frac{1}{c^2} \rb $ is as in Definition \ref{Defn5.3:Dualsum},
which by Lemma \ref{Lem5.3:dualsum} is nonzero only when 
\begin{equation}\label{Eq6.2:Congpk}
v_p\lb\frac{1}{c^2}-\frac{\overline{M_2\aa_p^2} n}{p^{2k}}\rb=v_p(1-\overline{M_2}n)-2k\geq -\cc(\pi),
\end{equation} 
 in which case
\begin{equation}\label{Eq6.2:GenRamasum}
|\widetilde{G}_p|\ll_p p^{\frac{3k-\cc(\pi)}{2}}.
\end{equation}

On the other hand, let $L=\frac{\sqrt{Z}}{c}$, $Q=\frac{\sqrt{nZ}}{c\sqrt{M_2}}$. 
The function $I \lb n \rb $ restricts the sum to essentially  (up to $ \lb cZM \rb ^\epsilon$)  \begin{equation}\label{Eq6.2:nrange}
|L-Q|\ll 1, \text{\ or equivalently }|1-\sqrt{\frac{n}{M_2}}|\ll L^{-1}.
\end{equation}  
In this range we have 
\begin{equation}\label{Eq6.2:Inbound}
I \lb n \rb \ll \sqrt{Z}\frac{L}{(1+L)^{3/2}}\frac{Q}{\lb 1+Q\rb^{3/2}}\ll \begin{cases}
\frac{\sqrt{Z}}{L}, &\text{\ if $L\gg 1$}\\
\sqrt{Z} L, &\text{\ if $L\ll 1$}
\end{cases}
\end{equation} by \cite[Lem2.1]{HT}. 

Now the number of $n$ satisfying the bound in \eqref{Eq6.2:nrange} with the congruence conditions \eqref{Eq6.2:Congcp}, \eqref{Eq6.2:Congpk} can be controlled by
$$\ll \lb 1+\frac{M_2}{\aa_{p,n}p^{2k-\cc(\pi)}}\frac{(1+L)^2}{L^2}\rb (cMZ)^\epsilon.$$
For each of these terms in \eqref{Eq6.2:KcZnew} we apply the bound $\lambda_{g^*}(n)\ll n^\epsilon$ and \eqref{Eq6.2:RamaSum}, \eqref{Eq6.2:GenRamasum} and \eqref{Eq6.2:Inbound}. The lemma is then clear.
\end{proof}

\begin{lem}\label{Lem6.2:Modupper}
For $M^{od}_g$ as in \eqref{Eq6.1:Moffdiag}, we have
$$M^{od}_g\ll_{p,\epsilon} (MN)^\epsilon\lb N^{1/4}p^{l/2}+N^{1/4}M^{1/2}p^{-l/2}\rb.$$
\end{lem}
\begin{proof}

We shall focus on the parts where when $v_p(c_l)\leq k< \cc \lb \pi \rb $, as the parts where $k\geq \cc \lb \pi \rb $ will be easier to control (and one can use the argument in \cite{HT} with slight modifications).  For conciseness we drop all $\epsilon$-terms in our computations.

By \eqref{Eq6.2:Modg}, \eqref{Eq4.6:CFl2},  $a_\pi\asymp_p p^{\cc(\pi)/2}=N^{1/2}$ from Lemma \ref{Lem5.1:twoApproach}, and Lemma \ref{Lem6.2:KcZ}, we get
\begin{align}\label{Eq6.2:NewMgod}
M^{od}_g\ll  N^{1/2}p^l\sum\limits_{Z\ll (MN)^{1+\epsilon}}\sum\limits_{M_2|M}\sum\limits_{v_p(c_l)\leq k< \cc \lb \pi \rb}[&\sum\limits_{C\ll \sqrt{Z}}\sum\limits_{c=\aa_p p^k\asymp C,  (M/M_2)|\aa_p }
\frac{1}{c}\lb 1+\frac{M_2}{\aa_pp^{2k-\cc(\pi)}}\rb  p^{\frac{k-\cc(\pi)}{2}}\sqrt{\frac{Z}{M_2}}\frac{c}{\sqrt{Z}}\\
&+\sum\limits_{C\gg \sqrt{Z}}\sum\limits_{c}    \frac{1}{c}\lb 1+\frac{M_2}{\aa_pp^{2k-\cc(\pi)}}\frac{c^2}{Z}\rb  p^{\frac{k-\cc(\pi)}{2}}\sqrt{\frac{Z}{M_2}}\frac{\sqrt{Z}}{c}  ].\notag
\end{align}
Here the sum over $c\asymp C$ is over integers in dyadic ranges. 
We shall break up the terms in the square bracket into four parts, and control their sums in $Z,M_2,k,c$ first.
In particular we have
\begin{align*}
&\sum\limits_{Z,M_2,k}\sum\limits_{C\ll \sqrt{Z}}\sum\limits_{c}
\frac{1}{c}  p^{\frac{k-\cc(\pi)}{2}}\sqrt{\frac{Z}{M_2}}\frac{c}{\sqrt{Z}}\ll \sum\limits_{Z,M_2,k}\sum\limits_{C\ll \sqrt{Z}}\frac{C M_2}{p^kM}\frac{1}{c}  p^{\frac{k-\cc(\pi)}{2}}\sqrt{\frac{Z}{M_2}}\frac{c}{\sqrt{Z}}\\
\ll& 	\sum\limits_{Z,M_2,k}\sum\limits_{C\ll \sqrt{Z}}\frac{C\sqrt{M_2}}{M}\frac{1}{p^{k/2+\cc(\pi)/2}}\ll\sum\limits_{k}\frac{1}{p^{k/2}}\ll \frac{1}{N^{1/4}p^{l/2}}	.	\notag
\end{align*}
Here the ranges of the summations in $Z,M_2,k,c$ are the same as in \eqref{Eq6.2:NewMgod}.
In the last inequality we have used that $k\geq v_p(c_l)$ which is $\cc(\pi)/2+l+O(1)$, where $O(1)$ means an absolutely bounded constant. Similarly we have
\begin{align*}
\sum\limits_{Z,M_2,k}\sum\limits_{C\ll \sqrt{Z}}\sum\limits_{c}
\frac{1}{c} \frac{M_2}{\aa_pp^{2k-\cc(\pi)}} p^{\frac{k-\cc(\pi)}{2}}\sqrt{\frac{Z}{M_2}}\frac{c}{\sqrt{Z}}\ll \frac{M^{1/2}}{N^{1/4}p^{3l/2}},
\end{align*}
\begin{align*}
\sum\limits_{Z,M_2,k}\sum\limits_{C\gg \sqrt{Z}}\sum\limits_{c} \frac{1}{c} p^{\frac{k-\cc(\pi)}{2}}\sqrt{\frac{Z}{M_2}}\frac{\sqrt{Z}}{c} \ll \frac{1}{N^{1/4}p^{l/2}},
\end{align*}
\begin{align*}
\sum\limits_{Z,M_2,k}\sum\limits_{C\gg \sqrt{Z}}\sum\limits_{c} \frac{1}{c}\frac{M_2}{\aa_pp^{2k-\cc(\pi)}}\frac{c^2}{Z} p^{\frac{k-\cc(\pi)}{2}}\sqrt{\frac{Z}{M_2}}\frac{\sqrt{Z}}{c} \ll \frac{M^{1/2}}{N^{1/4}p^{3l/2}}.
\end{align*}
The lemma follows easily now.
\end{proof}
We can now prove Theorem \ref{Theo:subconv}. From 
Lemma \ref{Lem6.2:Modupper} and \eqref{Eq6.2:Md},
\begin{align}
M_g\sim M^{od}_g+M^{d}_g\ll   \lb MN \rb ^\epsilon[N^{1/2}p^l+N^{1/4}M^{1/2}p^{-l/2}].
\end{align}
Here we have used that $C_\mathcal{F}[l]\asymp N^{1/2}p^l$.
Recall that $ 1\leq l< i_0$. We make different choices according to the relation between $N$ and $M$ as follows:
\begin{enumerate}
	\item When $N\leq \sqrt{M}$, we choose $l=i_0-1$, so $p^l\asymp N^{1/2}$ and  $$M_g\ll  \lb MN \rb ^\epsilon \sqrt{M}.$$
	\item When $\sqrt{M} \leq N\leq M^2$, we choose $1\leq l< i_0$ such that $p^l\asymp  \lb \frac{M}{\sqrt{N}} \rb ^{1/3}$, and 
	$$M_g\ll  \lb MN \rb ^{1/3+\epsilon} .$$
	\item When $N>M^2$, we choose $l=1$ and
	$$M_g\ll  \lb MN \rb ^\epsilon N^{1/2}.$$
\end{enumerate} 
Theorem \ref{Theo:subconv} now follows easily.

\begin{rem}\label{Rem:RamanujanConjNeeded}
If we work with the Maass forms instead of the  holomorphic modular forms, the Ramanujan conjecture does seem important for the bound in Lemma \ref{Lem6.2:KcZ}. It is unlikely that a Ramanujan-conjecture-on-average type of result would suffice. After all, the sum in $n$ in \eqref{Eq6.2:KcZnew} is over a thin arithmetic progression, especially when $N$ is large compared to $M$.
	
	On the other hand, a reasonable bound towards the Ramanujan conjecture can still give a slightly weaker hybrid subconvexity bound.
\end{rem}

\bibliographystyle{plain}

\end{document}